\documentclass{amsart}    
\usepackage{amsmath} 
 \usepackage{amsthm}
\usepackage{paralist}
\usepackage{graphics} 
\usepackage{epsfig} 
\usepackage{graphicx}  \usepackage{epstopdf}  \usepackage{dsfont}
\usepackage[colorlinks=true]{hyperref}
\hypersetup{urlcolor=blue, citecolor=red}

\newtheorem{Definition}{Definition}[section] 
\newtheorem{Proposition}{Proposition}[section]
\newtheorem{Lemma}{Lemma}[section]
\newtheorem{Theorem}{Theorem}[section]

\newtheorem{Remark}{Remark}[section]
\theoremstyle{definition} 
\newtheorem{defiprevia}{Definition}[section] 
 
\theoremstyle{definition} 
\newtheorem{teoprevio}{Theorem}[section] 
 
\theoremstyle{definition} 
\newtheorem{proprevia}{Proposition}[section] 
 
\def \vf{\vec{f}} 
\def \vg{\vec{g}} 
\def \vu{\vec{u}}
\def \vv{\vec{v}}
\def \vw{\vec{w}}
\def \va{\vec{a}}
\def \vpsi{\vec{\psi}}
\def \vvarphi{\vec{\varphi}} 
\def \P{\mathbb{P}} 
\def \U{\vec{U}}

\def \R{\mathbb{R}}
\def \Rt{\mathbb{R}^3}
\def \vpsi{\vec{\psi}}
\def \vvarphi{\vec{\varphi}}
\def \T{\mathbb{T}^3}
\def \Z{\mathbb{{\bf Z}}^3 \setminus \{0\}}

\def \ds{\displaystyle}

\usepackage{amssymb}
  
\title[\bf New Navier-Stokes-alpha model with nonlinear filter] %
{ Mathematical study of a new Navier-Stokes-alpha model with  nonlinear filter equation - Regularity theory and refined time asymptotics} 
\author[ Manuel Fernando Cortez  and Oscar Jarr\'in]{}


\email{oscar.jarrin@udla.edu.ec}
\email{manuel.cortez@epn.edu.ec} 

\thanks{$^*$Corresponding author: Manuel Fernando Cortez} 

\begin{document}
	\maketitle
\begin{center}
\begin{minipage}{6cm}	
	\centerline{\scshape Manuel Fernando Cortez$^*$}
	\medskip
	{\footnotesize
		\centerline{Departamento de Matem\'aticas}
		\centerline{Escuela Politécnica Nacional} 
		\centerline{Ladr\'on de Guevera E11-253, Quito, Ecuador} 
	}
\end{minipage}\hspace{1cm}
	\begin{minipage}{6cm}
		\centerline{\scshape Oscar Jarr\'in}
	\medskip
	{\footnotesize
		\centerline{Escuela de Ciencias Físicas y Matemáticas}
		\centerline{Universidad de Las Américas}
		\centerline{V\'ia a Nay\'on, C.P.170124, Quito, Ecuador}
	} 
\end{minipage} 
\end{center} 
	
	\bigskip

\begin{abstract}
This article is devoted to the mathematical analysis of a new Navier--Stokes-$\alpha$ model involving a nonlinear filter equation. This filter,  strongly depending on a prescribed indicator function $A(\cdot)$ and a parameter $\alpha>0$, was introduced by W.~Layton, L.~G.~Rebholz, and C.~Trenchea [\emph{Modular nonlinear filter stabilization of methods for higher Reynolds numbers flow}, J. Math. Fluid Mech. 14: 325--354 (2012)] to selectively damp eddies according to the local effects of the nonlinearity. It has also been used in numerical studies to analyze numerical diffusion and derive error estimates.

The resulting model is governed by a doubly nonlinear parabolic--elliptic coupled system. In our previous work [M.~F.~Cortez and O.~Jarrín, \emph{Mathematical study of a new Navier--Stokes-alpha model with nonlinear filter equation -- Part I}, J. Math. Fluid Mech. 28, no.~12 (2026)], we established the global well-posedness of weak Leray-type  solutions  and the existence of a global attractor.

\medskip

In the present work, under natural assumptions on $A(\cdot)$, we investigate several additional properties of this model, with particular emphasis on higher-order regularity and its consequences for the long-time dynamics. The regularity analysis constitutes a central and delicate issue, since the nonlinear structure of the elliptic filter and its coupling with the evolution equation preclude a direct application of the standard regularity arguments available for linearly filtered Navier--Stokes-$\alpha$ models. Overcoming this difficulty requires new higher-order estimates for the nonlinear elliptic filter equation, which may also be of independent interest.

Among the results obtained, two constitute the main contributions of the article. First, we establish uniform higher-order regularity for the global attractor. This result is then used as a fundamental ingredient in proving an exact determining-modes property for complete trajectories on the attractor: if the projections of two complete trajectories onto a sufficiently large finite-dimensional space of Stokes modes coincide at every time, then the two trajectories coincide identically. To the best of our knowledge, this is the first time that such an exact determining-modes property for complete trajectories has been explicitly established for a three-dimensional Navier--Stokes-$\alpha$ model involving a nonlinear filter equation. Although closely related asymptotic determining-observation and data-assimilation results are available for classical three-dimensional $\alpha$-models based on the linear Helmholtz filter, this stronger property does not appear to have been explicitly formulated and established even in that setting. \\[3mm]  
\end{abstract} 
{\footnotesize \tableofcontents}

\section{A rapid introduction to the model}
The alpha-models form a family of approximations to the Navier–Stokes equations, whose physical relevance has been established through both mathematical investigations and experimental evidence. Examples include the Bardina \cite{Bardina1,Bardina}, viscous Camassa-Holm \cite{Chen}, Clark-alpha \cite{Cao}, and Leray-alpha models \cite{CheskidovHolm}. The main feature of these models is that, in contrast to the original Navier-Stokes equations, the global well-posedness of finite energy weak solutions can be proved. More precisely, uniqueness and \emph{regularity} of weak finite energy solutions are among the \emph{key required properties} for both theoretical and experimental studies \cite{Bowers1,Bowers2,Cao,CheskidovHolm,Geurts0,Layton}.

\medskip

 Although these alpha-models differ in their mathematical formulation, they all share a common mechanism: for the velocity field $\vu$, the original nonlinear transport term $(\vu \cdot \vec{\nabla})\vu$ is now governed by a filtered velocity as $(\vu_\alpha \cdot \vec{\nabla})\vu$. Here, for a given filtering parameter $\alpha>0$, the filtered velocity $\vu_\alpha$ is defined through the \emph{linear} Helmholtz equation
\begin{equation}\label{Helmholtz}
-\alpha^2\Delta\vu_\alpha+\vu_\alpha=\vu,\qquad \operatorname{div}(\vu_\alpha)=0.
\end{equation}
Roughly speaking, this filtering suppresses flow structures with characteristic lengths below $\alpha$, while retaining the dominant large-scale dynamics, making it an \emph{essential ingredient} in turbulence modeling \cite{Bardina1,Bardina,Lunasin}. 

\medskip

Despite its extensive use, numerical and experimental investigations
\cite{Geurts0, Geurts1, Layton} have shown that the \emph{linear} Helmholtz
filter truncates the unresolved scales almost uniformly throughout the
computational domain. As a consequence, it may over-smooth laminar regions and
remove flow structures that are essential for accurately describing the
dynamics. To address this limitation, \cite[Definition $2.2$]{Layton2}
introduced the following \emph{nonlinear} filter:
\begin{equation}\label{Filter-Equation-Intro}
-\alpha^2 {\rm div} \left( A (\vu) (\vec{\nabla} \otimes \vu_\alpha)^T \right) + \vu_\alpha + \vec{\nabla}P_\alpha= \vu, \quad
{\rm div}(\vu_\alpha) = 0,
\end{equation}
where the main feature of this elliptic equation in divergence form is a \emph{prescribed indicator
	function}
\[A:\mathbb{R}^3 \to \mathbb{R}, \]
which determines the regions in which
filtering is applied. Additionally, $(\vec{\nabla} \otimes \vu_\alpha)^T$ denotes the transpose matrix of
$\vec{\nabla} \otimes \vu_\alpha$, while $P_\alpha$ is a pressure term related to $\vu$ and $\vu_\alpha$  by
\[
- \Delta P_\alpha = - \alpha^2 {\rm div}\Big( {\rm div} \left( A (\vu) (\vec{\nabla} \otimes \vu_\alpha)^T \right)\Big),
\]
which follows from the divergence-free condition satisfied by both $\vu$ and $\vu_\alpha$.

\medskip 

Unlike the linear filter (\ref{Helmholtz}), the nonlinear formulation given in equation (\ref{Filter-Equation-Intro}) 
adapts the amount of regularization to the local flow characteristics.
Consequently, it is capable of preserving coherent structures in nearly
laminar regions while providing stronger filtering in highly turbulent ones.
This behavior is consistent with the nonlinear energy transfer that
characterizes high Reynolds number flows. Further details and the physical
motivation behind this model can be found in \cite{Layton2}.

\medskip

The nonlinear filter equation (\ref{Filter-Equation-Intro})   was incorporated into the nonlinear convection term of the classical Navier-Stokes equations in \cite{Bowers1, Bowers2, Layton2}, and was subsequently employed in \cite{Takhirov-Trenchea} for the \emph{numerical and experimental investigation}. In this framework, if $\vu=\vu(t,x)\in \Rt$ denotes the fluid velocity and $P=P(t,x)\in \R$ the pressure, the resulting Navier-Stokes-alpha model is given by
\begin{equation*}
\begin{cases}\vspace{2mm}
\partial_t \vu - \nu \Delta \vu + (\vu_\alpha \cdot \vec{\nabla})\vu + \vec{\nabla} P = \vf, \qquad \ \nu>0,  \ \  {\rm div}(\vu)=0,\\
-\alpha^2 {\rm div}\left( A (\vu) (\vec{\nabla} \otimes \vu_\alpha)^T \right)+\vu_\alpha +\vec{\nabla}P_\alpha = \vu, \qquad \alpha>0, \ \  {\rm div}(\vu_\alpha)=0.
\end{cases}
\end{equation*}
This model consists of a coupled parabolic-elliptic system with nonlinearities in both equations. Here, $\nu$ is the kinematic viscosity, $\vf=\vf(t,x)\in \Rt$ represents a prescribed external force, and $A(\cdot)\in \R$ is the given indicator function that determines the filtering intensity.

\medskip

To the best of our knowledge, the theoretical analysis of this \emph{new Navier-Stokes-alpha model} was first undertaken in our previous work \cite{CortezJarrin2026}. Specifically, for a suitable given function $\varphi=\varphi(x)\in \R$, the following \emph{slightly modified} initial value problem was studied:
\begin{equation}\label{Alpha-model}
    \begin{cases}\vspace{2mm}
    \partial_t \vu - \nu\Delta \vu + (\vu_\alpha \cdot \vec{\nabla})\vu + \vec{\nabla} P  =  \vf, \quad \text{div}(\vu)=0, \\ \vspace{2mm}
     -\alpha^2 \text{div} \left( A (\varphi\ast\vu) (\vec{\nabla} \otimes \vu_\alpha)^T \right)+\vu_\alpha + \vec{\nabla}P_\alpha = \vu, \quad \text{div}(\vu_\alpha)=0,  \\ 
\vu(0,\cdot)=\vu_0, \quad \text{div}(\vu_0)=0,
    \end{cases}
\end{equation}
where $\vu_0=\vu_0(x)$  always denotes the initial data. 

\medskip

As explained in \cite{CortezJarrin2026}, standard \emph{a priori} energy estimates guarantee that $\vu$ satisfies the same energy bounds as Leray weak solutions of the classical Navier-Stokes equations, namely
$\vu\in L^\infty_t L^2_x \cap L^2_t \dot{H}^{1}_x$. As is well known in the theory of the Navier-Stokes equations, these estimates are not sufficient to establish the uniqueness of solutions, which, as mentioned, is one of the \emph{key required properties} of alpha models \cite{Bardina,Bowers1,Bowers2,Cao,CheskidovHolm,Geurts0,Geurts1,Layton}.

\medskip

Consequently, further regularity of the filtered velocity $\vu_\alpha$ is needed to guarantee uniqueness of solutions. This fact could be derived from the elliptic nonlinear filter equation (\ref{Filter-Equation-Intro}), provided that $A(\vu)$ together with its derivatives $\vec{\nabla}(A(\vu))$ are continuous and bounded. However, by the classical chain rule, these assumptions on $\vec{\nabla}(A(\vu))$ require control of $\vec{\nabla}\otimes \vu$ in spaces of continuous and bounded functions. Such regularity appears to be beyond what can be deduced from the energy estimates alone, since the latter only yield $\vu \in L^\infty_t L^2_x \cap L^2_t \dot{H}^{1}_x$ and, consequently, $\vec{\nabla}\otimes \vu \in L^2_t L^2_x$.
Therefore, the regularized term $A(\varphi\ast\vu)$ in the nonlinear filter equation will allow us to suitably control $\vec{\nabla}(A(\varphi\ast\vu))$. See \cite{CortezJarrin2026} for further technical details.

\medskip

Although the convolution term $\varphi$ is essentially introduced for technical reasons, in order to rigorously justify theoretical computations (see \cite[Remark 1, Remark 2]{CortezJarrin2026} for further details), it is worth mentioning that this type of regularization of the filter equation may also be necessary for practical computations \cite[Page $328$]{Layton2}.

\medskip 

In this context, the program addressed in \cite{CortezJarrin2026} on the qualitative study of the Navier-Stokes alpha model (\ref{Alpha-model}) was mainly focused on two directions:
\begin{itemize}
	\item The existence and uniqueness of  weak Leray-type solutions in the natural energy space. See Definition \ref{Def-Leray-Solutions} and Theorem \ref{Th-GWP} below.
	\item The long-term behavior of solutions using the concept of the global attractor. See Definition \ref{Def-Global-Attractor} and Theorem \ref{Th-Global-Attractor} below.
\end{itemize}

The main objective of this article is to address a more delicate and physically relevant question concerning the doubly nonlinear system (\ref{Alpha-model}), namely, the \emph{regularity issue} for both weak Leray solutions and the global attractor, together with its applications to a \emph{refined description of the long-time behavior}.

\medskip 

To this end, from a technical point of view, our main contribution is the development of higher-order regularity estimates for solutions of the \emph{nonlinear} filter equation
\begin{equation}\label{Filter-Equation}    
-\alpha^2 \text{div} \left( A (\varphi\ast\vu) (\vec{\nabla} \otimes \vu_\alpha)^T \right)+\vu_\alpha + \vec{\nabla}P_\alpha = \vu, \quad \text{div}(\vu_\alpha)=0.
\end{equation}
As explained in Section \ref{Sec:Higher regularity properties of the nonlinear filter equation} below, the nonlinear structure of this equation, together with the fact that the term in divergence form also depends on the velocity $\vu$, means that, to the best of our knowledge, existing regularity results for elliptic equations in divergence form are not directly applicable in this setting; see, for instance, \cite{Brezis}, \cite{Gilbarg}, and \cite{Ya}. Therefore, the regularity theory developed in this article for this  equation  may also be of independent interest.

\medskip 

Thereafter, building on this higher-order regularity theory, we establish the main results of this article, which can be summarized as follows.

\begin{itemize} 
	\item First, in Section \ref{Sec:Regularity of weak Leray solutions}, we show that the smoothing effects on weak Leray solutions of (\ref{Alpha-model}) are strongly determined by the prescribed regularity properties of the indicator function $A(\cdot)$ and the convolution kernel $\varphi$ appearing in equation (\ref{Filter-Equation}).
	
	\item Next, in Section \ref{Sec:Regularity of the global attractor}, using the notion of complete trajectories introduced in Definition \ref{Def-Complete-Trajectory} below, we derive new uniform higher-order regularity estimates for trajectories on the global attractor associated with system (\ref{Alpha-model}).
	
	\item Finally, in Section \ref{Sec:Refined long-time behavior}, we show that some of these new regularity results, together with the theory of determining modes introduced in Definitions \ref{def:asymptotic-determining-modes} and \ref{def:determining-modes-on-attractor}, allow us to obtain refined descriptions of the long-time behavior of both weak Leray solutions and trajectories on the global attractor.
	
\end{itemize}

To conclude, in Section \ref{Sec:Contributions}, we briefly discuss the results obtained mainly in Sections \ref{Sec:Regularity of the global attractor} and \ref{Sec:Refined long-time behavior}. 

\section{Main results}
We consider the system (\ref{Alpha-model}) under periodic boundary conditions on the torus $\T=[0, L]^3$, where $L > 0$ denotes the fixed period. We denote by $\mathcal{P}$ the set of vector-valued trigonometric polynomials satisfying periodic boundary conditions on $\T$. We then define $\mathcal{V} = \left\{ \vvarphi \in \mathcal{P} : {\rm div}(\vvarphi) = 0 \text{ and } \int_{\T} \vvarphi(x)dx = 0 \right\}$. Finally, we introduce the Hilbert space of {solenoidal, zero-mean vector fields $L^2_{\text{s}}(\T)$ as the closure of $\mathcal{V}$ in $L^2(\T)$.

\medskip

As mentioned above, the theoretical study of (\ref{Alpha-model}) is strongly driven by prescribed properties of the given indicator function $A(\cdot)$ and the convolution term $\varphi$, which we will detail in the statement of the results below. Specifically, for $\varphi$, these properties are given by the framework of Sobolev spaces. In contrast, for  
\[ A(\cdot): \Rt \to \R, \]
we will use the framework of H\"older spaces. For $m \in \mathbb{N}$ and  $0 < s \leq 1$, we recall that  the H\"older space $C^{m,s}(\mathbb{R}^3)$ consists of  $\mathcal{C}^m$-functions whose derivatives of order $m$ are $s$-H\"older continuous.  Consequently, for suitable values of $m$ and $s$, we will assume conditions of the type  
\begin{equation}\label{Reg-A-Motivation}
A(\cdot) \in \mathcal{C}^{m,s}(\Rt) \quad \text{and} \quad  \varphi \in H^m(\T).
\end{equation}

Another main property assumed for the indicator function $A(\cdot)$ is the following: for a given parameter $0 < \beta < 1$, we assume from now on that 
\begin{equation}\label{Conditions-a-existence-1}
\beta \leq A(x) \leq 1, \quad \text{for all } x \in \mathbb{R}^3.
\end{equation}
The lower bound $\beta < A(\cdot)$ is inspired by the numerical work \cite{Takhirov-Trenchea} to ensure the positivity of the indicator function. From an experimental standpoint, this positivity condition is beneficial for the numerical stabilization of the alpha-model. From a mathematical perspective, it also yields suitable control of the filtered velocity $\vu_\alpha$, as we will detail in Section \ref{Sec:Higher regularity properties of the nonlinear filter equation} below. On the other hand, the upper bound $A(\cdot) \leq 1$, introduced in \cite[Definition $2.1$]{Layton2}, guarantees that the nonlinear filter equation in (\ref{Filter-Equation}) locally exhibits the same behavior as the linear Helmholtz filter equation (\ref{Helmholtz}) in regions where $A(\cdot) = 1$.

\medskip

It is interesting to observe that these properties of the indicator function $A(\cdot)$ and the convolution term $\varphi$ naturally arise in the theoretical study of the system (\ref{Alpha-model}). Moreover, from an experimental perspective, the framework of H\"older spaces makes it entirely feasible to construct functions $A(\cdot)$ satisfying both (\ref{Reg-A-Motivation}) and (\ref{Conditions-a-existence-1}). See Appendix \ref{Sec:Appendix-Function-A} for an simple example.
\medskip

Once we have established our principal framework, we will present our main results in the following separate subsections for clarity.
 
\subsection{Regularity of weak Leray solutions}\label{Sec:Regularity of weak Leray solutions} We begin by recalling some definitions and previous results. 
\begin{defiprevia}[Weak Leray Solution, Definition $2.1$ of \cite{CortezJarrin2026}]\label{Def-Leray-Solutions} Let $\alpha>0$ be fixed. We shall say that $(\vu, \vu_\alpha, P, P_\alpha)$ is a weak Leray solution to the coupled parabolic-elliptic system (\ref{Alpha-model}) if the following conditions hold:
	\begin{enumerate}
		\item[(A)] We have 
		\[\vu \in L^{\infty}_{\rm loc}\big([0,+\infty[, L^2_{\text{s}}(\T)\big) \cap L^{2}_{\rm loc}\big( [0,+\infty[, \dot{H}^1(\T)\big), \quad \vu_\alpha \in L^{\infty}_{\rm loc}\big([0,+\infty[, H^1(\T)\big),\]
		and 
		\[ P\in L^2_{\rm loc}\big( [0,+\infty[, L^2(\T)\big), \quad P_\alpha \in L^\infty_{\rm loc}\big( [0,+\infty[, L^2(\T)\big). \]
		\item[(B)] The set $(\vu,\vu_\alpha,P, P_\alpha)$  solves the  system (\ref{Alpha-model}) in the weak sense. 
		\item[(C)] For all $t \geq 0$, the velocity $\vu$ verifies the energy control: 
		\begin{equation}\label{Energy-Control}
		\| \vu(t,\cdot)\|^{2}_{L^2} + 2 \nu \int_{0}^{t} \| \vu(s,\cdot)\|^{2}_{\dot{H}^1}ds \leq \| \vu_0 \|^{2}_{L^{2}} + 2 \int_{0}^{t} \big\langle f(s,\cdot), \vu(s,\cdot)\big\rangle_{\dot{H}^{-1}\times\dot{H}^1}\, ds,
		\end{equation}
		provided that $\vu_0 \in L^2_{\text{s}}(\T)$,  and  $\vf \in L^{2}_{\rm loc}\big([0,+\infty[, \dot{H}^{-1}(\T)\big)$ with ${\rm div}(\vf)=0$. 
	\end{enumerate}		
\end{defiprevia}	 
For simplicity, we restrict our analysis to divergence-free external forces. Nevertheless, the definitions and results presented  can be extended to the general case of arbitrary external forces, at the expense of only minor technical adjustments.

\medskip

\begin{teoprevio}[Global well-posedness of Leray solutions, Theorems $2.1$ and $2.2$ of \cite{CortezJarrin2026}]\label{Th-GWP} Let $\vu_0 \in L^2_{\text{s}}(\T)$ be an initial datum, and let 
\[ \vf \in L^{2}_{\rm loc}\big([0,+\infty[, \dot{H}^{-1}(\T)\big),  \]	
be a divergence-free and zero-mean  external force. In the nonlinear filter equation (\ref{Filter-Equation}) assume (\ref{Conditions-a-existence-1}) and that  
		\begin{equation}\label{conditions-a-varphi-GWP}
	A(\cdot) \in \mathcal{C}^{1,1}(\Rt) \quad \mbox{and} \quad \varphi \in H^1(\T).
	\end{equation}
	
Then, for every $\alpha>0$ the coupled system (\ref{Alpha-model}) has a \emph{unique}  weak Leray solution  $(\vu,\vu_\alpha,P, P_\alpha)$ in the sense of Definition \ref{Def-Leray-Solutions}.
\end{teoprevio}

In this framework, under prescribed regularity assumptions on the external force $\vf$ and the functions $A(\cdot)$ and $\varphi$ in the nonlinear filter equation (\ref{Filter-Equation}), we study the gain in regularity of weak Leray solutions to the full Navier-Stokes alpha model (\ref{Alpha-model}). In addition, recall that the pressure terms $P$ and $P_\alpha$ are related to the divergence-free vector fields $\vu$ and $\vu_\alpha$, respectively, through the expressions
\begin{equation*}
P= -(-\Delta)^{-1} {\rm div} \Big( {\rm div} (\vu \otimes \vu_\alpha) \Big), \qquad P_\alpha =-(\Delta)^{-1} {\rm div} \Big( {\rm div} \Big(  A(\varphi \ast \vu) (\vec{\nabla} \otimes \vu_\alpha)^T\Big) \Big),
\end{equation*}
and therefore, for simplicity, we focus only on the unknowns $\vu$ and $\vu_\alpha$.

\begin{Theorem}\label{Th-Regularity} Let $m\in \mathbb{N}$. Within the setting  of Theorem \ref{Th-GWP}, assume that the external force verifies 
	\begin{equation}\label{Reg-force}
	\vf \in L^{2}_{loc}([0,+\infty[, H^m(\T)).   
	\end{equation}
	Moreover, in the filter equation (\ref{Filter-Equation}) assume that
	\begin{equation}\label{conditions-a-varphi-regularity}
	A(\cdot)\in \mathcal{C}^{m+1,1}(\Rt)  \quad \mbox{and} \quad \varphi \in H^{m+1}(\T). 
	\end{equation}
	Then, for any $0<T_1<T_2<+\infty$   the unique  Leray's solution to the coupled system  (\ref{Alpha-model})  verifies:
	\begin{equation}\label{Regularity-Leray-Solution} 
	\vu \in \mathcal{C}\big([T_1,T_2], H^{m+1}(\T)\big) \cap  L^2\big([T_1,T_2], \dot{H}^{m+2}(\T)\big)  \quad \text{and} \quad   \vu_\alpha \in  \mathcal{C}\big([T_1,T_2], H^{m+2}(\T)\big).
	\end{equation} 
\end{Theorem}	

As previously mentioned, the proof of this result is based on first developing an independent regularity theory for the nonlinear filter equation (\ref{Filter-Equation}). In this analysis, assumptions (\ref{conditions-a-varphi-regularity}) play a key role in yielding a gain in the regularity of the filtered velocity $\vu_\alpha$; see Section \ref{Sec:Higher regularity properties of the nonlinear filter equation} below. Thereafter, through a rather standard bootstrap argument, we study the resulting gain in the regularity of the unfiltered velocity $\vu$ using its (mild) integral formulation:
\begin{equation*}
\vu(t,\cdot)=e^{\nu t\Delta }\vu_0 + \int_{0}^{t}e^{\nu(t-s) \Delta}\P(\vf)(s,\cdot) ds - \int_{0}^{t}e^{\nu (t-s) \Delta}\P\Big((\vu_\alpha \cdot \vec{\nabla})\vu\Big)(s,\cdot)ds,
\end{equation*} 
where the action of the operator $e^{\nu t\Delta }$ is classically represented by convolution with the heat kernel, and $\P$ denotes the Leray projector.

\medskip

Here, note that, given $\vu_0 \in L^2_{\rm s}(\T)$, the term $e^{\nu t \Delta}\vu_0$ does not gain regularity at $t=0$, and for this reason we must work on the time interval $[T_1,T_2]$ with $T_1>0$. On the other hand, for an external force satisfying (\ref{Reg-force}), the well-known \emph{maximal regularity} of the heat kernel (see Lemma \ref{Lem-Tech-2} below) yields at most
\[ \int_{0}^{t}e^{\nu(t-s) \Delta}\P(\vf)(s,\cdot) ds \in\mathcal{C}_t H^{m+1}_x \cap L^2_t \dot{H}^{m+2}_x.\]
 In this sense, the gain of regularity in $\vu$ given in (\ref{Regularity-Leray-Solution}) is \emph{optimal}.

\medskip

In the particular case when $\vec{f}\equiv 0$, one would expect a greater gain in regularity of the solution $\vu$. Indeed, this is the case, but the gain is not significantly greater than that obtained in Theorem \ref{Th-Regularity}, since it is ultimately controlled by the prescribed regularity of $A(\cdot)$ and $\varphi$.
\begin{Proposition}
	\label{Prop:Homogeneous-case} Under the same hypothesis of Theorem \ref{Th-Regularity}, assume that $\vf \equiv 0$. Then, it holds that 
	\begin{equation}\label{Regularity-Leray-Solution-improved}
		\vu \in \mathcal{C}\big([T_1,T_2], H^{m+2}(\T)\big) \cap  L^2\big([T_1,T_2], \dot{H}^{m+3}(\T)\big),
	\end{equation} 
	while the filtered velocity $\vu_\alpha$ maintains the same gain of regularity given in Theorem \ref{Th-Regularity}. 
\end{Proposition} 

Comparing (\ref{Regularity-Leray-Solution}) with (\ref{Regularity-Leray-Solution-improved}), it is interesting to observe that the improvement in regularity of the velocity $\vu$ does not continue indefinitely, provided that the prescribed regularity assumptions on $A(\cdot)$ and $\varphi$ given in (\ref{conditions-a-varphi-regularity}) are fixed. We refer to Remark \ref{Rmk:Reg-u-alpha} for further technical details. 

\medskip

In contrast, it is also interesting to observe that, in the \emph{linear} case of the Helmholtz equation (\ref{Helmholtz}), its smoothing effects are strong enough to indefinitely regularize the velocity $\vu$. More precisely, for the simplified Bardina model on the whole space $\Rt$:
\begin{equation*}
\partial_t \vu - \nu\Delta \vu + \left( (\vu \cdot \vec{\nabla})\vu\right)_\alpha + \vec{\nabla} P  =  0,  \quad {\rm div}(\vu)=0,
\end{equation*}
where, for a given vector field $\vec{\phi}$, the expression $(\vec{\phi})_\alpha$ denotes the solution of (\ref{Helmholtz}), the work \cite[Theorem $1.1$]{Berselli} shows that solutions \emph{remain smooth indefinitely}, in the sense that $\vu \in \mathcal{C}\big( [T_1, T_2], H^{m+n}(\Rt)\big)$ for any $n \in \mathbb{N}$.

\medskip

Returning to Theorem \ref{Th-Regularity}, since the times $0<T_1<T_2<+\infty$ in expression (\ref{Regularity-Leray-Solution}) are arbitrary, we directly deduce that, for any time $t>0$, we have $\| \vu(t,\cdot)\|_{H^{m+1}}<+\infty$. This naturally raises the question of whether, assuming that the external force $\vf$ is a time-independent function, this quantity remains uniformly bounded as the time $t$ goes to infinity.

\medskip

To the best of our knowledge, this long-time behavior of the higher-order $H^{m+1}$-norm of Leray solutions cannot be directly deduced from Theorem \ref{Th-Regularity}. Indeed, by performing a higher-order energy estimate on the velocity $\vu$ appearing in the first equation of system (\ref{Alpha-model}), together with the higher-order estimates for the filtered velocity $\vu_\alpha$ obtained in Section \ref{Sec:Higher regularity properties of the nonlinear filter equation}, we find, for any $t>0$, that
\[ \frac{d}{dt}\| \vu(t,\cdot)\|^2_{H^{m+1}} \lesssim   \| \vu_\alpha(t,\cdot) \|^2_{H^{m+2}}\, \| \vu(t,\cdot) \|^2_{H^{m+1}}+  \| \vf \|^2_{H^m}, \quad m\geq 0. \]
A closely related estimate is established in the technical Lemma \ref{Lem-Reg-u-alpha-Hm} below, whose proof also provides the main ideas leading to this high-order energy estimate.  Thus, in contrast to the well-known $L^2$ energy estimate
\[ \frac{d}{dt}\| \vu(t,\cdot)\|^2_{L^2} \lesssim  - \| \vu(t,\cdot) \|^2_{L^2}+  \| \vf \|^2_{H^{-1}}, \]
a classical Gr\"onwall-type argument does not appear to be sufficient to obtain uniform-in-time bounds on $\| \vu(t,\cdot)\|_{H^{m+1}}$ for arbitrary Leray solutions.

\medskip

In the next subsection, we show that this long-time uniform bound  can be addressed for solutions of the Navier-Stokes alpha model belonging to its global attractor.

\subsection{Regularity of the global attractor}\label{Sec:Regularity of the global attractor} 
We set our framework as follows: first, we fix a stationary (time-independent) divergence-free external force $\vf \in \dot{H}^{-1}(\T)\cap L^2_{\text{s}}(\T)$. Next, since the pressure terms $P$ and $P_\alpha$ do not play any substantial role in this study, from now on, we shall consider the system
\begin{equation}\label{Alpha-model-Proy}
\begin{cases}\vspace{2mm}
\partial_t \vu - \nu\Delta \vu + \P \big( (\vu_\alpha \cdot \vec{\nabla})\vu \big)  =  \vf, \quad {\rm div}(\vu)=0,  \\ \vspace{2mm}
-\alpha^2 \, \P \Big( {\rm div}\left( A (\varphi\ast\vu) (\vec{\nabla} \otimes \vu_\alpha)^T \right)\Big)+\vu_\alpha = \vu, \quad {\rm div}(\vu_\alpha)=0, \quad \alpha >0,  \\ 
\vu(0,\cdot)=\vu_0, \quad {\rm div}(\vu_0)=0.
\end{cases}
\end{equation}
Moreover, in the nonlinear filter equation above,  we shall assume that  the indicator function $A(\cdot)$ and the convolution term $\varphi$ satisfies the conditions (\ref{conditions-a-varphi-GWP}). Therefore, by Theorem \ref{Th-GWP} the system (\ref{Alpha-model-Proy}) defines a  continuous semi-group acting on the Hilbert space $L^2_{\text{s}}(\T)$ as follows:  for any time $t\geq 0$, we define the semi-group 
\begin{equation}\label{Semi-group}
S(t):L^2_{\text{s}}(\T)\to L^2_{\text{s}}(\T), \quad \vu_0 \mapsto S(t)\vu_0= \vu(t,\cdot),
\end{equation}
where $\vu(t,\cdot)$ denotes the \emph{unique}   Leray solution to (\ref{Alpha-model-Proy}) arising from the initial datum $\vu_0$.

\medskip

In our previous work \cite{CortezJarrin2026}, we studied the long-time behavior of the semi-group $S(t)$ through the notion of the global attractor. For completeness, we recall here some definitions and the main results obtained. 

\begin{defiprevia}[Global attractor, Definition $2.2$ of \cite{CortezJarrin2026}]\label{Def-Global-Attractor}  A global attractor for the semi-group $(S(t))_{t\geq 0}$ defined in (\ref{Semi-group}) is a set $\mathcal{A}\subset L^2_{\text{s}}(\T)$ which satisfies:
	\begin{enumerate}
		\item[(A)] The set $\mathcal{A}$ is compact in the strong topology of $L^2_s(\T)$, which is the same as the strong topology of $L^2(\T)$.
		\item[(B)] The set $\mathcal{A}$ is strictly invariant under $S(t)$: for any time $t\geq 0$, we have $S(t)\mathcal{A}=\mathcal{A}$.
		\item[(C)] $\mathcal{A}$ is an attracting set: for any bounded set $B \subset L^2_{\text{s}}(\T)$ and for any neighborhood $\mathcal{O}$ of $\mathcal{A}$, there exists a time $T=T(B,\mathcal{O})>0$ (which depends on the set $B$ and the neighborhood $\mathcal{O}$) such that for all $t>T$,   it holds $S(t)B\subset \mathcal{O}$. 
	\end{enumerate}
	\end{defiprevia} 
Here, point $(C)$ captures the essential long-time property of a global attractor. In broad terms, the asymptotic behavior of the semi-group $S(t)$ defined in (\ref{Semi-group}), and consequently that of Leray solutions to the Navier-Stokes alpha model (\ref{Alpha-model})), can be characterized in terms of the global attractor. More precisely, given any initial datum $\vu_0 \in L^2_{\rm s}(\T)$, setting $B:= \{ \vu_0\}$, and any neighborhood $\mathcal{O}$ of $\mathcal{A}$, the corresponding Leray solution $\vu(t,\cdot)$ eventually enters $\mathcal{O}$ as the time increases. Hence, every Leray solution of (\ref{Alpha-model})) approaches the global attractor $\mathcal{A}$ arbitrarily closely as $t\to +\infty$.

\medskip

With this first definition in mind, it is natural to ask about the elements that belong to a global attractor, whenever such an attractor exists. To this end, we introduce the following definition.

\begin{defiprevia}[Eternal solutions, Definition $2.3$ of \cite{CortezJarrin2026}]\label{Def-Eternal-Solutions} Let $\vf \in \dot{H}^{-1}(\T)\cap L^2_{\text{s}}(\T)$ be the stationary external force. We shall say that $\vu_{\text{e}}$  is an eternal solution to the  Navier-Stokes-alpha model associated with $\vf$ if:  
	\begin{enumerate}
		\item[(A)] We have $\ds{\vu_{\text{e}} \in L^{\infty}_{\rm loc}\big(\R, L^2_{\text{s}}(\T)\big) \cap L^{2}_{\rm loc}\big(\R,\dot{H}^1(\T)\big)}$. 
		
		\medskip 
		
		\item[(B)] The vector field $\vu_\text{e}$ solves the following   coupled system  in the weak sense:
		\begin{equation*}
		\begin{cases}\vspace{2mm}
		\partial_t \vu_\text{e} - \nu\Delta \vu_\text{e} + \P\big( (\vu_{\text{e},\alpha} \cdot \vec{\nabla})\vu_\text{e}\big)  = \vf, \quad {\rm div}(\vu_\text{e})=0, \\ 
		-\alpha^2 \, \P\Big( {\rm div} \left( A (\varphi \ast \vu_\text{e}) (\vec{\nabla} \otimes \vu_{\text{e},\alpha})^T \right)\Big)+\vu_{\text{e},\alpha} = \vu_\text{e}, \quad {\rm div}(\vu_{\text{e},\alpha})=0.
		\end{cases}
		\end{equation*} 
	\end{enumerate}
	Moreover, when $\vu_{\text{e}}\in L^{\infty}\big(\R, L^2_{\text{s}}(\T)\big)$, it is called a  \emph{bounded eternal solution}. We will denote by $\mathcal{E}(\vf)$ the set of all the bounded eternal solutions associates to the external force $\vf$.
\end{defiprevia} 
In \cite[Lemma $7.1$]{CortezJarrin2026}, we construct these bounded eternal solutions of the Navier-Stokes alpha model, thereby ensuring that $\mathcal{E}(\vf) \neq \emptyset$.

\medskip

With these definitions in place, we recall our previous result.

\begin{teoprevio}[Existence of a unique global attractor, Theorem $2.3$ of \cite{CortezJarrin2026}]\label{Th-Global-Attractor} In the nonlinear filter equation (\ref{Filter-Equation}), assume that $A(\cdot)$ and $\varphi$ satisfy (\ref{conditions-a-varphi-GWP}). Let $\vf \in \dot{H}^{-1}\cap L^2_{\text{s}}(\T)$ be the stationary external force. 
	
\medskip

Then, the semi-group $(S(t))_{t\geq 0}$ defined in (\ref{Semi-group}) has a unique \emph{global attractor} $\mathcal{A}=\mathcal{A}(\vf)$ in the sense of Definition \ref{Def-Global-Attractor}. Moreover, we have the following characterization:
	\begin{equation*}
	\mathcal{A}(\vf) = \left\{  \va \in  L^2_{\text{s}}(\T) : \ \ \va:=\vu_\text{e}(0,\cdot)  \ \ \text{where} \ \ \vu_{\text{e}} \in \mathcal{E}(\vf)  \ \right\},
	\end{equation*} 
	where the set $\mathcal{E}(\vf)$ is given in Definition \ref{Def-Eternal-Solutions}.
\end{teoprevio} 



We now introduce the final definition needed for our analysis.


\begin{Definition}\label{Def-Complete-Trajectory} Let $\vf \in \dot{H}^{-1}(\T)\cap L^2_{\text{s}}(\T)$ be the stationary external force and let $\mathcal{A}(\vf)$ be the global attractor given by Definition \ref{Def-Global-Attractor} and obtained in Theorem \ref{Th-Global-Attractor}. 	
	\medskip
	
	A complete trajectory in the global attractor  is an 
	eternal solution	$\vu_{\text{e}}$   of the Navier-Stokes-alpha model, given in Definition \ref{Def-Eternal-Solutions}, which satisfies 
	\[ \vu_{\text{e}}(t,\cdot) \in \mathcal{A}(\vf), \qquad \text{for any time $t \in \R$}.   \]

	We will denote by $\mathcal{T}(\mathcal{A})$ the set of all the complete trajectories in the global attractor $\mathcal{A}(\vf)$. 
\end{Definition}	

Following the same ideas as in the proof of \cite[Lemma $7.1$]{CortezJarrin2026} and using the strict invariance property of $\mathcal{A}(\vf)$ given in point $(B)$ of Definition \ref{Def-Global-Attractor}, we readily deduce that $\mathcal{T}(\mathcal{A}) \neq \emptyset$.

\medskip

One of the main features of the global attractor $\mathcal{A}(\vf)$ is that it is contained in a precise absorbing set.  From the external force $\vf$, the viscosity constant $\nu$ and the fixed period $L>0$, define the radius
\begin{equation}\label{R}
\mathcal{R}^2:=\frac{L^2}{2\pi^2 \nu^2}\| \vf \|^2_{\dot{H}^{-1}}>0,
\end{equation}
and the closed ball
\begin{equation}\label{Absorbing-Set}
\mathcal{B}=\left\{  \vg \in L^2_{\text{s}}(\T) : \ \| \vg \|^2_{L^2} \leq \mathcal{R}^2\right\}.
\end{equation}

Then it is well-known that  we have  
\[ \mathcal{A}(\vf)\subset \mathcal{B}, \]
see for instance  \cite{Raugel,Temam}.  Using this embedding and the notion of complete trajectories introduced in Definition \ref{Def-Complete-Trajectory} above, we directly obtain the uniform bound
\begin{equation}\label{Unif-Control-L2-Intro}
\sup_{\vu_{\rm e} \in \mathcal{T}(\mathcal{A})} \left( \sup_{t \in \R} \| \vu_{\rm e} (t,\cdot)\|^2_{L^2}\right) \leq \mathcal{R}^2. 
\end{equation}

Within this context, and inspired by the previous work \cite{Paicu}, the precise goal of this subsection is to establish analogous estimates when the standard $L^2$-norm is replaced by the $H^{m+1}$-norm with the regularity parameter $m \geq 0$.

\medskip

First, we will consider the case when $m=0$, and therefore we aim to study a uniform bound in the $H^1$-norm. Interestingly, we show that  in this case the global attractor $\mathcal{A}(\vf)\in L^2_{\rm s}(\T)$, obtained in Theorem \ref{Th-Global-Attractor}, exhibits an \emph{instantaneous smoothing effect} in the $H^1$-norm \emph{without any additional assumptions} on the external force $\vf$ or on the functions $A(\cdot)$ and $\varphi$ in the filter equation (\ref{Filter-Equation}).

\medskip

For the reader's convenience, we recall below the assumptions made in Theorem \ref{Th-Global-Attractor}.

\begin{Theorem}
	\label{Th-Regularity-Atrtractor}
	Let $\mathcal{A}(\vf)$ be the global attractor given by Theorem \ref{Th-Global-Attractor}, where
	\begin{equation}\label{Reg-force-attractor-L2}
	\vf \in L^2_{\rm s}(\T),
	\end{equation}
	and, in the filter equation (\ref{Filter-Equation}),
	\begin{equation}\label{conditions-a-varphi-regularity-attractor-H1}
	A(\cdot)\in \mathcal{C}^{1,1}(\Rt)  \quad \text{and} \quad \varphi \in H^{1}(\T).
	\end{equation}
	Then, there exists a constant $\mathcal{K}_0>0$, depending on $\vf$, $A(\cdot)$, $\varphi$, as well as on the parameters $\alpha$ and $\nu$, such that the following uniform estimate holds:
	\begin{equation}
	\label{Regularity-Atrractor}
	\sup_{\vu_{\rm e} \in \mathcal{T}(\mathcal{A}) }\, \left( \sup_{t\in \R}\| \vu_{\rm e}(t,\cdot) \|^2_{H^{1}} \right) \leq \mathcal{K}_0,
	\end{equation}
	where the set $\mathcal{T}(\mathcal{A})$ is given in Definition \ref{Def-Complete-Trajectory}.
\end{Theorem}

The proof of this theorem is based on the following idea. For any $\vu_{\rm e} \in \mathcal{T}(\mathcal{A})$, the main strategy is to decompose this complete trajectory as
\[\vu_{\rm e} = \vv_n + \vw_n,\]
where $(\vv_n)_{n \in \mathbb{N}}$ and $(\vw_n)_{n \in \mathbb{N}}$ are suitable sequences of divergence-free vector fields. Roughly speaking, we show that the functions $\vv_n$ remain uniformly bounded in $H^1(\T)$, while the functions $\vw_n$ converge to zero as $n\to \infty$ in $L^2(\T)$.

\medskip

We now explain the overall strategy of the proof in more detail. First, for a fixed time $t_0\in \R$, we construct the vector fields $\vv_n= \vv_{t_0,n}$ and $\vw_n= \vw_{t_0,n}$, which depend on $t_0$ and essentially solve the following affine equations:
\begin{equation*}
\partial_t \vv_{t_0,n}-\nu\Delta \vv_{t_0,n}+\P\big((\vu_{{\rm e},\alpha}\cdot\vec{\nabla})\vv_{t_0,n}\big)=\vf \qquad \text{and} \qquad  \partial_t \vw_{t_0,n}-\nu\Delta \vw_{t_0,n}+\P\big((\vu_{{\rm e},\alpha}\cdot\vec{\nabla})\vw_{t_0,n}\big)=0,
\end{equation*}
respectively. Here, $\vu_{{\rm e},\alpha}$ denotes the filtered trajectory obtained from $\vu_{\rm e}$ as the solution of the nonlinear filter equation (\ref{Filter-Equation}).

\medskip

For the first equation, proceeding as in the proof of estimate (\ref{Unif-Control-L2-Intro}), we can directly deduce the following uniform $L^2$-bound:
\begin{equation}\label{Unif-Control-L2-Intro-vn}
\sup_{t_0\in \R, \, n \in \mathbb{N}} \| \vw_{t_0,n}(t,\cdot)\|^2_{L^2} \lesssim \mathcal{R}^2,
\end{equation}
which holds for any time $t\in \R$, where the radius $\mathcal{R}$ is defined in (\ref{R}). Then, using the higher regularity of $\vu_{{\rm e},\alpha}$ established in Section \ref{Sec:Higher regularity properties of the nonlinear filter equation} under the assumptions (\ref{conditions-a-varphi-regularity-attractor-H1}), together with the uniform $L^2-$ bound (\ref{Unif-Control-L2-Intro-vn}) and assumption (\ref{Reg-force-attractor-L2}) on the external force $\vf$, we can derive \emph{new $H^1$- estimates} that yield the following uniform bound \emph{on the time window} $[t_0, t_0+1]$:
\begin{equation}\label{Estim-vn-Intro}
\sup_{t_0 \leq t \leq t_0+1} \| \vv_{t_0,n}(t,\cdot)\|^2_{H^1} \lesssim \mathcal{K}_0,
\end{equation}
where $\mathcal{K}_0$ is independent of $\vv_{n,t_0}$, $t_0$ and $n$. 

\medskip

For the second equation, the right-hand side vanishes. Standard $L^2$-energy estimates, together with the divergence-free properties of $\vw_{t_0,n}$ and $\vu_{{\rm e}, \alpha}$, therefore give
\begin{equation}\label{Conv-wn-Intro}
\vw_{t_0,n}(t_0,\cdot)\to 0, \qquad n \to +\infty,
\end{equation}
in the strong topology of the space $L^2(\T)$.

\medskip

Thus, the uniform estimate (\ref{Regularity-Atrractor}) follows by combining (\ref{Estim-vn-Intro}) and (\ref{Conv-wn-Intro}) and using the fact that $t_0 \in \R$ was arbitrary.

\medskip

With this program in mind, it is natural to seek uniform higher-order regularity bounds analogous to (\ref{Regularity-Atrractor}) in the $H^{m+1}$-norm for $m\geq 1$. However, a technical difficulty arises. The key estimate (\ref{Estim-vn-Intro}) relies on the preceding $t_0$-uniform estimate (\ref{Unif-Control-L2-Intro-vn}), but is itself \emph{valid only on the time window} $[t_0, t_0+1]$ for fixed $t_0$. Consequently, an $H^2$- estimate analogous to (\ref{Estim-vn-Intro}) would require a prior $H^1$- estimate that is uniform in $t_0$ (as in (\ref{Unif-Control-L2-Intro-vn})). To the best of our knowledge, such an estimate cannot be deduced from (\ref{Estim-vn-Intro}).

\medskip

To overcome this difficulty when $m\geq 1$, we follow some of the ideas in \cite{Paicu} and assume an \emph{a priori} bound on trajectories in the global attractor. Specifically, for any $\vu_{\rm e} \in \mathcal{T}(\mathcal{A})$ and any time $t\in \R$, the higher-order regularity estimate for the filter equation (\ref{Filter-Equation}) established in Theorem \ref{Th:Hm-Filter} below, together with the standard Sobolev embedding, gives
\[ \| \vu_{{\rm e}, \alpha}(t,\cdot)\|_{W^{m,\infty}} \lesssim \| \vu_{{\rm e}, \alpha}(t,\cdot)\|_{H^{m+2}} \lesssim \| \vu_{\rm e}(t,\cdot)\|_{H^m}\lesssim \| \vu_{\rm e}(t,\cdot)\|_{H^{m+1}}.  \]
This chain of estimates shows that a uniform \emph{a priori} bound on $\| \vu_{{\rm e}, \alpha}(t,\cdot)\|_{W^{m,\infty}}$ is a weaker assumption than a uniform bound on $ \| \vu_{\rm e}(t,\cdot)\|_{H^{m+1}}$. Under the stated assumptions, we use the former to establish the latter.

\begin{Proposition}
	\label{Prop-Hihger-Regularity-Atrtractor} 
	Let $m\in \mathbb{N}$ be such that $m\geq 1$. Within the setting of Theorem \ref{Th-Global-Attractor}, assume that the external force satisfies
	\begin{equation}\label{Reg-force-attractor}
	\vf \in H^m(\T).   
	\end{equation}
	Moreover, in the filter equation (\ref{Filter-Equation}), assume that
	\begin{equation}\label{conditions-a-varphi-regularity-attractor}
	A(\cdot)\in \mathcal{C}^{m+1,1}(\Rt)  \quad \text{and} \quad \varphi \in H^{m+1}(\T).
	\end{equation}
	
On the other hand, given the viscosity constant $\nu$ and the set $\mathcal{T}(\mathcal{A})$ introduced in Definition \ref{Def-Complete-Trajectory}, assume that
\begin{equation}\label{Assumption-unif-bound-filtered}
\Lambda(\alpha,m):= \sup_{\vu_{\text{e}} \in \mathcal{T}(\mathcal{A})}\left( \sup_{t \in \R} \| \vu_{\text{e},\alpha}(t,\cdot)\|_{W^{m,\infty}} \right)<\nu.
\end{equation}
	
	Then,  the following uniform estimate holds:
	\begin{equation}
	\label{Regularity-Atrractor-H-m+1}
	\sup_{\vu_{\rm e} \in \mathcal{T}(\mathcal{A}) }\, \left( \sup_{t\in \R}\| \vu_{\rm e}(t,\cdot) \|^2_{H^{m+1}} \right) \leq \frac{C}{\nu^2-\Lambda^2(\alpha,m)}\, \| \vf \|^2_{H^m},
	\end{equation} 
	where $C>0$ is a generic constant. 
\end{Proposition}

The following comments are in order. First, the aforementioned bound on $\| \vu_{{\rm e}, \alpha}(t,\cdot)\|_{W^{m,\infty}}$ is explicitly stated in (\ref{Assumption-unif-bound-filtered}), where we assume that the $L^\infty$-norm of the $m$-th derivatives of filtered trajectories in the global attractor is controlled by the viscosity $\nu$. As mentioned above, a similar type of \emph{a priori} bound involving the $L^\infty$-norm and the viscosity $\nu$ appears in the interesting related work \cite{Paicu} on second-grade fluid equations.

\medskip

Under (\ref{Assumption-unif-bound-filtered}), together with the natural higher-order regularity assumptions (\ref{Reg-force-attractor}) and (\ref{conditions-a-varphi-regularity-attractor}) on the external force $\vf$ and the functions $A(\cdot)$ and $\varphi$ in the nonlinear filter equation (\ref{Filter-Equation}), we obtain a uniform $H^{m+1}$ bound on the trajectories in the global attractor. This gives a partial answer to the question raised above.

\medskip

In particular, the right-hand side of (\ref{Regularity-Atrractor-H-m+1}) is well-defined only when $\nu^2-\Lambda^2(\alpha,m)>0$, as required by (\ref{Assumption-unif-bound-filtered}). To the best of our knowledge, obtaining a uniform higher-order regularity bound as in (\ref{Regularity-Atrractor-H-m+1}) under assumptions (\ref{Reg-force-attractor}) and (\ref{conditions-a-varphi-regularity-attractor}) alone remains out of reach and is a possible subject for future research.

\medskip

The result of Proposition \ref{Prop-Hihger-Regularity-Atrtractor} also holds when $m=0$. In this case, however, Theorem \ref{Th-Regularity-Atrtractor} gives a stronger regularity result for the global attractor.

\medskip

Theorem \ref{Th-Regularity-Atrtractor} also provides a key tool for studying another property of the global attractor: the existence of finitely many determining modes. Defined precisely in Definitions \ref{def:asymptotic-determining-modes} and \ref{def:determining-modes-on-attractor} below, this property gives a more detailed description of the long-time behavior of the Navier-Stokes alpha model (\ref{Alpha-model}), as we explain in the final subsection.

\subsection{Refined long-time behavior}\label{Sec:Refined long-time behavior}  A central question in turbulence theory is that, despite the infinite-dimensional character of the Navier--Stokes dynamics, the long-time behavior of solutions is effectively governed by finitely many degrees of freedom. Mathematically, this leads to the concept of finitely many determining modes or, more generally, determining functionals. Roughly speaking, these concepts aim to quantify how much coarse information on a solution suffices to recover its asymptotic dynamics \cite{Cheskidov,FoiasProdi1967,FoiasTiti1991,JonesTiti1992,JonesTiti1993}.

\medskip

The notion was first introduced in a seminal work \cite{FoiasProdi1967}  for the two-dimensional Navier--Stokes equations. There, the authors proved  that there exists an integer $m$ such that, if the first $m$ Fourier modes of two solutions have the same asymptotic behavior as $t\to+\infty$, then the entire solutions synchronize asymptotically, i.e., the remaining infinitely many modes are uniquely determined in the limit $t\to+\infty$.  This result provides a rigorous formulation of finite-dimensionality in the long-time regime and has motivated a large literature devoted to explicit estimates in terms of physical and geometric parameters.

\medskip



Thereafter, the original determining-modes theory has been sharpened and generalized in many directions, including determining nodes and volume elements, and extensions to broader classes of dissipative systems; see, for instance, \cite{Chueshov1998,FoiasTiti1991,JonesTiti1992,JonesTiti1993,Kalantarov}.

\medskip

In this subsection, we derive \emph{explicit lower bounds} for the number of determining modes for the Navier--Stokes--alpha model (\ref{Alpha-model}). We begin by introducing some definitions and notation. Let $(\vw_m)_{m\ge1}$ be the divergence-free and mean-zero eigenfunctions of the Stokes operator on $\T$:
\begin{equation}\label{Stokes}
-\nu \P \Delta \vw_m =  \lambda_m \vw_m, \quad \text{div}(\vw_m)=0,
\end{equation}
associated with increasing sequence of   eigenvalues $ (\lambda_m)_{m\geq 1}$ (where $\lambda_1>0$), and let
\begin{equation}\label{span-Hm}
H_m(\T):=\operatorname{span}\{\vw_1,\cdots,\vw_m\}.
\end{equation}
We denote by $P_m: L^2_{\text{s}}(\T)\to H_m(\T)$ the standard  $L^2$--orthogonal projector, given by the expression $\ds{P_m(\vg):=\sum_{k=1}^{m} \big\langle \vg, \vw_k \big\rangle_{L^2\times L^2}\, \vw_k}$. 

\medskip

Fist, we introduce the notion of an \emph{asymptotic} set of determining modes for the Navier-Stokes alpha model (\ref{Alpha-model}) in the general setting of weak Leray solutions given by Definition \ref{Def-Leray-Solutions}.

\begin{Definition}\label{def:asymptotic-determining-modes} Let $m  \geq 1$ be fixed. Then $\{\vw_1,\cdots,\vw_m\}$ is called an \emph{asymptotic} set of determining modes for the Navier-Stokes-Alpha model (\ref{Alpha-model}) if the following property holds.
 
\medskip

For  any  initial data $\vu_{0,1}, \vu_{0,2} \in L^2_{\text{s}}(\T)$, let $\vu_1, \vu_2 \in L^\infty_{loc}\big([0,+\infty[, L^2_{\text{s}}(\T) \big)\cap L^2_{loc}\big( [0,+\infty[, \dot{H}^1(\T)\big)$  be the unique associated Leray solutions of    (\ref{Alpha-model}),  obtained in  Theorem \ref{Th-GWP}. 

\medskip

If for the orthogonal projector $P_m$ defined above we have 
\begin{equation}\label{Determining-Modes-Pm}
\lim_{t\to+\infty}\left\|P_m\Big(\vu_1(t,\cdot)-\vu_2(t,\cdot)\Big)\right\|_{L^2}=0,
\end{equation}
then 
\begin{equation}\label{Determining-Modes-solution}
	\lim_{t\to+\infty}\|\vu_1(t,\cdot)-\vu_2(t,\cdot)\|_{L^2}=0.
\end{equation}
\end{Definition}

The key features of this definition are given by (\ref{Determining-Modes-Pm}) and (\ref{Determining-Modes-solution}). It is well known that the family of eigenvectors of the Stokes operator $(\vw_m)_{m \geq 0}$ yields the decomposition
\[ \vu(t,\cdot)= \sum_{m\geq 1}  \big\langle \vu(t,\cdot) , \vw_m \big\rangle_{L^2\times L^2} \, \vw_m, \]
with convergence in the strong topology of $L^2_{\rm s}(\T)$. In this framework, (\ref{Determining-Modes-Pm}) and (\ref{Determining-Modes-solution}) show that, for fixed $m$, the long-time difference between two Leray solutions, $\vu_1(t,\cdot)-\vu_2(t,\cdot)$, can be determined from the finite-dimensional quantity $P_m\Big(\vu_1(t,\cdot)-\vu_2(t,\cdot)\Big)$. This determination of the long-time behavior of the higher modes by \emph{finitely many lower} modes is therefore useful in computational and experimental studies \cite{CockburnJonesTiti1997,FoiasTiti1991,OlsonTiti2003}, with particular applications to data assimilation algorithms \cite{Albanez,Korn}.

\medskip

Now, we can state our first result of this subsection. For this, we introduce the following function with respect to the parameter $\alpha$ in the nonlinear filter equation (\ref{Filter-Equation}):
\begin{equation}\label{Phi}
\Phi(\alpha):= \begin{cases}\vspace{2mm}
\ds{\frac{1}{\alpha^2}}, & \quad 0<\alpha < 1, \\
\alpha^4, & \quad \alpha >1.  
\end{cases}
\end{equation}

\begin{Theorem}\label{Th:Determining-Modes}
Let $m_0\geq 1$ be sufficiently large such that its associated eigenvalue of the stokes operator (\ref{Stokes}), denoted as $\lambda_{m_0}$, satisfies 
		\begin{equation}\label{eq:lambda-condition-explicit}
	\lambda_{m_0} \gtrsim \frac{\Phi^2(\alpha)}{\nu^2} \left( \frac{\| \vf \|_{\dot{H}^{-1}}}{\nu} +1\right)^4\, \left( \frac{\| \vf \|^2_{\dot{H}^{-1}}}{\nu^3}+ \frac{\| \vf \|^2_{\dot{H}^{-1}}}{\nu^2}\right).
	\end{equation}
	
	Then, for any $m\geq  m_0$,  the set $\{  \vw_1, \cdots, \vw_m \}$ is an \emph{asymptotic} set of determining modes for  (\ref{Alpha-model}), in the sense of Definition \ref{def:asymptotic-determining-modes}.
\end{Theorem}

In this result, we first prove the existence of a finite set of asymptotically determining modes for the Navier-Stokes alpha model (\ref{Alpha-model}). We also derive a lower bound on the parameter $m_0$. Indeed, using $\lambda_{m_0} \simeq m^2_0$, we obtain
\begin{equation}\label{Lower-Bound-m0}
m_0  \gtrsim  \frac{\Phi(\alpha)}{\nu} \left( \frac{\| \vf \|_{\dot{H}^{-1}}}{\nu} +1\right)^2\, \left( \frac{\| \vf \|^2_{\dot{H}^{-1}}}{\nu^3}+ \frac{\| \vf \|^2_{\dot{H}^{-1}}}{\nu^2}\right)^{\frac{1}{2}}.
\end{equation}
Thus, the number of modes needed to describe the long-time asymptotic behavior of Leray solutions of (\ref{Alpha-model}) is bounded below by a quantity essentially expressed in terms of the filtering parameter $\alpha$ through expression (\ref{Phi}), the viscosity $\nu$ and the $\dot{H}^{-1}$-norm of the external force $\vf$. Additionally,  in the homogeneous case when $\vf \equiv 0$, we observe that  the lower bound on $m_0$ is essentially controlled by the factor $\ds{\frac{\Phi(\alpha)}{\nu}}$. 

\medskip

Returning to the lower bound above, it is interesting to examine how it depends on the filtering parameter $\alpha$. For simplicity, taking $\nu \simeq 1$ and $\| \vf \|_{\dot{H}^{-1}}\simeq 1$, we obtain from the first case in (\ref{Phi})
\begin{equation*}
m_0  \gtrsim  \frac{1}{\alpha^2} \quad  \text{when} \quad  0<\alpha < 1.
\end{equation*}
This bound diverges as $\alpha \to 0^{+}$, so it provides no finite estimate in that limit. Accordingly, Theorem \ref{Th:Determining-Modes} applies only when $\alpha >0$, which exhibits the strong influence of the nonlinear filter equation (\ref{Filter-Equation}) on the refined long-time asymptotic behavior of the full system (\ref{Alpha-model}).  

\medskip

In the limit as $\alpha \to 0^{+}$, it follows from \cite[Proposition $2.1$]{CortezJarrin2026} that Leray solutions of (\ref{Alpha-model}) converge in the weak-$*$ topology of the energy space to a Leray solution of the classical Navier-Stokes equations. For the latter equations, the property of \emph{asymptotic}  finitely many determining modes was studied in \cite{Cheskidov}. In contrast to Theorem \ref{Th:Determining-Modes}, the main difference lies in the fact that, roughly speaking, the parameter $m_0$ also depends on time $t$ through a technical quantity $Q(t)$. See \cite[Theorems $1.1$ and $1.2$]{Cheskidov} for further details. In this sense, the lower bound (\ref{Lower-Bound-m0}) provides a stronger result, since its right-hand side depends only on the physical parameters of the model.

\medskip

On the other hand, still taking $\nu \simeq 1$ and $\| \vf \|_{\dot{H}^{-1}}\simeq 1$, the second case in (\ref{Phi}), together with estimate (\ref{Lower-Bound-m0}), yields
\begin{equation*}
m_0  \gtrsim  \alpha^4   \quad  \text{when} \quad  \alpha >1.
\end{equation*}
This bound shows that larger values of the filtering parameter $\alpha$ lead to a larger lower bound on $m_0$. Surprisingly, this growth, which results from the \emph{nonlinear effects} of the filter equation (\ref{Filter-Equation}), differs from that obtained in previous works \cite{Holst,TheThiTran} involving the \emph{linear} Helmholtz filter (\ref{Helmholtz}), for which $m_0 \gtrsim 1$.
 
\medskip

Theorem \ref{Th:Determining-Modes} provides a general result on \emph{asymptotically} determining modes within the framework of Leray solutions of the Navier-Stokes alpha model (\ref{Alpha-model}). Nevertheless, it is interesting to observe that this asymptotic-in-time property becomes a stronger pointwise-in-time property when one considers solutions of (\ref{Alpha-model}) belonging to the global attractor.

\begin{Definition}\label{def:determining-modes-on-attractor} Let $\mathcal{A}(\vf)$ be the global attractor obtained in Theorem \ref{Th-Global-Attractor} and let $\mathcal{T}(\mathcal{A})$ be the set given in Definition \ref{Def-Complete-Trajectory}. Let $m\geq 1$ be fixed. The set $\{ \vw_1, \cdots, \vw_m \}$ is called a \emph{strong} set of determining modes on $\mathcal{A}(\vf)$ if the following property holds. 
	
\medskip

For the orthogonal projector $P_m$ and any $\vu_{\rm e,1}, \vu_{\rm e,2}\in \mathcal{T}(\mathcal{A})$, if 
\begin{equation}\label{Assumption-Determining-Modes-Attractor}
P_m \big( \vu_{\rm e,1}(t,\cdot) \big)= P_m \big( \vu_{\rm e,2}(t,\cdot) \big), \qquad \text{for all} \quad t \in \R, 
\end{equation}
then
\begin{equation}\label{Conclusion-Determining-Modes-Attractor}
\vu_{\rm e,1}(t,\cdot) = \vu_{\rm e,2}(t,\cdot), \qquad \text{for all} \quad t \in \R. 
\end{equation}
\end{Definition}

The uniform $H^1$-bound obtained in Theorem \ref{Th-Regularity-Atrtractor} is one of the key ingredients  to proof the main result of this subsection.  
\begin{Theorem}\label{Th:Determining-modes-on-attractor} Let $n_0\geq 1$ be sufficiently large such that its associated eigenvalue of the Stokes operator (\ref{Stokes}) satisfies
	\begin{equation}\label{eq:lambda-condition-explicit-attractor}
	\lambda_{n_0} \gtrsim \frac{\Phi^2(\alpha)}{\nu^2} \left(\frac{\|\vf \|_{\dot{H}^{-1}}}{\nu}+1\right)^4\,  \mathcal{K}_0,
	\end{equation}
	where $\mathcal{K}_0>0$ is the constant given in estimate (\ref{Regularity-Atrractor}), proven in Theorem \ref{Th-Regularity-Atrtractor}, and the function $\Phi(\alpha)$ is defined in expression (\ref{Phi}).  
	
	\medskip
	
	Then, for any $m\geq n_0$, the set $\{  \vw_1, \cdots, \vw_m \}$ is a \emph{strong} set of determining modes on $\mathcal{A}(\vf)$ in the sense of Definition \ref{def:determining-modes-on-attractor}.
\end{Theorem}

To the best of our knowledge, this stronger property of finitely many determining modes for the global attractor was first proved in \cite{Ladyzhenskaya1963} for the two-dimensional Navier-Stokes equations, while its validity in the three-dimensional case remains a nontrivial open problem. In the setting of the three-dimensional Navier-Stokes alpha model (\ref{Alpha-model}), the presence of the function $\Phi(\alpha)$ on the right-hand side of estimate (\ref{eq:lambda-condition-explicit-attractor}) shows that this property holds owing to the filtering parameter $\alpha>0$ and, ultimately, to the effects of the coupled filter equation (\ref{Filter-Equation}). We provide further comments on this particular result in the final subsection.

\medskip

\subsection{Conclusion on the main contributions of this work}\label{Sec:Contributions}
The main contributions of this work are Theorem
\ref{Th-Regularity-Atrtractor} and Theorem
\ref{Th:Determining-modes-on-attractor}. The former establishes
uniform higher-order regularity for the global attractor, which is a
fundamental ingredient in the proof of the strong determining-modes
property for complete trajectories established in the latter.

\medskip 

More precisely, Theorem
\ref{Th:Determining-modes-on-attractor} shows that, whenever the
projections of two complete trajectories on the global attractor onto
a sufficiently large finite-dimensional space of Stokes modes coincide
at every time $t\in\mathbb{R}$, the two trajectories coincide identically
at every time $t\in\mathbb{R}$. This is substantially stronger than the
usual asymptotic determining-modes property, which only ensures that
the difference between two trajectories converges to zero as
$t\to+\infty$ when the difference between their low-mode projections
does so.

\medskip

To the best of our knowledge, this is the first time that such a strong determining-modes property for complete trajectories has been explicitly established for a three-dimensional Navier--Stokes-$\alpha$ model involving a nonlinear filter equation whose behavior is strongly governed by a prescribed indicator function $A(\cdot)$. Together with the uniform regularity result of Theorem \ref{Th-Regularity-Atrtractor}, this finding opens the door to a broader research program involving physically and experimentally relevant indicator functions $A(\cdot)$ in the filter equation \eqref{Filter-Equation}; see, for instance, \cite{Geurts0,Geurts1,Layton,Layton2}.

\medskip

Related results on asymptotically determining observations and continuous data assimilation are available for some three-dimensional $\alpha$-models based on the linear Helmholtz filter \eqref{Helmholtz}, most notably the LANS-$\alpha$ model \cite{Korn} and the simplified Bardina model \cite{Albanez}. These results primarily establish \emph{asymptotic} reconstruction or synchronization from finitely many observations, rather than the stronger identification of complete trajectories on the global attractor obtained in Theorem \ref{Th:Determining-modes-on-attractor}. Although some of their estimates could implicitly yield an analogous conclusion through an additional argument involving complete bounded trajectories, this stronger property does not appear to have been explicitly formulated or established in the literature. This observation suggests a natural direction for further research on three-dimensional $\alpha$-models involving linear filters.

\medskip 

\textbf{Organization of the article}. Section \ref{Sec:Higher regularity properties of the nonlinear filter equation} is devoted to the higher-order regularity theory developed for the nonlinear elliptic equation (\ref{Filter-Equation}). In Section \ref{Sec:Reg-Leray}, we prove Theorem \ref{Th-Regularity} and Proposition \ref{Prop:Homogeneous-case}. Section \ref{Sec:Reg-Attractor} is devoted to the proofs of Theorem \ref{Th-Regularity-Atrtractor} and Proposition \ref{Prop-Hihger-Regularity-Atrtractor}. Finally, in Section \ref{Sec:Determining-Modes}, we prove Theorem \ref{Th:Determining-Modes} and Theorem \ref{Th:Determining-modes-on-attractor}.

\medskip

\textbf{Preliminaries and notation}. We collect here some standard definitions and notation in the setting of periodic boundary conditions $\T:=[0,L]^3$. As usual, we assume that both the initial data $\vu_0$ and the external force $\vf$ are mean-free vector fields, \emph{i.e.}, $\int_{\T}\vu_0(x)\,dx=\int_{\T}\vf(x)\,dx=0$. This condition ensures that the velocity field $\vu(t,x)$ remains mean-free for all $t\geq 0$. For convenience, we set ${\bf Z}^3:=\frac{2\pi}{L}\mathbb{Z}^3$, and we frequently use the Fourier expansion
\begin{equation}\label{Fourier-decomposition}
\vu(t,x)=\sum_{k\in \Z} e^{ik\cdot x}\,\widehat{\vu}_k(t,\cdot), \quad  \widehat{\vu}_k(t,\cdot):=\frac{1}{L^3}\int_{\T} e^{-ik\cdot x}\vu(t,x)\,dx.
\end{equation}

We denote by $L^p(\T)$, for $1\leq p\leq+\infty$, the usual Lebesgue spaces, and $(\cdot,\cdot)_{L^2}$ denotes the $L^2$ inner product. The space $\dot{L}^2(\T)$ consists of $L^2$ functions with zero mean. By the Helmholtz decomposition (see \cite[Theorem 2.6]{Wiedemann}), we have
\[
\dot{L}^2(\T)=L^2_{\mathrm{s}}(\T)\oplus (L^2_{\mathrm{s}})^\perp(\T),
\]
where $(L^2_{\mathrm{s}})^\perp(\T)$ denotes the set of mean-free vector fields $\vg$ such that $\vg=\nabla \phi$ with $\phi\in \dot{H}^1(\T)$.

\medskip

For $s\in\R$, the homogeneous Sobolev spaces of zero-mean functions can be characterized in terms of Fourier series as follows:
\begin{equation*}
\dot{H}^s(\T)=\Big\{ g(x)=\sum_{k\in \Z} e^{ik\cdot x}\widehat{g}_k:\ 
\overline{\widehat{g}_k}=\widehat{g}_{-k},\ \widehat{g}_0=0,\ 
\text{and}\ 
\|g\|^2_{\dot{H}^s}:=\sum_{k\in \Z}|k|^{2s}|\widehat{g}_k|^2<\infty \Big\}.
\end{equation*}
The condition $\overline{\widehat{g}_k}=\widehat{g}_{-k}$ ensures that $g$ is real-valued. Note that $\dot{H}^0(\T)=\dot{L}^2(\T)$, and whenever $s_1<s_2$, the embedding $\dot{H}^{s_2}(\T)\subset \dot{H}^{s_1}(\T)$ is \emph{compact}. Similarly, the non-homogeneous Sobolev spaces are defined by
\begin{equation*}
H^s(\T)=\Big\{ g(x)=\sum_{k\in \Z} e^{ik\cdot x}\widehat{g}_k:\ 
\overline{\widehat{g}_k}=\widehat{g}_{-k},\ \widehat{g}_0=0,\ 
\text{and}\ 
\|g\|^2_{H^s}:=\sum_{k\in \Z}(1+|k|^2)^s|\widehat{g}_k|^2<\infty \Big\}.
\end{equation*}

Next, let $\P:\dot{L}^2(\T)\to L^2_{\mathrm{s}}(\T)$ denote the Leray projector, defined by
\[
\P(\vg)(x)=\sum_{k\in \Z} e^{ik\cdot x}\left(\widehat{\vg}_k-\frac{\widehat{\vg}_k\cdot k}{|k|^2}k\right).
\]
For $s\in\R$, fractional powers of the Laplacian are defined by
\begin{equation*}
(-\Delta)^{\frac{s}{2}}g(x)=\sum_{k\in \Z} e^{ik\cdot x}\,|k|^s\,\widehat{g}_k,
\end{equation*}
and the Bessel potential  defined by 
\begin{equation}\label{Bessel-Potental}
(I_d-\Delta)^{\frac{s}{2}}g(x)=\sum_{k\in \Z} e^{ik\cdot x}\,(1+|k|^2)|^{\frac{s}{2}}\,\widehat{g}_k.
\end{equation}
Finally, we recall that the  heat kernel semigroup  in the periodic setting is given by
\begin{equation}\label{Heat-Kernel}
e^{\nu t\Delta}\vg=\sum_{k \in \Z}e^{i k\cdot x}\, e^{- \nu |k|^2 t}\, \widehat{\vg}_k.
\end{equation}

On the other hand, for simplicity, given $A,B \geq 0$, we will  use the notation $A \simeq B$ to mean that $c_1A \leq B \leq c_2B$ for some constants $0<c_1 \leq c_2$ independent of $A$ and $B$. Similarly, we will write $A \lesssim B$ to mean that $A \leq c B$. In addition, we systematically use $C >0$ to denote a generic constant whose value may change from one line to the next and whose precise value is immaterial.

\section{Higher regularity properties of the nonlinear filter equation}\label{Sec:Higher regularity properties of the nonlinear filter  equation}

One of the key tools throughout this work is an \emph{independent} analysis of the regularity properties of the associated nonlinear filter equation. For the reader’s convenience, we restate here this elliptic equation which, for any fixed $\alpha>0$ and $0<\beta<1$, is given by
\begin{equation}\label{Filter-Equation-Reg}
-\alpha^2 \text{div} \left( A (\varphi\ast\vu) (\vec{\nabla} \otimes \vu_\alpha)^T \right)+\vu_\alpha + \vec{\nabla}P_\alpha = \vu, \quad \text{div}(\vu_\alpha)=0, \quad \beta \leq A(\cdot)\leq 1,
\end{equation}
with periodic boundary conditions on $\T$. Note that, for any time $t>0$, the velocity $\vu(t,\cdot)\in L^2_{\text{s}}(\T)$ is regarded as a \emph{given datum}. We begin by recalling the following result.

\begin{proprevia}[Existence and uniqueness of $H^1$-solutions: Propositions $4.1$ and $4.2$ of \cite{CortezJarrin2026}]\label{Prop:Existence-Filter}
	Let $\alpha>0$ be fixed and let $0<T<+\infty$. Assume that $\vu \in L^{\infty}([0,T],L^2_{\text{s}}(\T))$, and that the function $A(\cdot)$ and the convolution term $\varphi$ satisfy (\ref{conditions-a-varphi-GWP}). Then, there exists a \emph{unique} weak solution
	\[
	(\vu_\alpha, P_\alpha) \in L^\infty([0,T],H^1(\T)) \times L^\infty([0,T],L^2(\T)),
	\]
	to equation (\ref{Filter-Equation-Reg}). Moreover, there exist two positive constants $K_0, K_1>0$, depending only on $\alpha$, $\beta$, $A(\cdot)$, and $\varphi$, such that for all $0\leq t \leq T$, the following estimate holds:
	\begin{equation*}
	\| \vu_{\alpha}(t,\cdot)\|^{2}_{H^1} \leq  K_0\| \vu(t,\cdot)\|^{2}_{L^2},
	\end{equation*}
	and, for any $0\leq t_1, t_2 \leq T$, the following estimate holds:
	\begin{equation}\label{Control-Filter-Diff-H1}
	\| \vu_{\alpha}(t_1,\cdot) - \vu_{\alpha}(t_2,\cdot)\|_{H^1} \leq K_1\Big(\| \vu(t_2,\cdot)\|_{L^2}+1\Big) \| \vu(t_1,\cdot)-\vu(t_2,\cdot)\|_{L^2}.
	\end{equation}
\end{proprevia}

In this context, we investigate higher regularity properties of these solutions. We emphasize that, in the particular case where $A(\cdot)\equiv 1$, such results are well-known for the linear Helmholtz filter equation
\[-\alpha^2 \Delta \vu_\alpha + \vu_\alpha = \vu, \quad {\rm div}(\vu_\alpha)=0,\]
in bounded smooth domains $\Omega\subset \Rt$ as well as in the whole space $\Rt$. See, for instance, \cite[Chapter IX]{Brezis} and \cite[Chapter 6.3]{Evans}.

\medskip

Furthermore, these results are also well established when the indicator function $A(\cdot)$ is \emph{independent} of the datum $\vu$, and one considers the equation
\[
-\alpha^2 {\rm div}(A(x)(\nabla \otimes \vu_\alpha)^T)+\vu_\alpha+\vec{\nabla}P_\alpha=\vu,  \quad {\rm div}(\vu_\alpha)=0.
\]
provided that suitable additional regularity assumptions are imposed on $A(\cdot)$. See, for instance, \cite[Chapter 9]{Gilbarg} and \cite[Chapter 8]{Ya}.

\medskip

In our setting, however, the presence of the datum $\vu$ in the coefficient $A(\varphi\ast \vu)$ means that the aforementioned results cannot be applied directly. Consequently, the higher-order regularity estimates must instead be derived through a tailored argument.

\medskip

The $H^2$-regularity properties of solutions to equation (\ref{Filter-Equation-Reg}) were already investigated in our previous work \cite{CortezJarrin2026}. For the reader’s convenience, we recall this result below.

\begin{proprevia}[$H^2$-regularity: Propositions $4.3$ and $4.4$ of \cite{CortezJarrin2026}]\label{Prop:H2-Filter} 
	Within the framework of Proposition \ref{Prop:Existence-Filter}, the solution $\vu_\alpha \in L^\infty([0,T],H^1(\T))$ to equation (\ref{Filter-Equation-Reg}) further satisfies
	\[  \vu_\alpha \in L^\infty([0,T], H^2(\T)),\]
	and there exists a constant $K_2>0$, depending only on $\alpha$, $\beta$, $A(\cdot)$, and $\varphi$, such that for any $0<t\leq T$, the following estimate holds:
	\begin{equation}\label{Control-H2-Filter}
	\| \vu_\alpha(t,\cdot)\|_{H^2} \leq K_2 \big(\| \vu(t,\cdot)\|^2_{L^2}+\|\vu(t,\cdot)\|_{L^2}\big). 
	\end{equation} 
	
Additionally,  for $i=1,2$, let $\vu_{i,\alpha} \in L^\infty([0,T],H^2(\T)) $ be two solutions of equation (\ref{Filter-Equation-Reg}) arising from $\vu_i \in L^\infty([0,T],L^2_{\text{s}}(\T))$. Then, there exists a constant $K'_2>0$ (also only depending  on $\alpha$, $\beta$, $A(\cdot)$, and $\varphi$)	such that for any $0<t\leq T$ we have
\begin{equation}\label{Control-Filter-Diff-H2}
\| \vu_{1,\alpha}(t,\cdot)-\vu_{2,\alpha}(t,\cdot)\|_{H^2}\leq K'_2\,\Big(\| \vu_1(t,\cdot) \|_{L^2}+\| \vu_2(t,\cdot) \|_{L^2}+1\Big)^2  \| \vu_1(t,\cdot)-\vu_2(t,\cdot) \|_{L^2}. 
\end{equation}

\end{proprevia}

In the following technical theorems, we establish the higher-order regularity properties that, as mentioned above, constitute the core of this work.   

\begin{Theorem}[$H^m$-regularity]\label{Th:Hm-Filter} 
	Within the framework of Proposition \ref{Prop:Existence-Filter}, assume in addition that the indicator function $A(\cdot)$ and the convolution term $\varphi$ satisfy (\ref{conditions-a-varphi-regularity}) for some $m \in \mathbb{N}$: 
	\[ 	A(\cdot)\in \mathcal{C}^{m+1,1}(\Rt)  \quad \mbox{and} \quad \varphi \in H^{m+1}(\T).   \]
	Define
	\begin{equation}\label{Norm-D-m+1-A}
	\| D^{m+1} A(\cdot)\|_{L^\infty}:=\max_{1\leq k \leq m+1}\, \max_{|{\bf {a}}|\leq k} \| \partial^{\bf a} A(\cdot)\|_{L^\infty}. 
	\end{equation}
	Moreover, for fixed $0< T_1<T_2\leq T$, assume that  
	\[ \vu \in L^\infty([T_1,T_2],  H^m(\T)). \]
	
	\medskip 
	
	Then, the solution $\vu_\alpha \in L^\infty([0,T], H^1(\T))$ to equation (\ref{Filter-Equation-Reg}), obtained in Proposition \ref{Prop:Existence-Filter}, satisfies  
	\[ \vu_\alpha \in L^\infty([T_1,T_2], H^{m+2}(\T)),  \]
	and there exists a constant $K_{m+2}>0$, depending only on $\alpha$, $\beta$, $\| D^{m+1} A(\cdot)\|_{L^\infty}$, and $\| \varphi \|_{H^{m+1}}$, such that for any $T_1\leq t\leq T_2$, the following estimate holds:
	\begin{equation}\label{Control-Hm-Filter} 
	\| \vu_{\alpha}(t,\cdot)\|_{H^{m+2}}\leq   K_{m+2}\,  \sum_{k=1}^{m+2}\| \vu(t,\cdot)\|^k_{H^m}.
	\end{equation}
\end{Theorem}

\begin{Remark}\label{Rmk:Reg-Filter-Constraint} We emphasize that the gain in regularity of $\vu_\alpha$ is strongly governed by the assumed regularity of the functions $A(\cdot)$ and $\varphi$. Specifically, assuming regularity of order $m+1$ for these two functions yields regularity of order $m+2$ for $\vu_\alpha$. 
\end{Remark}	
\begin{proof} 
	We apply a bootstrapping process with respect to the parameter $m$. 
	
	\medskip 
	
\emph{The initial case when $m=0$.} This case is directly given by Proposition \ref{Prop:H2-Filter}, where, from $\vu \in L^\infty_t L^2_x$ and $\vu_\alpha \in L^\infty_t H^1_x$, we obtain that $\vu_\alpha \in L^\infty_t H^2_x$. Moreover, from (\ref{Control-H2-Filter}), the estimate (\ref{Control-Hm-Filter}) holds.
	
	\medskip 
	
	\emph{The iterative process for $m\geq 1$.} Assume that $\vu \in L^\infty([T_1,T_2], H^m(\T))$, $A(\cdot)\in \mathcal{C}^{m+1,1}(\Rt)$, and $\varphi \in H^{m+1}(\T)$. Moreover, we assume that $\vu_\alpha \in L^\infty ([T_1,T_2], H^{m+1}(\T))$, and that 
	\begin{equation}\label{Control-Hm-1-Filter}
	\| \vu_\alpha(t,\cdot)\|_{H^{m+1}} \leq K_{m+1}\, \sum_{k=1}^{m+1} \| \vu(t,\cdot)\|^k_{H^{m-1}}. 
	\end{equation}
	We shall prove the gain of regularity $\vu_\alpha \in L^\infty ([T_1,T_2], H^{m+2}(\T))$ and that the estimate (\ref{Control-Hm-Filter}) holds.
	
	\medskip 
	
Applying the Leray projector $\P$ to equation (\ref{Filter-Equation-Reg}), we obtain 
\[
-\alpha^2 \P\, \mathrm{div}\Big( A(\varphi \ast \vu)(\vec{\nabla}\otimes \vu_\alpha)^T \Big)+\vu_\alpha=\vu. 
\]

After performing integration by parts in this equation, for fixed $T_1\leq t \leq T_2$ and for all $\vec{\phi} \in H^1(\T)$ (recall that $\P$ is a self-adjoint operator), we arrive at the following variational formulation:
\begin{equation*}
\begin{split}
&\alpha^2 \int_{\T} A((\varphi\ast \vu)(t,x)) (\vec{\nabla} \otimes \vu_\alpha)^T (t,x): (\vec{\nabla} \otimes \P\vec{\phi})^T (x) dx + \int_{\T} \vu_\alpha(t,x) \cdot \vec{\phi} (x) dx \\
=& \int_{\T} \vu(t,x)\cdot \vec{\phi}(x) dx, 
\end{split}
\end{equation*}
where, for two matrices $A$ and $B$, we denote $A:B= \mathrm{tr}(A^T B)$. For simplicity, we omit the time dependence of the functions. 
	
	\medskip 
	
	For a multi-index $|{\bf a}|=m$ and $\vw \in \mathcal{P}$, where $\mathcal{P}$ denotes the set of periodic vector-valued trigonometric polynomials on $\T$, we set the test function $\vec{\phi} := \partial^{\bf a} \vw$. Then, in the identity above, we obtain
	\begin{equation*}
	\alpha^2 \int_{\T} A(\varphi\ast \vu) (\vec{\nabla} \otimes \vu_\alpha)^T : (\vec{\nabla} \otimes \partial^{{\bf a}} \P \vw)^T dx + \int_{\T} \vu_\alpha \cdot \partial^{{\bf a}} \vw dx = \int_{\T} \vu \cdot \partial^{{\bf a}} \vw dx.
	\end{equation*}
	
To simplify the notation, we write $\P\vw := \vec{\omega}=(\omega_i)_{1\leq i \leq 3}$. Hence, we have 
\begin{equation*}
\alpha^2 \sum_{i,j=1}^{3} \int_{\T} A(\varphi\ast \vu) (\partial_j u_{\alpha,i})(\partial_j\partial^{{\bf a}} \omega_i) dx + \sum_{i=1}^{3}\int_{\T} u_{\alpha,i} \, \partial^{{\bf a}} w_i dx = \sum_{i=1}^{3}\int_{\T} u_i \, \partial^{{\bf a}} w_i dx.
\end{equation*}
Additionally, to make the presentation simpler, it suffices to consider the generic expression:
\begin{equation*}
\alpha^2 \int_{\T} A(\varphi\ast \vu) (\partial_j u_{\alpha,i})(\partial_j\partial^{{\bf a}} \omega_i) dx + \int_{\T} u_{\alpha,i} \, \partial^{{\bf a}} w_i dx = \int_{\T} u_i \, \partial^{{\bf a}} w_i dx.
\end{equation*}

In each term, we integrate by parts to transfer the derivative $\partial^{{\bf a}}$ to the left. Moreover, in the first term, we also use the Leibniz rule. Therefore, for constants $C_{{\bf a},{\bf b}}>0$ depending on the multi-index ${\bf a}$ and ${\bf b}$ (with $|{\bf b}| \leq |{\bf a}|$), we write
	\begin{equation*}
	\begin{split}
	&(-1)^{m}\alpha^2 \int_{\T} \sum_{|{\bf b}|\leq |{\bf a}|} C_{{\bf a},{\bf b}}\,\big(\partial^{{\bf b}} A (\varphi \ast \vu)\big)\big(\partial^{{\bf b}}\varphi \ast \vu\big) \big(\partial^{{\bf a}-{\bf b}} \partial_j u_{\alpha,i}\big)\, \partial_j \omega_i dx+ (-1)^m \int_{\T}\partial^{{\bf a}} u_{\alpha,i}\, w_i dx \\
	=& (-1)^m \int_{\T} \partial^{{\bf a}} u_{i}\, w_i dx.
	\end{split}
	\end{equation*}
	
Then, we integrate by parts once more to transfer the derivative $\partial_j$ to the left. We obtain
	\begin{equation*}
	\begin{split}
	&(-1)^{m+1}\alpha^2 \int_{\T} \sum_{|{\bf b}|\leq |{\bf a}|} C_{{\bf a},{\bf b}}\, \partial_j \Big( \big(\partial^{{\bf b}} A (\varphi \ast \vu)\big)\big(\partial^{{\bf b}}\varphi \ast \vu\big) \Big)\big(\partial^{{\bf a}-{\bf b}} \partial_j u_{\alpha,i}\big)\, \omega_i \, dx \\
	&\quad + (-1)^{m+1}\alpha^2 \int_{\T} \sum_{|{\bf b}|\leq |{\bf a}|} C_{{\bf a},{\bf b}}\,\big(\partial^{{\bf b}} A (\varphi \ast \vu)\big) \big(\partial^{{\bf b}}\varphi \ast \vu\big) \Big( \partial^{{\bf a}-{\bf b}} \partial^2_j u_{\alpha,i}\Big)\, \omega_i \, dx \\
	&\quad + (-1)^m \int_{\T}\partial^{{\bf a}} u_{\alpha,i}\, w_i dx 
	= (-1)^m \int_{\T} \partial^{{\bf a}} u_{i}\, w_i dx.
	\end{split}
	\end{equation*}
	
	Dividing each term by $(-1)^{m+1}$, we obtain
	\begin{equation*}
	\begin{split}
	&\alpha^2 \int_{\T} \sum_{|{\bf b}|\leq |{\bf a}|} C_{{\bf a},{\bf b}}\, \partial_j \Big( \big(\partial^{{\bf b}} A (\varphi \ast \vu)\big) (\partial^{{\bf b}}\varphi \ast \vu) \Big)\big(\partial^{{\bf a}-{\bf b}} \partial_j u_{\alpha,i}\big)\, \omega_i \,  dx \\
	&\quad + \alpha^2 \int_{\T} \sum_{|{\bf b}|\leq |{\bf a}|} C_{{\bf a},{\bf b}}\, \big(\partial^{{\bf b}} A (\varphi \ast \vu)\big)(\partial^{{\bf b}}\varphi \ast \vu) \Big( \partial^{{\bf a}-{\bf b}} \partial^2_j u_{\alpha,i}\Big)\, \omega_i \, dx \\
	&\quad - \int_{\T}\partial^{{\bf a}} u_{\alpha,i}\, w_i dx 
	= - \int_{\T} \partial^{{\bf a}} u_{i}\, w_i dx.
	\end{split}
	\end{equation*}
	
	Now, splitting the second term above, we write
	\begin{equation*}
	\begin{split}
	&\alpha^2 \int_{\T} \sum_{|{\bf b}|\leq |{\bf a}|} C_{{\bf a},{\bf b}}\, \big(\partial^{{\bf b}} A (\varphi \ast \vu)\big)(\partial^{{\bf b}}\varphi \ast \vu) \Big( \partial^{{\bf a}-{\bf b}} \partial^2_j u_{\alpha,i}\Big)\, \omega_i \,  dx= \alpha^2 \int_{\T} A (\varphi \ast \vu)\, \big(\partial^{{\bf a}} \partial^2_j u_{\alpha,i}\big)\, \omega_i \, dx \\
	&\quad + \alpha^2 \int_{\T} \sum_{1\leq|{\bf b}|\leq |{\bf a}|} C_{{\bf a},{\bf b}}\,\big( \partial^{{\bf b}} A (\varphi \ast \vu)\big)(\partial^{{\bf b}}\varphi \ast \vu) \Big( \partial^{{\bf a}-{\bf b}} \partial^2_j u_{\alpha,i}\Big)\, \omega_i \, dx,
	\end{split}
	\end{equation*}
and rearranging terms, we obtain
	\begin{equation*}
	\begin{split}
&\, 	\alpha^2 \int_{\T} A(\varphi \ast \vu) \big( \partial^{{\bf a}}\partial^2_j u_{\alpha,i}\big)\, \omega_i\,  dx =  -\alpha^2 \int_{\T} \sum_{1\leq |{\bf b}|\leq |{\bf a}|} C_{{\bf a},{\bf b}}\,\big(\partial^{{\bf b}} A (\varphi \ast \vu)\big)(\partial^{{\bf b}}\varphi \ast \vu)\big( \partial^{{\bf a}-{\bf b}} \partial^2_j u_{\alpha,i}\big)\, \omega_i\, dx \\
	& -\alpha^2 \int_{\T} \sum_{|{\bf b}|\leq |{\bf a}|} C_{{\bf a},{\bf b}}\, \partial_j \Big( \big(\partial^{{\bf b}} A (\varphi \ast \vu)\big)(\partial^{{\bf b}}\varphi \ast \vu) \Big)\big(\partial^{{\bf a}-{\bf b}} \partial_j u_{\alpha,i}\big)\, \omega_i \, dx + \int_{\T}\partial^{{\bf a}} u_{\alpha,i}\, w_i dx -\int_{\T} \partial^{{\bf a}} u_{i}\, w_i dx.
	\end{split}
	\end{equation*}
	
In each term on the right-hand side, we apply the Cauchy--Schwarz and Young inequalities to obtain the following estimates.  

\medskip

For the first term, as $|{\bf b}| \leq |{\bf a}| \leq m$ and $|{\bf a} - {\bf b}| \leq m-1$ (since $1\leq |{\bf b}|$), we obtain 
\begin{equation*}
\begin{split}
&\, \alpha^2 \int_{\T} \sum_{1\leq |{\bf b}|\leq |{\bf a}|} C_{{\bf a},{\bf b}}\,\left| \big(\partial^{{\bf b}} A (\varphi \ast \vu)\big)(\partial^{{\bf b}}\varphi \ast \vu)\big( \partial^{{\bf a}-{\bf b}} \partial^2_j u_{\alpha,i}\big)\, \omega_i \right|\, dx\\
\leq &\, \alpha^2\, \| D^{m}A(\cdot)\|_{L^\infty}\, \| \partial^{{\bf b}} \varphi \ast \vu \|_{L^\infty}\, \| \partial^{{\bf a}-{\bf b}} \partial^2_j u_{\alpha,i}\|_{L^2}\, \| \omega_i \|_{L^2} \\
\leq &\, \alpha^2\, \| D^{m}A(\cdot)\|_{L^\infty}\, \| \partial^{{\bf b}} \varphi  \|_{L^2}\, \| \vu \|_{L^2}\, \| \vu_\alpha \|_{H^{m+1}}\, \| \omega_i \|_{L^2}\\
\leq &\, \alpha^2\, \| D^{m}A(\cdot)\|_{L^\infty}\, \|  \varphi  \|_{H^m}\, \| \vu \|_{H^m}\, \| \vu_\alpha \|_{H^{m+1}}\, \| \vec{\omega} \|_{L^2}.
\end{split}
\end{equation*}

For the second term, using the same arguments, we obtain
\begin{equation*}
\begin{split}
&\, \alpha^2 \int_{\T} \sum_{|{\bf b}|\leq |{\bf a}|} C_{{\bf a},{\bf b}}\, \left|\partial_j \Big( \big(\partial^{{\bf b}} A (\varphi \ast \vu)\big)(\partial^{{\bf b}}\varphi \ast \vu) \Big)\big(\partial^{{\bf a}-{\bf b}} \partial_j u_{\alpha,i}\big)\, \omega_i \right| \, dx\\
\leq &\, \alpha^2\, \| D^{m+1}A(\cdot)\|_{L^\infty}\, \|  \varphi  \|_{H^{m+1}}\, \| \vu \|_{H^m}\, \| \vu_\alpha \|_{H^{m+1}}\, \| \vec{\omega} \|_{L^2}.
\end{split}
\end{equation*}

The third and fourth terms are straightforward to estimate. 

\medskip

Gathering these estimates, we have

	\begin{equation*}
	\begin{split}
	\alpha^2 \left| \int_{\T} A(\varphi \ast \vu) \partial^{{\bf a}}\partial^2_j u_{\alpha,i}\, w_i dx \right| \leq & \alpha^2\, \| D^{m+1}A(\cdot)\|_{L^\infty}\, \|\varphi \|_{H^{m+1}}\, \| \vu \|_{H^{m}}\| \vu_\alpha \|_{H^{m+1}}\, \| \vec{\omega} \|_{L^2} \\
	& + \| \vu_\alpha \|_{H^{m+1}} \| \vw \|_{L^2}+ \| \vu \|_{H^{m}} \| \vw \|_{L^2}.
	\end{split}
	\end{equation*}
	
		Since $\vec{\omega}:=\P \vw$, using standard properties of $\P$ and defining
	\[ L_{m+2}:= \max\left(  \alpha^2\, \| D^{m+1}A(\cdot)\|_{L^\infty}\, \|\varphi \|_{H^{m+1}}, 1 \right), \]
	we obtain
	\begin{equation*}
	\alpha^2 \left| \int_{\T} A(\varphi \ast \vu) \partial^{{\bf a}}\partial^2_j u_{\alpha,i}\, w_i \, dx \right| 
	\leq L_{m+2} \Big( \| \vu_\alpha\|_{H^{m+1}}\, \| \vu \|_{H^m}+  \| \vu_\alpha\|_{H^{m+1}}+ \| \vu \|_{H^m}\Big)\| \vw \|_{L^2}. 
	\end{equation*}
	
	By density of $\mathcal{P}$ in $L^2(\T)$, this estimate extends to all $\vw \in L^2(\T)$, yielding
	\[ \|A(\varphi \ast \vu) \,\partial^{\bf a} \partial^2_j \vu_\alpha\|_{L^2} \leq \frac{L_{m+2}}{\alpha^2} \Big( \| \vu_\alpha\|_{H^{m+1}}\, \| \vu \|_{H^m}+  \| \vu_\alpha\|_{H^{m+1}}+ \| \vu \|_{H^m}\Big).\]
	
	Using the lower bound in (\ref{Conditions-a-existence-1}), we also have
	\[ \beta \| \partial^{\bf a} \partial^2_j  \vu_\alpha \|_{L^2}\leq \|A(\varphi \ast \vu) \,\partial^{\bf a} \partial^2_j \vu_\alpha\|_{L^2},\]
	and therefore
	\begin{equation*}
	\begin{split}
	\| \partial^{\bf a} \partial^2_j  \vu_\alpha \|_{L^2} 
	\leq &\,  \frac{L_{m+2}}{\alpha^2\,\beta} \Big( \| \vu_\alpha\|_{H^{m+1}}\, \| \vu \|_{H^m}+  \| \vu_\alpha\|_{H^{m+1}}+ \| \vu \|_{H^m}\Big)\\
	=:&\, K_{m+2} \Big( \| \vu_\alpha\|_{H^{m+1}}\, \| \vu \|_{H^m}+  \| \vu_\alpha\|_{H^{m+1}}+ \| \vu \|_{H^m}\Big).
	\end{split}
	\end{equation*}
	As the multi-index $|{\bf a|}=m$ was arbitrary, we finally obtain that 
	\begin{equation}\label{Control-Hm-Filter-prevous}
	\| \vu_\alpha(t,\cdot)\|_{H^{m+2}} \leq K_{m+2} \Big( \| \vu_\alpha\|_{H^{m+1}}\, \| \vu \|_{H^m}+  \| \vu_\alpha\|_{H^{m+1}}+ \| \vu \|_{H^m}\Big).
	\end{equation}
	
	Since $\vu \in L^\infty([T_1, T_2], H^m(\T))$ and $\vu_\alpha \in L^\infty([T_1, T_2], H^{m+1}(\T))$, we conclude that 
	\[ \vu_\alpha \in L^\infty([T_1, T_2], H^{m+2}(\T)). \]
	Moreover, the wished estimate (\ref{Control-Hm-Filter}) follows directly from the estimates (\ref{Control-Hm-1-Filter}) and  (\ref{Control-Hm-Filter-prevous}).  This completes the proof of Theorem \ref{Th:Hm-Filter}.  
\end{proof}

\begin{Theorem}[$H^m$-continuous dependence in time]\label{Th:Cont-Dep-Hm-Filter} 
	Under the same hypothesis of Theorem \ref{Th:Hm-Filter}, denote 
	\[ \sup_{T_1\leq t \leq T_2}\| \vu(t,\cdot)\|_{H^m}:= \| \vu \|_{L^\infty_t H^m_x}. \]
	
	Then, there exists a quantity $K'_{m+2}>0$, which depends on $\| \vu \|_{L^\infty_t H^m_x}$ together with $\alpha$, $\beta$, $\| D^{m+1} A(\cdot)\|_{L^\infty}$, and $\| \varphi \|_{H^{m+1}}$, such that for any $T_1\leq t_1, t_2 \leq T_2$, the following estimate holds:
	\begin{equation}\label{Control-Filter-Diff-Hm}
	\| \vu_{\alpha}(t_1,\cdot)-\vu_{\alpha}(t_2,\cdot)\|_{H^{m+2}}\leq K'_{m+2}\,  \| \vu(t_1,\cdot)-\vu(t_2,\cdot) \|_{H^m}.
	\end{equation}  
\end{Theorem}
\begin{proof} It follows from the following recursive estimate.
\begin{Proposition}\label{Prop-Tech:Recursive-Estim-Filter-Diff-Hm}  
	Let $m\in \mathbb{N}$. There exists a quantity $L'_{m+2}>0$, which depends on $\alpha$, $\beta$, $\| D^{m+1} A(\cdot)\|_{L^\infty}$, and $\| \varphi \|_{H^{m+1}}$, such that for any $T_1\leq t_1, t_2 \leq T_2$, the following estimate holds:
	\begin{equation}\label{Recursive-Control-Diff-Filter-Hm}
	\begin{split}
	&\,\|  \vu_{\alpha}(t_1,\cdot)-\vu_{\alpha}(t_2,\cdot) \|_{H^{m+2}} \\
	\leq  & \,  L'_{m+2}\, \| \vu(t_1,\cdot) \|_{L^2}\, \| \vu_{\alpha}(t_1,\cdot)-\vu_{\alpha}(t_2,\cdot)  \|_{H^{m+1}}\ \\
	&\, +  L'_{m+2} \Big(\,\Big(\| \vu(t_1,\cdot) \|_{L^2}+1\Big)  \| \vu_{\alpha}(t_2,\cdot)\|_{H^{m+2}}  + 1\Big)\,  \| \vu(t_1,\cdot)-\vu(t_2,\cdot) \|_{H^m}.
	\end{split}
	\end{equation}
\end{Proposition}
\begin{proof} For simplicity, denote 
	\[ \vu(t_1, \cdot):= \vu_1, \quad \vu(t_2,\cdot):= \vu_2, \quad \vu_{\alpha}(t_1,\cdot):=\vu_{1,\alpha} \quad \text{and}\quad \vu_{\alpha}(t_2,\cdot):=\vu_{2,\alpha},    \]
	and define
	\[ \vw:=\vu_1-\vu_2, \qquad \vw_\alpha := \vu_{1,\alpha}- \vu_{2,\alpha}. \]
	Note that $\vw_\alpha$ satisfies the equation  
	\begin{equation*}\label{PDE-Filter-difference}
	-\alpha^2 \P\, \text{div}\Big(A(\varphi \ast \vu_1) (\vec{\nabla}\otimes \vw_\alpha  )^T +\big(A(\varphi\ast \vu_1)-A(\varphi \ast \vu_2)\big)(\vec{\nabla}\otimes \vu_{2,\alpha})^T\Big) + \vw_\alpha= \vw.
	\end{equation*}
For any  multi-index $|{\bf a }|=m$, applying the derivative $\partial^{{\bf a}}$ to each term above and using the Leibniz rule in the first term on the left-hand side, we write
	\begin{equation*}
	\begin{split}\vspace{2mm}
	&-\alpha^2 \P \, \text{div} \left( A(\varphi \ast \vu_1)(\vec{\nabla} \otimes \partial^{{\bf a}}\vw_\alpha)^T  \right) -\alpha^2   \, \sum_{1\leq |{\bf b}| \leq |{\bf a}|} C_{{\bf a}, {\bf b}} \left( \P\, \text{div}\left(  \partial^{{\bf b}} \big(A(\varphi \ast \vu_1)\big)  (\vec{\nabla}\otimes \partial^{{\bf a}-{\bf b}} \vw_\alpha  )^T \right)\right) \\
	&-\alpha^2  \, \sum_{|{\bf b}| \leq |{\bf a}|} C_{{\bf a}, {\bf b}} \left( \P\, \text{div}\left(   \partial^{{\bf b}}\Big(A(\varphi\ast \vu_1)-A(\varphi \ast \vu_2)\Big)(\vec{\nabla}\otimes \partial^{{\bf a}-{\bf b}} \vu_{2,\alpha})^T\right)  \right)+ \partial^{{\bf a}} \vw_\alpha=  \partial^{{\bf a}}\vw.
 \end{split}
\end{equation*}

As $\vu_{1,\alpha}, \vu_{2,\alpha}, \vw_\alpha \in H^{m+2}(\T)$ and $\vu_1, \vu_2, \vw \in H^{m}(\T)$, note that each term belongs to $L^2(\T)$. Then, for any $k=1,2,3$, we can multiply each term by $\partial^2_k \partial^{\bf{a}}\vw_\alpha \in L^2(\T)$ and integrate by parts. Here, we first move the Leray projector $\P$ and the divergence operator to the right, and then we move $\partial_k$ to the left. By the divergence-free property of $\vw_\alpha$, we thus find:
\begin{equation*}
\begin{split}
&\,-\alpha^2 \int_{\T} \partial_k \left(A(\varphi \ast \vu_1)(\vec{\nabla} \otimes \partial^{{\bf a}}  \vw_\alpha)^T\right): (\vec{\nabla} \otimes \partial_k \partial^{\bf{a}} \vw_\alpha)^T \, dx \\
&\, -\alpha^2\,  \sum_{1\leq |{\bf b}| \leq |{\bf a}|}  C_{{\bf a}, {\bf b}} \int_{\T} \partial_k \left(   \partial^{{\bf b}} \big(A(\varphi \ast \vu_1)\big)  (\vec{\nabla}\otimes \partial^{{\bf a}-{\bf b}} \vw_\alpha  )^T\right) :  (\vec{\nabla} \otimes \partial_k \partial^{\bf{a}} \vw_{\alpha})^T \, dx\\
&\, -\alpha^2 \,   \sum_{|{\bf b}| \leq |{\bf a}|}  C_{{\bf a}, {\bf b}} \int_{\T}  \partial_k \left(   \partial^{{\bf b}}\Big(A(\varphi\ast \vu_1)-A(\varphi \ast \vu_2)\Big)(\vec{\nabla}\otimes \partial^{{\bf a}-{\bf b}} \vu_{2,\alpha})^T\right) :   (\vec{\nabla} \otimes \partial_k \partial^{\bf{a}} \vw_{\alpha})^T \, dx\\ 
&\,- \| \partial_k \partial^{\bf{a}} \vw_\alpha\|^2_{L^2}= \int_{\T} \vw \cdot \partial^2_k  \partial^{\bf{a}} \vw_\alpha\, dx.
\end{split}
\end{equation*}	 

We multiply by $-1$ each term. Additionally, in the first term on the left-hand side, we develop the derivative of the product, to obtain

\begin{equation*}
\begin{split}
&\, \alpha^2 \int_{\T}  A(\varphi \ast \vu_1) | \vec{\nabla} \otimes \partial_k \partial^{{\bf a}} \vw_\alpha |^2\, dx \\
+&\,\alpha^2 \int_{\T} \partial_k \left(A(\varphi \ast \vu_1)\right)(\vec{\nabla} \otimes \partial^{{\bf a}}  \vw_\alpha): (\vec{\nabla} \otimes \partial_k \partial^{\bf{a}} \vw_\alpha)^T \, dx \\
+&\, \alpha^2\,  \sum_{1\leq |{\bf b}| \leq |{\bf a}|}  C_{{\bf a}, {\bf b}} \int_{\T} \partial_k \left(   \partial^{{\bf b}} \big(A(\varphi \ast \vu_1)\big)  (\vec{\nabla}\otimes \partial^{{\bf a}-{\bf b}} \vw_\alpha  )^T\right) :  (\vec{\nabla} \otimes \partial_k \partial^{\bf{a}} \vw_{\alpha})^T \, dx\\
+&\, \alpha^2 \,   \sum_{|{\bf b}| \leq |{\bf a}|}  C_{{\bf a}, {\bf b}} \int_{\T}  \partial_k \left(   \partial^{{\bf b}}\Big(A(\varphi\ast \vu_1)-A(\varphi \ast \vu_2)\Big)(\vec{\nabla}\otimes \partial^{{\bf a}-{\bf b}} \vu_{2,\alpha})^T\right) :   (\vec{\nabla} \otimes \partial_k \partial^{\bf{a}} \vw_{\alpha})^T \, dx\\ 
+&\,\| \partial_k \partial^{\bf{a}} \vw_\alpha\|^2_{L^2}= -\int_{\T} \vw \cdot \partial^2_k  \partial^{\bf{a}} \vw_\alpha\, dx.
\end{split}
\end{equation*}

According to the lower bound in (\ref{Conditions-a-existence-1}), the first term above is estimated from below as:
\[ \alpha^2 \beta \| \vec{\nabla}\otimes \partial_k \partial^{{\bf a}} \vw_\alpha \|^{2}_{L^2}\leq  \alpha^2 \int_{\T}  A(\varphi \ast \vu_1) | \vec{\nabla} \otimes \partial_k  \partial^{{\bf a}} \vw_\alpha |^2 \, dx.\]
Moreover, setting the  constant $C_{\alpha,\beta}=\min(\alpha^2\beta,1)>0$,  we write
\[ C_{\alpha,\beta} \| \partial_{k} \partial^{{\bf a}}  \vw_\alpha \|^{2}_{H^1} \leq \alpha^2 \beta  \| \vec{\nabla}\otimes \partial_k \partial^{{\bf a}} \vw_\alpha \|^{2}_{L^2} +  \| \partial_k \partial^{{\bf a}}  \vw_\alpha \|^{2}_{L^2}.\] 

Thus, rearranging terms, we have
\begin{equation}\label{Estim-Tech-Filter-0}
\begin{split}
&\, C_{\alpha,\beta} \| \partial_{k} \partial^{{\bf a}}  \vw_\alpha \|^{2}_{H^1} \\
\leq  &\, \alpha^2 \int_{\T} \left| \partial_k \left(A(\varphi \ast \vu_1)\right)(\vec{\nabla} \otimes \partial^{{\bf a}}  \vw_\alpha): (\vec{\nabla} \otimes \partial_k \partial^{\bf{a}} \vw_\alpha)^T \right| \, dx \\
&\, + \alpha^2\,  \sum_{1\leq |{\bf b}| \leq |{\bf a}|}  C_{{\bf a}, {\bf b}} \int_{\T} \left| \partial_k \left(   \partial^{{\bf b}} \big(A(\varphi \ast \vu_1)\big)  (\vec{\nabla}\otimes \partial^{{\bf a}-{\bf b}} \vw_\alpha  )^T\right) :  (\vec{\nabla} \otimes \partial_k \partial^{\bf{a}} \vw_{\alpha})^T \right| \, dx\\
&\,+ \alpha^2 \,   \sum_{|{\bf b}| \leq |{\bf a}|}  C_{{\bf a}, {\bf b}} \int_{\T} \left| \partial_k \left(   \partial^{{\bf b}}\Big(A(\varphi\ast \vu_1)-A(\varphi \ast \vu_2)\Big)(\vec{\nabla}\otimes \partial^{{\bf a}-{\bf b}} \vu_{2,\alpha})^T\right) :   (\vec{\nabla} \otimes \partial_k \partial^{\bf{a}} \vw_{\alpha})^T \right| \, dx\\ 
&\,+ \int_{\T} |\vw \cdot \partial^2_k  \partial^{\bf{a}} \vw_\alpha|\, dx.
\end{split}
\end{equation}

We need to estimate each term on the right-hand side. For clarity, we will prove each estimate in the following technical lemmas. Additionally, in order to simplify our writing, we will denote by $L'_{m+2}>0$ a generic constant, which may change from one line to the next and depends only on $\alpha$, $\beta$, $\| D^{m+1} A(\cdot)\|_{L^\infty}$, and $\| \varphi \|_{H^{m+1}}$.

\begin{Lemma} For the first and the second term, we have
	\begin{equation}\label{Estim-Tech-Filter-1-2}
	\begin{split}
	&\, \alpha^2 \int_{\T} \left| \partial_k \left(A(\varphi \ast \vu_1)\right)(\vec{\nabla} \otimes \partial^{{\bf a}}  \vw_\alpha): (\vec{\nabla} \otimes \partial_k \partial^{\bf{a}} \vw_\alpha)^T \right| \, dx  \\
+	&\, \alpha^2\,  \sum_{1\leq |{\bf b}| \leq |{\bf a}|}  C_{{\bf a}, {\bf b}} \int_{\T} \left| \partial_k \left(   \partial^{{\bf b}} \big(A(\varphi \ast \vu_1)\big)  (\vec{\nabla}\otimes \partial^{{\bf a}-{\bf b}} \vw_\alpha  )^T\right) :  (\vec{\nabla} \otimes \partial_k \partial^{\bf{a}} \vw_{\alpha})^T \right| \, dx\\
	\leq &\,  L'_{m+2}\, \| \vu_1 \|_{L^2}\, \| \vw_\alpha \|_{H^{m+1}}\,  \| \partial_k \partial^{{\bf a}} \vw_\alpha \|_{H^1}.
	\end{split}
	\end{equation}
\end{Lemma}
\begin{proof} 
For the first term, using the Cauchy--Schwarz, H\"older, and Young inequalities, and as $|{\bf a}|=m$, we write
	\begin{equation*}
	\begin{split}
		&\, \alpha^2 \int_{\T} \left| \partial_k \left(A(\varphi \ast \vu_1)\right)(\vec{\nabla} \otimes \partial^{{\bf a}}  \vw_\alpha): (\vec{\nabla} \otimes \partial_k \partial^{\bf{a}} \vw_\alpha)^T \right| \, dx \\
		\leq &\, \alpha^2\, \| \partial_k \left(A(\varphi \ast \vu_1)\right)\|_{L^\infty}\, \| \partial^{{\bf a}} \vw_\alpha \|_{H^1}\, \| \partial_k \partial^{{\bf a}} \vw_\alpha \|_{H^1}\\
		\leq &\, \alpha^2 \| \partial_ k A(\varphi\ast \vu_1) (\partial_k \varphi \ast \vu_1)\|_{L^\infty}\, \| \vw_\alpha \|_{H^{m+1}}\, \| \partial_k \partial^{{\bf a}} \vw_\alpha \|_{H^1}\\
		\leq &\,  \alpha^2  \| \vec{\nabla} A(\cdot)\|_{L^\infty}\, \| \varphi \|_{H^1}\, \| \vu_1 \|_{L^2}\, \| \vw_\alpha \|_{H^{m+1}}\,  \| \partial_k \partial^{{\bf a}} \vw_\alpha \|_{H^1}\\
		\leq &\, L'_{m+2}\, \| \vu_1 \|_{L^2}\, \| \vw_\alpha \|_{H^{m+1}}\,  \| \partial_k \partial^{{\bf a}} \vw_\alpha \|_{H^1}. 
	\end{split}
	\end{equation*}

The estimate for the second term follows from the same arguments, and we also have
\begin{equation*}
	\begin{split}
	&\, \alpha^2\,  \sum_{1\leq |{\bf b}| \leq |{\bf a}|}  C_{{\bf a}, {\bf b}} \int_{\T} \left| \partial_k \left(   \partial^{{\bf b}} \big(A(\varphi \ast \vu_1)\big)  (\vec{\nabla}\otimes \partial^{{\bf a}-{\bf b}} \vw_\alpha  )^T\right) :  (\vec{\nabla} \otimes \partial_k \partial^{\bf{a}} \vw_{\alpha})^T \right| \, dx\\
	\leq &\, L'_{m+2}\, \| \vu_1 \|_{L^2}\, \| \vw_\alpha \|_{H^{m+1}}\,  \| \partial_k \partial^{{\bf a}} \vw_\alpha \|_{H^1}.
	\end{split}
	\end{equation*}
\end{proof}	

\begin{Lemma} For the third term, we have
	\begin{equation}\label{Estim-Tech-Filter-3}
	\begin{split}
	 &\,\alpha^2 \,   \sum_{|{\bf b}| \leq |{\bf a}|}  C_{{\bf a}, {\bf b}} \int_{\T} \left| \partial_k \left(   \partial^{{\bf b}}\big(A(\varphi\ast \vu_1)-A(\varphi \ast \vu_2)\big)(\vec{\nabla}\otimes \partial^{{\bf a}-{\bf b}} \vu_{2,\alpha})^T\right) :   (\vec{\nabla} \otimes \partial_k \partial^{\bf{a}} \vw_{\alpha})^T \right| \, dx \\
\leq &\,  L'_{m+2} \Big(\| \vu_1 \|_{L^2}+1\Big) \| \vw \|_{L^2}\, \| \vu_{2,\alpha}\|_{H^{m+2}}\, \| \partial_k \partial^{{\bf a}} \vw_\alpha \|_{H^1}.
	\end{split}
	\end{equation}
\end{Lemma}	
\begin{proof} For simplicity, it suffices to consider the generic term
	\[ I:=\int_{\T} \left| \partial_k \left(   \partial^{{\bf b}}\big(A(\varphi\ast \vu_1)-A(\varphi \ast \vu_2)\big)(\vec{\nabla}\otimes \partial^{{\bf a}-{\bf b}} \vu_{2,\alpha})^T\right) :   (\vec{\nabla} \otimes \partial_k \partial^{\bf{a}} \vw_{\alpha})^T \right| \, dx. \]
We write
	\begin{equation*}
	\begin{split}
I\leq &\,\int_{\T} \left| \partial_k \partial^{{\bf b}}    \Big(A(\varphi\ast \vu_1)-A(\varphi \ast \vu_2)\Big)(\vec{\nabla}\otimes \partial^{{\bf a}-{\bf b}} \vu_{2,\alpha})^T :   (\vec{\nabla} \otimes \partial_k \partial^{\bf{a}} \vw_{\alpha})^T \right| \, dx \\
&\, + \int_{\T} \left|  \partial^{{\bf b}}   \Big(A(\varphi\ast \vu_1)-A(\varphi \ast \vu_2)\Big)(\vec{\nabla}\otimes \partial^{{\bf a}-{\bf b}} \partial_k  \vu_{2,\alpha})^T :   (\vec{\nabla} \otimes \partial_k \partial^{\bf{a}} \vw_{\alpha})^T \right| \, dx =: I_1+ I_2,
	\end{split}
	\end{equation*}
where we must estimate each term separately. 

\medskip

For $I_1$, we carry out the following pointwise computations
\begin{equation*}
\begin{split}\vspace{2mm}
&\, \left|\partial_k \partial^{{\bf b}}\Big(A(\varphi\ast \vu_1)-A(\varphi \ast \vu_2)\Big) \right| \\ \vspace{2mm}
=&\, \Big| \partial_k \partial^{{\bf b}} A (\varphi\ast \vu_1)\big((\partial_k \partial^{{\bf b}} \varphi)\ast \vu_1\big) -  \partial_k\partial^{{\bf b}} A (\varphi\ast \vu_2)\big((\partial_k \partial^{{\bf b}}\varphi)\ast \vu_2\big)\Big|\\ \vspace{2mm}
=&\, \left| \Big( \partial_k\partial^{{\bf b}} A (\varphi\ast \vu_1) - \partial_k\partial^{{\bf b}} A (\varphi\ast \vu_2) \Big)\big( (\partial_k\partial^{{\bf b}} \varphi)\ast \vu_1\big) + \partial_k\partial^{{\bf b}} A (\varphi \ast \vu_2) (\partial_k\partial^{{\bf b}} \varphi)\ast (\vu_1-\vu_2)\right| \\
\leq &\, \Big|  \partial_k\partial^{{\bf b}} A (\varphi\ast \vu_1) - \partial_k\partial^{{\bf b}} A (\varphi\ast \vu_2)\Big|\, \Big| (\partial_k\partial^{{\bf b}} \varphi)\ast \vu_1\Big| + \Big|\partial_k\partial^{{\bf b}} A (\varphi \ast \vu_2) \Big|\, \Big|(\partial_k\partial^{{\bf b}} \varphi)\ast (\vu_1-\vu_2)\Big|:=(a).
\end{split}
\end{equation*}		

To estimate the term $\Big|  \partial_k\partial^{{\bf b}} A (\varphi\ast \vu_1) - \partial_k\partial^{{\bf b}} A (\varphi\ast \vu_2)\Big|$, note that, from the assumption (\ref{conditions-a-varphi-regularity}) and the fact that $|{\bf b}|\leq m$, the function $\partial_k \partial^{{\bf b}}A(\cdot)$ is Lipschitz continuous, and there exists a constant $C_{A}>0$ such that 
\[ \Big|  \partial_k\partial^{{\bf b}} A (\varphi\ast \vu_1) - \partial_k\partial^{{\bf b}} A (\varphi\ast \vu_2)\Big| \leq C_A  C_{A} | \varphi \ast (\vu_1-\vu_2)|. \]

Therefore, we write
\begin{equation*}
\begin{split}
(a)\leq &\, C_{A} \Big| \varphi \ast (\vu_1-\vu_2)\Big|\, \Big| (\partial_k\partial^{{\bf b}} \varphi)\ast \vu_1\Big|+\Big|\partial_k\partial^{{\bf b}} A (\varphi \ast \vu_2) \Big|\, \Big|(\partial_k\partial^{{\bf b}} \varphi)\ast (\vu_1-\vu_2)\Big|\\
=&\, C_{A}  \Big| \varphi \ast \vw\Big|\, \Big| (\partial_k\partial^{{\bf b}} \varphi)\ast \vu_1\Big|+\Big|\partial_k\partial^{{\bf b}} A (\varphi \ast \vu_2)\Big|\,  \Big|(\partial_k\partial^{{\bf b}} \varphi)\ast \vw\Big|.
\end{split}
\end{equation*}

With this pointwise estimate, using the H\"older and Young inequalities, we obtain
\begin{equation*}
\begin{split}
I_1 \leq &\,  \Big( C_A \| ( \varphi \ast \vw)  \, ((\partial_k\partial^{{\bf b}} \varphi)\ast \vu_1) \|_{L^\infty} + \| \partial_k\partial^{{\bf b}} A (\varphi \ast \vu_2)  ((\partial_k\partial^{{\bf b}} \varphi)\ast \vw) \|_{L^\infty} \Big) \times \\
&\, \times  \left\|  (\vec{\nabla}\otimes \partial^{{\bf a}-{\bf b}} \vu_{2,\alpha})^T :   (\vec{\nabla} \otimes \partial_k \partial^{\bf{a}} \vw_{\alpha})^T \right\|_{L^1}\\
\leq &\, \Big( C_A \| \varphi \|_{L^2}\, \| \vw \|_{L^2} \, \| \partial_k \partial^{{\bf b}} \varphi \|_{L^2}\, \| \vu_1 \|_{L^2}   +  \| \partial_k \partial^{{\bf b}} A(\cdot)\|_{L^\infty} \, \| \partial_k \partial^{{\bf b}} \varphi \|_{L^2}\, \| \vw \|_{L^2}   \Big) \times \\
&\, \times  \left\|  \vec{\nabla}\otimes \partial^{{\bf a}-{\bf b}} \vu_{2,\alpha} \right\|_{L^2} \, \left\|  \vec{\nabla} \otimes \partial_k \partial^{\bf{a}} \vw_{\alpha} \right\|_{L^2}\\
\leq &\, L'_{m+2} \Big(\| \vu_1 \|_{L^2}+1\Big) \| \vw \|_{L^2}\, \| \vu_{2,\alpha}\|_{H^{m+1}}\, \| \partial_k \partial^{{\bf a}} \vw_\alpha \|_{H^1}. 
\end{split}
\end{equation*}	

Finally, note that the term $I_2$ follows from the same arguments. We thus have
\[ I_2 \leq  L'_{m+2} (\| \vu_1 \|_{L^2}+1) \| \vw \|_{L^2}\, \| \vu_{2,\alpha}\|_{H^{m+2}}\, \| \partial_k \partial^{{\bf a}} \vw_\alpha \|_{H^1},\]
where the only difference is the term $ \| \vu_{2,\alpha}\|_{H^{m+2}}$ instead of the term $ \| \vu_{2,\alpha}\|_{H^{m+1}}$ in $I_1$. Therefore, to obtain a uniform bound for $I_1$ and $I_2$, we just write
\[ I_1 \leq  L'_{m+2} (\| \vu_1 \|_{L^2}+1) \| \vw \|_{L^2}\, \| \vu_{2,\alpha}\|_{H^{m+2}}\, \| \partial_k \partial^{{\bf a}} \vw_\alpha \|_{H^1},\] 
from which follows the desired estimate.
\end{proof}	

For the fourth term, using the Cauchy--Schwarz inequality, we find that 
\begin{equation}\label{Estim-Tech-Filter-4}
\int_{\T} | \vw \cdot \partial^2_k \partial^{\bf{a}} \vw_\alpha | \, dx \leq \| \vw \|_{L^2}\, \| \partial_k (\partial_k \partial^{{\bf a}} \vw_\alpha) \|_{L^2}\leq \| \vw \|_{L^2}\, \| \partial_k \partial^{{\bf a}} \vw_\alpha\|_{H^1}. 
\end{equation}

Gathering estimates (\ref{Estim-Tech-Filter-1-2}), (\ref{Estim-Tech-Filter-3}), and (\ref{Estim-Tech-Filter-4}) into estimate (\ref{Estim-Tech-Filter-0}), we obtain
\begin{equation*}
\begin{split}
 C_{\alpha,\beta} \| \partial_{k} \partial^{{\bf a}}  \vw_\alpha \|^{2}_{H^1}  \leq  &\,  L'_{m+2}\, \| \vu_1 \|_{L^2}\, \| \vw_\alpha \|_{H^{m+1}}\,  \| \partial_k \partial^{{\bf a}} \vw_\alpha \|_{H^1} \\
&\, +  L'_{m+2} \Big(\| \vu_1 \|_{L^2}+1\Big) \| \vw \|_{L^2}\, \| \vu_{2,\alpha}\|_{H^{m+2}}\, \| \partial_k \partial^{{\bf a}} \vw_\alpha \|_{H^1} \\
&\, +   \| \vw \|_{L^2}\, \| \partial_k \partial^{{\bf a}} \vw_\alpha\|_{H^1}. 
\end{split}
\end{equation*}

Then, for a larger generic constant $L'_{m+2}\geq 1$, we write
\begin{equation*}
\begin{split}
C_{\alpha,\beta} \| \partial_{k} \partial^{{\bf a}}  \vw_\alpha \|_{H^1}  \leq  &\,  L'_{m+2}\, \| \vu_1 \|_{L^2}\, \| \vw_\alpha \|_{H^{m+1}}\ \\
&\, +  L'_{m+2} \Big(\| \vu_1 \|_{L^2}+1\Big) \| \vw \|_{L^2}\, \| \vu_{2,\alpha}\|_{H^{m+2}}\\
&\, + L'_{m+1}   \| \vw \|_{L^2}.
\end{split}
\end{equation*}

Using the fact that the index $k$ and the multi-index ${\bf a}$ are arbitrary, and writing $L'_{m+2}$ instead of $ \frac{L'_{m+2}}{C_{\alpha,\beta}}$, we have
\begin{equation*}
\begin{split}
\|  \vw_\alpha \|_{H^{m+2}}  \leq  &\,  L'_{m+2}\, \| \vu_1 \|_{L^2}\, \| \vw_\alpha \|_{H^{m+1}}\ \\
&\, +  L'_{m+2} \Big(\,\Big(\| \vu_1 \|_{L^2}+1\Big)  \| \vu_{2,\alpha}\|_{H^{m+2}}  + 1\Big)\,  \| \vw \|_{L^2} \\
\leq  &\,  L'_{m+2}\, \| \vu_1 \|_{L^2}\, \| \vw_\alpha \|_{H^{m+1}}\ \\
&\, +  L'_{m+2} \Big(\,\Big(\| \vu_1 \|_{L^2}+1\Big)  \| \vu_{2,\alpha}\|_{H^{m+2}}  + 1\Big)\,  \| \vw \|_{H^m}. 
\end{split}
\end{equation*}
Hence, we obtain the desired estimate (\ref{Recursive-Control-Diff-Filter-Hm}). Proposition \ref{Prop-Tech:Recursive-Estim-Filter-Diff-Hm} is proven.
\end{proof} 

With the estimate (\ref{Recursive-Control-Diff-Filter-Hm}) at our disposal, the desired estimate (\ref{Control-Filter-Diff-Hm}) can be obtained through the following iterative process.

\medskip

\emph{The initial case when $m=0$}.  By (\ref{Recursive-Control-Diff-Filter-Hm})  we have
\begin{equation*}
\begin{split}
&\, \|  \vu_{\alpha}(t_1,\cdot)-\vu_{\alpha}(t_2,\cdot) \|_{H^{2}} \\
 \leq & \,  L'_{2}\, \| \vu(t_1,\cdot) \|_{L^2}\, \| \vu_{\alpha}(t_1,\cdot)-\vu_{\alpha}(t_2,\cdot)\|_{H^{1}}\ \\
&\, +  L'_{2} \Big(\,\Big(\| \vu(t_1,\cdot) \|_{L^2}+1\Big)  \| \vu_{\alpha}(t_2,\cdot)\|_{H^{2}}  + 1\Big)\,  \| \vu(t_1,\cdot)-\vu(t_2,\cdot)\|_{L^2}.
\end{split}
\end{equation*}
The term $\|  \vu_{\alpha}(t_1,\cdot)-\vu_{\alpha}(t_2,\cdot) \|_{H^{2}}$ was estimated in (\ref{Control-Filter-Diff-H1}), and the term $\| \vu_{\alpha}(t_2,\cdot)\|_{H^{2}}$ was estimated in (\ref{Control-H2-Filter}). Using these estimates and denoting $\ds{\| g \|_{L^\infty_t H^k_x}:= \sup_{T_1\leq t \leq T_2}\| g(t,\cdot)\|_{H^k}}$ (with $k\geq 0$), we write
\begin{equation*}
\begin{split}
&\|  \vu_{\alpha}(t_1,\cdot)-\vu_{\alpha}(t_2,\cdot) \|_{H^{2}} \\
 \leq & \,  L'_{2}\, \| \vu \|_{L^\infty_t L^2_x}\,  K_1\Big(\| \vu\|_{L^\infty_t L^2_x}+1\Big) \|  \vu(t_1,\cdot)-\vu(t_2,\cdot)\|_{L^2} \\
&\, +  L'_{2} \left(\,(\| \vu \|_{L^\infty_tL^2_x}+1)   K_2\left( \| \vu \|^2_{L^\infty_t L^2_x }+ \| \vu \|_{L^\infty_t L^2_x}\right)    + 1\right)\,  \| \vu(t_1,\cdot)-\vu(t_2,\cdot)\|_{L^2}\\
=:&\, K'_2\, \| \vu(t_1,\cdot)-\vu(t_2,\cdot)\|_{L^2}. 
\end{split}
\end{equation*}

\emph{The iterative process with $m\geq 1$}. Assume that 
\[ \| \vu_{\alpha}(t_1,\cdot)-\vu_{\alpha}(t_2,\cdot)\|_{H^{m+1}}\leq K'_{m+1}\| \vu(t_1,\cdot)-\vu(t_2,\cdot)\|_{H^{m-1}}. \]
With this estimate, returning to (\ref{Recursive-Control-Diff-Filter-Hm}) and noticing that the expression $\| \vu_{\alpha}(t_2,\cdot)\|_{H^{m+2}}$ was estimated in (\ref{Control-Hm-Filter}), we directly write
		\begin{equation*}
\begin{split}
&\, 	\|  \vu_{\alpha}(t_1,\cdot)-\vu_{\alpha}(t_2,\cdot) \|_{H^{m+2}} \\ \leq  & \,  L'_{m+2}\, \| \vu(t_1,\cdot) \|_{L^2}\, \|  \vu_{\alpha}(t_1,\cdot)-\vu_{\alpha}(t_2,\cdot)  \|_{H^{m+1}}\ \\
&\, +  L'_{m+2} \left(\,\big(\| \vu(t_1,\cdot) \|_{L^2}+1\big)  K_{m+2}\sum_{k=0}^{m+2}\| \vu(t_2,\cdot)\|^k_{H^m}\,  + 1\right)\,  \| \vu(t_1,\cdot)-\vu(t_2,\cdot) \|_{H^m}\\
\leq &\,  L'_{m+2}\, \| \vu \|_{L^\infty_t L^2_x}\,  K'_{m+1}\| \vu(t_1,\cdot)-\vu(t_2,\cdot)\|_{H^{m}}\\
&\, +  L'_{m+2} \left(\,\big(\| \vu \|_{L^\infty_t L^2_x}+1\big)  K_{m+2}\sum_{k=0}^{m+2}\| \vu\|^k_{L^\infty_t H^m_x}\,  + 1\right)\,  \| \vu(t_1,\cdot)-\vu(t_2,\cdot) \|_{H^m}\\
=:&\, K'_{m+2}\, \| \vu(t_1,\cdot)-\vu(t_2,\cdot) \|_{H^m}.
\end{split}
\end{equation*}
This concludes the proof of Theorem \ref{Th:Cont-Dep-Hm-Filter}. \end{proof}	

\section{Theory of regularity for the weak Leray  solutions}\label{Sec:Reg-Leray}

\subsection{Proof of Theorem \ref{Th-Regularity}} 

Before proceeding with the proof, we state some well-known lemmas that will be useful in what follows.

\begin{Lemma}\label{Lem-Tech-1} 
	Let $0\leq \sigma_1, \sigma_2$ be two parameters, and let $e^{\nu t \Delta}(\cdot)$ denote the heat semigroup defined in (\ref{Heat-Kernel}). 
	\begin{enumerate}
		\item For any $t>0$, we have $\| e^{\nu t\Delta} g \|_{\dot{H}^{\sigma_1+\sigma_2}} \leq C  t^{-\frac{\sigma_2}{2}}\| g \|_{\dot{H}^{\sigma_1}}$, for some constant $C>0$ depending on $\sigma_2$ and $\nu$. 
		\item For any $\varepsilon>0$ and for any $\varepsilon<t_1,t_2$, we have $\|e^{\nu t_1 \Delta} g - e^{\nu t_2 \Delta} g\|_{H^{\sigma_1+\sigma_2}} \leq C |t_1 - t_2|^{1/2}\, \| g \|_{H^{\sigma_1}}$, for some constant $C>0$ depending on $\sigma_2$, $\varepsilon$, and $\nu$.
	\end{enumerate} 
\end{Lemma}

\begin{proof} 
	To prove the first point, we write
	\begin{equation*}
	\begin{split}
	\left\| e^{\nu t \Delta} g \right\|^2_{\dot{H}^{\sigma_1+\sigma_2}}= &\,  \sum_{k \in \Z} |k|^{2(\sigma_1+\sigma_2)} e^{-2\nu |k|^2 t}|\widehat{g}_k|^2 \\
	\leq&\,  \left( \sup_{\xi \in \Rt} |\xi|^{2\sigma_2} e^{-2\nu |\xi|^2 t} \right)\, \sum_{k\in \Z} |k|^{2\sigma_1} |\widehat{g}_k|^2 \leq   C t^{-\sigma_2}\, \| g \|^2_{\dot{H}^{\sigma_1}}. 
	\end{split}
	\end{equation*}
	
	To verify the second point, we write	
	\begin{equation*}
	\begin{split}
	\left\| e^{\nu \Delta t_1} g -  e^{\nu \Delta t_2} g \right\|^2_{H^{\sigma_1+\sigma_2}}= &\,  \sum_{k \in \Z} (1+|k|^2)^{\sigma_1+\sigma_2} \left(e^{-\nu |k|^2 t_1}- e^{-\nu |k|^2 t_2}\right)^2|\widehat{g}_k|^2 \\
	=&\, \sum_{k \in \Z} (1+|k|^2)^{\sigma_2} e^{-2\nu |k|^2 t_2} \left( e^{-\nu |k|^2 (t_1-t_2)}-1 \right)^2 (1+|k|^2)^{\sigma_1}|\widehat{g}_k|^2.  
	\end{split}
	\end{equation*}
	
	In the term $\left( e^{- \nu |k|^2 (t_1-t_2)}-1 \right)^2$, without loss of generality, we may assume that $t_1>t_2$. Hence, we obtain 
	\[ \left( e^{-\nu |k|^2 (t_1-t_2)}-1 \right)^2 = \left( e^{-\nu |k|^2 (t_1-t_2)}-1 \right) \left( e^{-\nu |k|^2 (t_1-t_2)}-1 \right) \leq 2 \left( e^{- \nu|k|^2 (t_1-t_2)}-1 \right). \] 
	
	Moreover, applying the mean value theorem with respect to the time variable, we find that 
	\[ 2 \left( e^{- \nu |k|^2 (t_1-t_2)}-1 \right) \leq C |k|^2 |t_1-t_2|.  \]
	
	With this estimate, returning to the previous identity, and since $t_2>\varepsilon$, we have
	\begin{equation*}
	\begin{split}
	\left\| e^{\nu t \Delta} g \right\|^2_{\dot{H}^{\sigma_1+\sigma_2}}\leq &\, C |t_1-t_2|\, \sum_{k \in \Z}(1+|k|^2)^{\sigma_2}|k|^2\, e^{-2\nu |k|^2 t_2} (1+|k|^2)^{\sigma_1} |\widehat{g}_k|^2 \\
	\leq  &\, C |t_1-t_2|\, \sum_{k \in \Z}(1+|k|^2)^{\sigma_2+1}\, e^{-2\nu |k|^2 \varepsilon} (1+|k|^2)^{\sigma_1} |\widehat{g}_k|^2\\
	\leq &\,  C |t_1-t_2| \left( \sup_{\xi \in \Rt} (1+|\xi|^2)^{\sigma_2+1} e^{-2\nu |\xi|^2 \varepsilon} \right)\, \sum_{k\in \Z} (1+|k|^2)^{\sigma_1}|\widehat{g}_k|^2 \\
	\leq &\,  C |t_1-t_2|\, \| g \|^2_{H^{\sigma_1}}. 
	\end{split}
	\end{equation*}
\end{proof}

\begin{Lemma}\label{Lem-Tech-2} 
	Let $\sigma \geq 0$ and let $g \in L^{2}([0,T], \dot{H}^{\sigma}(\T))$. Define $\ds{G(t,x)= \int_{0}^{t} e^{\nu (t-s)\Delta} g(s,x)ds}$. Then, there exists a constant $C>0$, depending only on $\nu>0$, such that the following properties hold:   
	\begin{enumerate} 
		\item We have $G \in \mathcal{C}([0,T],H^{\sigma+1}(\T))$ and $\ds{\|G\|_{L^{\infty}_{t}H^{\sigma+1}_{x}} \leq C\,  \|g\|_{L^{2}_{t} \dot{H}^{\sigma}_{x}}}$.
		\item We have $G\in L^2([0,T],\dot{H}^{\sigma+2}(\T))$ and $\ds{\|G\|_{L^{2}_{t}\dot{H}^{\sigma+2}_{x}} \leq C\,  \|g\|_{L^{2}_{t} \dot{H}^{\sigma}_{x}}}$.   
	\end{enumerate}   
\end{Lemma}

For a proof, see \cite[Lemma $7.2$]{PLe} in the case $\sigma=0$; the general case $\sigma \geq 0$ follows by a straightforward extension.

\begin{Lemma}\label{Lem-Tech-3}  Let $\sigma >0$ and $(I_d - \Delta)^{\frac{\sigma}{2}}$ be the Bessel potential defined in (\ref{Bessel-Potental}). Then, the following estimate holds
	\[ \left\|  (I_d - \Delta)^{\frac{\sigma}{2}} ( g\, h )\right\|_{L^2} \leq C \Big( \left\| (I_d - \Delta)^{\frac{\sigma}{2}} g \right\|_{L^2}\, \| h \|_{L^\infty}  + \| g \|_{L^2}\, \left\| (I_d - \Delta)^{\frac{\sigma}{2}} h \right\|_{L^\infty}\Big).\]
\end{Lemma}	
This estimate is a particular case of the so-called fractional Leibniz rule on the torus. For a proof, refer to \cite[Proposition $1$]{Benyi}. 

\medskip 

We are now in a position to prove Theorem \ref{Th-Regularity}. Let $(\vu, \vu_\alpha, P, P_\alpha)$ be the unique weak Leray solution to the system (\ref{Alpha-model}) constructed in Theorem \ref{Th-GWP}. Moreover, let $0<T_1<T_2<+\infty$ be fixed. We prove the gain of regularity (\ref{Regularity-Leray-Solution}) by induction with respect to the regularity parameter $m$.
\medskip

Our general strategy is to study the velocity $\vu$ first. Once a gain in regularity for $\vu$ has been established, the remaining unknown $\vu_\alpha$ can be analyzed using the results proved in Section \ref{Sec:Higher regularity properties of the nonlinear filter  equation}.

\medskip

\emph{The initial case when $m=0$.} Assume that $\vf \in L^2_{\rm loc}([0,+\infty[, L^2(\T))$, $A(\cdot)\in \mathcal{C}^{1,1}(\Rt)$, and $\varphi \in H^1(\T)$. On the other hand, let $0<\varepsilon_0<1$ be sufficiently small such that $T_1-\varepsilon_0 >0$. We prove that $\vu \in \mathcal{C}([T_1-\varepsilon_0,T_2],H^1(\T))\cap L^2([T_1-\varepsilon_0,T_2],\dot{H}^2(\T))$. To this end, we use the (equivalent) coupled integral-elliptic formulation of system (\ref{Alpha-model}), which, for the reader's convenience, we rewrite here:
\begin{equation}\label{Mild-Formulation}
\begin{cases}\vspace{2mm}
\ds{\vu(t,\cdot)=e^{\nu t\Delta }\vu_0 + \int_{0}^{t}e^{\nu(t-s) \Delta}\vf(s,\cdot) ds - \int_{0}^{t}e^{\nu (t-s) \Delta}\P\Big((\vu_\alpha \cdot \vec{\nabla})\vu\Big)(s,\cdot)ds}, \quad {\rm div}(\vu)=0,\\
\vu_\alpha(t,\cdot) = \alpha^2 \P\, {\rm div}\Big(A(\varphi\ast \vu(t,\cdot)) (\vec{\nabla}\otimes \vu_\alpha(t,\cdot))^T\Big)+ \vu(t,\cdot), \quad {\rm div}(\vu_\alpha)=0.
\end{cases}
\end{equation}
Specifically, we study each term on the right-hand side of the first equation above.

\begin{Remark}
In (\ref{Mild-Formulation}), the gain of regularity for the initial data term $e^{\nu t \Delta}\vu_0$ is given by Lemma \ref{Lem-Tech-1} only for times $t>0$. We thus consider the time interval $[T_1-\varepsilon_0,T_2]$ with $T_1-\varepsilon_0>0$. Here, the parameter $\varepsilon_0>0$ is introduced to ensure a gain of regularity even slightly before time $T_1$.
\end{Remark}

For the first term, as $\vu_0 \in L^2_{\text{s}}(\T)$, using Lemma \ref{Lem-Tech-1}, first with $\sigma_1=\sigma_2=0$, we obtain that $e^{\nu t \Delta}\vu_0 \in \mathcal{C}([T_1-\varepsilon_0, T_2], L^2(\T))$. The same argument with $\sigma_1=0$ and $\sigma_2=1$,  since $t\geq T_1-\varepsilon_0$ yields that $e^{\nu t \Delta}\vu_0 \in \mathcal{C}([T_1-\varepsilon_0, T_2], \dot{H}^1(\T))$. On the other hand, setting now $\sigma_1=0$ and $\sigma_2=0$, and noting again that $t\geq T_1-\varepsilon_0$, we obtain that $e^{\nu t \Delta}\vu_0 \in L^2([T_1-\varepsilon_0,T_2], \dot{H}^2(\T))$. We thus have
\begin{equation}\label{Reg-Data-0}
e^{\nu t\Delta }\vu_0  \in   \mathcal{C}\big([T_1-\varepsilon_0,T_2],H^1(\T)\big)\cap L^2\big([T_1-\varepsilon_0,T_2],\dot{H}^2(\T)\big). 
\end{equation}

For the second term, as $\vf \in L^2([0,T_2], L^2(\T))$, by Lemma \ref{Lem-Tech-2} (with $\sigma=0$) we obtain
\begin{equation}\label{Reg-Force-0}
\int_{0}^{t}e^{\nu(t-s) \Delta}\P(\vf)(s,\cdot) ds \in   \mathcal{C}([0,T_2],H^1(\T))\cap L^2([0,T_2],\dot{H}^2(\T)). 
\end{equation}

For the third term, we first verify that $(\vu_\alpha \cdot \vec{\nabla})\vu \in L^{2}([0,T_2], L^2(\T))$. Indeed, from the assumptions that $A(\cdot)\in \mathcal{C}^{1,1}(\Rt)$, $\varphi \in H^1(\T)$, and the fact that $\vu \in L^{\infty}([0,T_2],L^2_{\text{s}} (\T))$, by Proposition \ref{Prop:H2-Filter} it follows that $\vu_\alpha \in L^\infty_t H^2_x$. Additionally, recalling that $\vu \in L^2([0,T_2], \dot{H}^1(\T))$, using H\"older inequalities and Sobolev embeddings we write
\[
\|  (\vu_\alpha \cdot \vec{\nabla})\vu  \|_{L^2_t L^2_x} \leq \| \vu_\alpha \|_{L^\infty_t L^{\infty}_x} \, \| \vec{\nabla}\otimes \vu \|_{L^2_t L^2_x} \leq C\, \| \vu_\alpha\|_{L^\infty_t H^2_x}\, \| \vu \|_{L^2_t \dot{H}^1_x}<+\infty.
\]

Then, by Lemma \ref{Lem-Tech-2} (with $\sigma=0$), we obtain
\begin{equation}\label{Reg-Nonlinear-0}
\int_{0}^{t}e^{\nu(t-s) \Delta}\P((\vu_\alpha \cdot \vec{\nabla})\vu )(s,\cdot) ds \in   \mathcal{C}([0,T_2],H^1(\T))\cap L^2([0,T_2],\dot{H}^2(\T)).
\end{equation}

Gathering the information from (\ref{Reg-Data-0}), (\ref{Reg-Force-0}), and (\ref{Reg-Nonlinear-0}) in expression (\ref{Mild-Formulation}), we conclude that
\[
\vu \in \mathcal{C}([T_1-\varepsilon_0,T_2],H^1(\T))\cap L^2([T_1-\varepsilon_0,T_2],\dot{H}^2(\T)).
\]

With this fact, we can now deduce a gain of regularity for the remaining unknowns $\vu_\alpha$, $P$, and $P_\alpha$ as follows. For the filtered velocity $\vu_\alpha$, as $\vu \in \mathcal{C}([T_1-\varepsilon_0,T_2],H^1(\T))\subset \mathcal{C}([T_1-\varepsilon_0,T_2],L^2(\T))$, $A(\cdot)\in \mathcal{C}^{1,1}(\Rt)$  and $\varphi \in H^1(\T)$, by Theorems \ref{Th:Hm-Filter} and \ref{Th:Cont-Dep-Hm-Filter}  we have 
\[ \vu_\alpha \in \mathcal{C}([T_1-\varepsilon_0, T_2], H^{2}(\T)). \]
\begin{Remark} Note that, under these precise regularity assumptions on $A(\cdot)$ and $\varphi$, the fact that $\vu \in \mathcal{C}_t H^1_x$ does not imply that $\vu_\alpha \in \mathcal{C}_t H^3_x$. See Remark \ref{Rmk:Reg-Filter-Constraint} above.
\end{Remark}

\emph{The iterative process with $m \geq 1$.} Assume that $\vf \in L^2_{loc}([0,+\infty[, H^m(\T))$, $A(\cdot)\in \mathcal{C}^{m+1,1}(\Rt)$, and $\varphi \in H^{m+1}(\T)$. Moreover, for $0<\varepsilon_m <\varepsilon_0$ such that $T_1-\varepsilon_m>0$, assume that we have 
\[ \vu \in \mathcal{C}([T_1-\varepsilon_m,T_2], H^{m}(\T))\cap L^2([T_1-\varepsilon_m,T_2], \dot{H}^{m+1}(\T)). \] 
We prove that this also holds for $m+1$.
 
\medskip

Once again, we return to equation (\ref{Mild-Formulation}), which for any time $t \geq T_1 -\varepsilon_m$ reads as follows:
\begin{equation}\label{Mild-Formilation-2}
\begin{split}
\vu(t,\cdot)=&\, e^{(t-(T_1 -\varepsilon_m))\Delta}\vu(T_1-\varepsilon_m,\cdot)+\int_{T_1-\varepsilon_m}^{t}e^{(t-s)\Delta}\P(\vf)(s,\cdot)ds \\
&\, - \int_{T_1-\varepsilon_m}^{t}e^{(t-s)\Delta}\P((\vu_\alpha \cdot \vec{\nabla})\vu)(s,\cdot)ds,
\end{split}
\end{equation}
where we must estimate each term on the right-hand side.

\medskip

For the first term, we now set $0<\varepsilon_{m+1}<\varepsilon_{m}$ such that $T_1-\varepsilon_{m+1}>0$. As  $\vu \in \mathcal{C}([T_1-\varepsilon_{m}, T_2], H^m(\T))$,  applying Lemma \ref{Lem-Tech-1} we have 
\begin{equation}\label{Reg-Data-m}
e^{(t-(T_1 -\varepsilon_m))\Delta}\vu(T_1-\varepsilon_m,\cdot) \in \mathcal{C}([T_1-\varepsilon_{m+1}, T_2], H^{m+1}(\T))\cap L^2([T_1-\varepsilon_{m+1}, T_2], \dot{H}^{m+2}(\T)).
\end{equation}

For the second term, applying Lemma \ref{Lem-Tech-2} to the function $\mathds{1}_{[T_1-\varepsilon_m,T_2]}(t)\vf$, we have
\begin{equation}\label{Reg-Force-m}
\int_{T_1-\varepsilon_m}^{t}e^{(t-s)\Delta}\P(\vf)(s,\cdot)ds \in \mathcal{C}([T_1-\varepsilon_m,T_2],H^{m+1}(\T))\cap L^2([T_1 - \varepsilon_m,T_2],\dot{H}^{m+1}(\T)).
\end{equation}

Finally, for the third term, we start by verifying that $(\vu_\alpha \cdot \vec{\nabla})\vu\in L^2([T_1-\varepsilon_{m},T_2],\dot{H}^m(\T))$. In fact, as $A(\cdot) \in \mathcal{C}^{m+1,1}(\Rt)$, $\varphi \in H^{m+1}(\T)$, and $\vu \in \mathcal{C}([T_1-\varepsilon_m, T_2], H^m(\T))$, by Theorems \ref{Th:Hm-Filter} and \ref{Th:Cont-Dep-Hm-Filter} we have $\vu_\alpha \in \mathcal{C}([T_1-\varepsilon_m, T_2], H^{m+2}(\T))$. Then, for any multi-index $|{\bf b}|\leq m$, by Sobolev embeddings we have 
\[ \partial^{\bf b} \vu_\alpha  \in \mathcal{C}([T_1-\varepsilon_m, T_2],H^2(\T)) \subset \mathcal{C}([T_1-\varepsilon_m, T_2], L^\infty(\T)).  \]
With this information, and recalling that we have assumed $\vu \in L^2([T_1-\varepsilon_m, T_2], \dot{H}^{m+1}(\T))$, using the Leibniz rule we can write 
\begin{equation*}
\begin{split}
&\left\| (\vu_\alpha \cdot \vec{\nabla})\vu \right\|_{L^2_t \dot{H}^m_x}
=\, \sum_{|{\bf a}|\leq m} \left\| \partial^{{\bf a}} \Big((\vu_\alpha \cdot \vec{\nabla}) \vu \Big)\right\|_{L^2_t L^2_x}\\
\leq &\, \sum_{|{\bf a}|\leq m} \sum_{|{\bf b}|\leq |{\bf a}|} C_{{\bf a}, {\bf b}}\left\| \partial^{\bf b} ( \vu_\alpha) \right\|_{L^\infty_t L^\infty_x}\, \left\| \partial^{{\bf a}-{\bf b}}\Big( \vec{\nabla}\otimes \vu \Big) \right\|_{L^2_t L^2_x}<+\infty,
\end{split}
\end{equation*}
with a constant $C_{{\bf a}, {\bf b}}>0$. Once we have $(\vu_\alpha \cdot \vec{\nabla})\vu\in L^2_t \dot{H}^m_x$, by Lemma \ref{Lem-Tech-2} we obtain
\begin{equation}\label{Reg-Nonlinear-m}
\int_{T_1-\varepsilon_m}^{t}e^{(t-s)\Delta}\P((\vu_\alpha \cdot \vec{\nabla})\vu)(s,\cdot)ds  \in \mathcal{C}([T_1-\varepsilon_m,T_2],H^{m+1}(\T))\cap L^2([T_1 - \varepsilon_m,T_2],\dot{H}^{m+2}(\T)).    
\end{equation}

As before, we gather the information (\ref{Reg-Data-m}), (\ref{Reg-Force-m}), and (\ref{Reg-Nonlinear-m}) in equation (\ref{Mild-Formilation-2}) to obtain that 
\begin{equation}\label{Th-Reg-Nonlinear}
\vu \in \mathcal{C}([T_1-\varepsilon_{m+1},T_2], H^{m+1}(\T))\cap L^2([T_1-\varepsilon_{m+1},T_2], \dot{H}^{m+2}(\T)). 
\end{equation}

By this iterative process, and since we have $0<\varepsilon_{m+1}<\varepsilon_{m}<\cdots <\varepsilon_0$, we conclude that 
\[ \vu \in \mathcal{C}([T_1,T_2], H^{m+1}(\T))\cap L^2([T_1,T_2], \dot{H}^{m+2}(\T)).  \]

With this information, the regularity of the remaining unknown $\vu_\alpha$ follows from the same arguments used in the case $m=0$. This concludes the proof of Theorem \ref{Th-Regularity}.

\subsection{Proof of Proposition \ref{Prop:Homogeneous-case}} 
We follow the same ideas as in the proof of Theorem \ref{Th-Regularity}. Returning to the iterative process detailed above with $m\geq 1$, recall that by assuming 
\[ \vu \in \mathcal{C}([T_1-\varepsilon_m,T_2], H^{m}(\T))\cap L^2([T_1-\varepsilon_m,T_2], \dot{H}^{m+1}(\T)), \]
we obtained  
\[ \vu \in \mathcal{C}([T_1-\varepsilon_{m+1},T_2], H^{m+1}(\T))\cap L^2([T_1-\varepsilon_{m+1},T_2], \dot{H}^{m+2}(\T)), \]
for $0<\varepsilon_{m+1}<\varepsilon_{m}$ such that $T_1-\varepsilon_{m+1}>0$. 

\medskip

As $\vf \equiv 0$, we now prove that 
\begin{equation}\label{Information-Reg-u}
\vu \in \mathcal{C}([T_1-\varepsilon_{m+2},T_2],H^{m+2}(\T)) \cap L^2([T_1-\varepsilon_{m+2},T_2],\dot{H}^{m+3}(\T)),   
\end{equation}
with $0<\varepsilon_{m+2}<\varepsilon_{m+1}$ satisfying $T_1-\varepsilon_{m+2}>0$.

\medskip

For any time $t\geq T_1-\varepsilon_{m+1}$, we now write
\[ \vu(t,\cdot)=\, e^{(t-(T_1 -\varepsilon_{m+1}))\Delta}\vu(T_1-\varepsilon_{m+1},\cdot) - \int_{T_1-\varepsilon_{m+1}}^{t}e^{(t-s)\Delta}\P((\vu_\alpha \cdot \vec{\nabla})\vu)(s,\cdot)ds.\]
For the first term, as $\vu \in \mathcal{C}([T_1-\varepsilon_{m+1},T_2], H^{m+1}(\T))$, applying Lemma \ref{Lem-Tech-1}, we have 
\begin{equation}\label{Cor-Reg-Data}
e^{(t-(T_1 -\varepsilon_{m+1}))\Delta}\vu(T_1-\varepsilon_{m+1},\cdot) \in \mathcal{C}([T_1-\varepsilon_{m+2},T_2],H^{m+2}(\T)) \cap L^2([T_1-\varepsilon_{m+2},T_2],\dot{H}^{m+3}(\T)).
\end{equation}

For the second term, we start by verifying that $(\vu_\alpha \cdot \vec{\nabla})\vu\in L^2([T_1-\varepsilon_{m+1},T_2], \dot{H}^{m+1}(\T))$. Indeed, on the one hand, from (\ref{Th-Reg-Nonlinear}), as $\vu\in L^2([T_1-\varepsilon_{m+1},T_2], \dot{H}^{m+2}(\T))$, we directly obtain that $\vec{\nabla} \otimes \vu \in L^2([T_1-\varepsilon_{m+1}, T_2], \dot{H}^{m+1}(\T))$. Additionally, as $\vu\in \mathcal{C}([T_1-\varepsilon_{m+1},T_2], {H}^{m+1}(\T))$, we have $\vec{\nabla} \otimes \vu\in \mathcal{C}([T_1-\varepsilon_{m+1},T_2], {H}^{m}(\T)) \subset L^2([T_1-\varepsilon_{m+1}, T_2], L^2(\T))$. Then, it holds that 
$\vec{\nabla} \otimes \vu \in L^2([T_1-\varepsilon_{m+1}, T_2], H^{m+1}(\T))$. 

\medskip

On the other hand, also from (\ref{Th-Reg-Nonlinear}), as $\vu\in \mathcal{C}([T_1-\varepsilon_{m+1},T_2], {H}^{m+1}(\T)) \subset \mathcal{C}([T_1-\varepsilon_{m+1},T_2], {H}^{m}(\T))$, by Theorems \ref{Th:Hm-Filter} and \ref{Th:Cont-Dep-Hm-Filter} we obtain that $\vu_\alpha \in \mathcal{C}([T_1-\varepsilon_{m+1},T_2], H^{m+2}(\T))\subset \mathcal{C}([T_1-\varepsilon_{m+1},T_2], H^{m+1}(\T))$. 

\begin{Remark}\label{Rmk:Reg-u-alpha} 
As before,	even though we have $\vu \in \mathcal{C}_t H^{m+1}_x$, Theorems \ref{Th:Hm-Filter} and \ref{Th:Cont-Dep-Hm-Filter} do not allow us to conclude that $\vu_{\alpha} \in \mathcal{C}_t H^{m+3}_x$, since this would require the higher regularity assumptions 
	\[
	A(\cdot)\in \mathcal{C}^{m+2,1}(\Rt) \quad \text{and} \quad \varphi \in H^{m+2}(\T).
	\]
	Therefore, under the assumptions $A(\cdot)\in \mathcal{C}^{m+1,1}(\Rt)$ and $\varphi\in H^{m+1}(\T)$, the gain of regularity for $\vu_\alpha$ is restricted to $\vu_\alpha \in \mathcal{C}_t H^{m+2}_x$.
\end{Remark}

\medskip

Then, as $m+1>3/2$ (recall that $m\geq 1$), we can apply product laws in non-homogeneous Sobolev spaces to get $(\vu_\alpha \cdot \vec{\nabla})\vu\in L^{2}([T_1-\varepsilon_{m+1},T_2], H^{m+1}(\T))\subset L^{2}([T_1-\varepsilon_{m+1},T_2], \dot{H}^{m+1}(\T))$. 

\medskip

Therefore, by Lemma \ref{Lem-Tech-1}, we obtain 
\begin{equation}\label{Cor-Reg-Nonlinear}
\int_{T_1-\varepsilon_{m+1}}^{t}e^{(t-s)\Delta}\P((\vu_\alpha \cdot \vec{\nabla})\vu)(s,\cdot)ds  \in \mathcal{C}([T_1-\varepsilon_{m+1},T_2],H^{m+2}(\T))\cap L^2([T_1 - \varepsilon_{m+1},T_2],\dot{H}^{m+3}(\T)).
\end{equation}

In this manner, gathering the information in (\ref{Cor-Reg-Data}) and (\ref{Cor-Reg-Nonlinear}), we obtain (\ref{Information-Reg-u}). 

\medskip

Once we have (\ref{Information-Reg-u}), and recalling that $\vu_\alpha \in \mathcal{C}([T_1-\varepsilon_{m+1},T_2], H^{m+2}(\T))$, we can repeat the process detailed above to obtain
\begin{equation*}
\vu \in \mathcal{C}([T_1-\varepsilon_{m+3},T_2],H^{m+3}(\T)) \cap L^2([T_1-\varepsilon_{m+3},T_2],\dot{H}^{m+3}(\T)),   
\end{equation*}
with $0<\varepsilon_{m+3}<\varepsilon_{m+2}$ satisfying $T_1-\varepsilon_{m+3}>0$. We thus obtain (\ref{Regularity-Leray-Solution-improved}), which concludes the proof of Proposition \ref{Prop:Homogeneous-case}.

\section{Theory of regularity for the global attractor}\label{Sec:Reg-Attractor}
\subsection{Proof of Theorem \ref{Th-Regularity-Atrtractor}}
Let $\vu_{\rm e} \in \mathcal{T}(\mathcal{A})$, where the set $\mathcal{T}(\mathcal{A})$ is given in Definition \ref{Def-Complete-Trajectory}.   To simplify the notation, from now on we will write $\vu$ instead of $\vu_{\text{e}}$. In this context, recall that for any $t \in \R$ and $x\in \T$ the vector field $\vu$ solves the coupled system 
		\begin{equation}\label{Alpha-model-eternal-attractor} 
\begin{cases}\vspace{2mm}
\partial_t \vu - \nu\Delta \vu + \P\big( (\vu_{\alpha} \cdot \vec{\nabla})\vu\big)  = \vf, \qquad {\rm div}(\vu)=0, \\ 
-\alpha^2 \, \P\Big( {\rm div} \left( A (\varphi \ast \vu) (\vec{\nabla} \otimes \vu_{\alpha})^T \right)\Big)+\vu_{\alpha} = \vu, \qquad {\rm div}(\vu_{\alpha})=0.
\end{cases} 
\end{equation}

Note that   as $A(\cdot)$  and $\varphi$ satisfy the assumption (\ref{conditions-a-varphi-regularity-attractor-H1}), and as $\vu \in \mathcal{T}(\mathcal{A})\subset  L^\infty\big(\R, L^2_{\text{s}}(\T)\big)$, by Propositions \ref{Prop:Existence-Filter} and \ref{Prop:H2-Filter} we have that the eternal filtered velocity $\vu_\alpha$ satisfies  $\vu_\alpha \in L^\infty\big(\R, H^2(\T)\big)$ and $\text{div}(\vu_\alpha)=0$. 

\medskip

The main strategy in the  proof of  (\ref{Regularity-Atrractor}) is to decompose the complete trajectory   $\vu$ as the sum of solutions of two suitable affine systems related to (\ref{Alpha-model-eternal-attractor}). Then, we will exploit some useful properties of each system.  In the following technical propositions, we detail each step of the proof. 

\medskip

\emph{Step 1: Decomposition of the complete trajectory}. 

\begin{Proposition}\label{Prop1:Decomposition-Eternal-Sol}  Let  $(\vu, \vu_\alpha) \in \mathcal{T}(\mathcal{A}) \times L^\infty\big(\R, H^2(\T)\big)$ be the solution of the system (\ref{Alpha-model-eternal-attractor}). On the other hand, let $t_0 \in \R$ be a fixed time. Then, let $(t_n)_{n\in \mathbb{N}}$ be a sequence of times such that $t_n>2-t_0$, $t_n>0$ and 
$t_n \to +\infty$ as $n\to +\infty$.
	
\medskip

\begin{enumerate}
	\item For the given  function $\vu_\alpha$ and  for  any $t_n>0$ fixed, consider the initial value problems:
	\begin{equation}\label{Alpha-model-vn}
	\begin{cases}\vspace{2mm}
	\partial_t \vv_n-\nu\Delta \vv_n+\P\big((\vu_\alpha\cdot\vec{\nabla})\vv_n\big)=\vf, \quad {\rm div}(\vv_n)=0,\\
	\vv_n(-t_n,\cdot)=0,
	\end{cases}
	\end{equation}
	and 
	\begin{equation}\label{Alpha-model-wn}
	\begin{cases}\vspace{2mm}
	\partial_t \vw_n-\nu\Delta \vw_n+\P\big((\vu_\alpha\cdot\vec{\nabla})\vw_n\big)=0, \quad {\rm div}(\vw_n)=0,\\
	\vw_n(-t_n,\cdot)=\vu(-t_n),
	\end{cases}
	\end{equation}
	
	Then, each equation respectively  admits a unique solution  
	\[ \vv_n, \vw_n\in L^{\infty}_{\rm loc}\big( [-t_n, +\infty[, L^2_{\text{s}}(\T) \big) \cap L^2_{\rm loc}\big( [-t_n, +\infty[, \dot{H}^1(\T) \big).\]

	\medskip

	\item  For any $t\geq - t_n$, the following identity  holds:
	\begin{equation}\label{Decomposition}
	\vu(t,\cdot)=\vv_n(t,\cdot)+ \vw_n(t,\cdot). 
	\end{equation}
	\end{enumerate}		
	 \end{Proposition}	
 \begin{Remark} Note that the sequence $(t_n)_{n \in \mathbb{N}}$ depends on the fixed time $t_0\in \R$, since we assume that $t_n>2-t_0$. This latter technical constraint will be used later; see Remark \ref{Rmk-tn-t0} below. 	
 \end{Remark}	
\begin{proof}
	\emph{First point}. As both equations (\ref{Alpha-model-vn}) and (\ref{Alpha-model-wn}) essentially have the same structure, the existence and uniqueness of solutions can be proven for a unified generic system.
	\begin{Lemma}\label{Lem:GWP-Unified-Equation}Let $\vu_\alpha \in L^\infty\big(\R, H^2(\T)\big)$, with $\text{div}(\vu_\alpha)=0$, and let $t_n> \max\{2-t_0, 0\}$ be fixed. Consider the initial value problem
		\begin{equation}\label{Alpha-model-q}
		\begin{cases}\vspace{2mm}
		\partial_t \vpsi-\nu\Delta \vpsi+\P\left(\big(\vu_\alpha\cdot\vec{\nabla}\big)\vpsi\,\right)=\vg, \quad {\rm div}(\vpsi)=0,\\
		\vpsi(-t_n,\cdot)=\vpsi_0. 
		\end{cases}
		\end{equation}
		
		Assume that $\vpsi_0 \in L^2_{\text{s}}(\T)$ and  $\vg \in L^2_{\text{s}}(\T)$. Then, there exists a unique solution of (\ref{Alpha-model-q}):
		\[ \vpsi \in L^\infty_{\rm loc}\big( [-t_n, +\infty[, L^2_{\text{s}}(\T)\big) \cap L^2_{\rm loc}\big( [-t_n,+\infty[, \dot{H}^1(\T)\big), \]	
		which, for any $t\geq -t_n$, satisfies the following energy estimate 
		\begin{equation}\label{Energy-estimate-psi}
		\frac{d}{dt}\| \vpsi(t,\cdot)\|^2_{L^2}+ 2\nu\| \vpsi(t,\cdot)\|^2_{\dot{H}^{1}} \leq 2\int_{\T} \vg(t,x)\cdot \vpsi(t,x)dx. 
		\end{equation}
	\end{Lemma}
\begin{proof} It follows from standard arguments, so we only detail the main steps. First, for a time $0<T\leq 1$, we consider the Banach space $E_T := L^\infty\big( [-t_n,-t_n+T], L^2_{\text{s}}(\T)\big) \cap L^2 \big([-t_n,-t_n+T], \dot{H}^{1}(\T)\big)$ endowed with the natural norm $\| \cdot \|_{E_T}:= \| \cdot \|_{L^\infty_t L^2_x}+ \sqrt{\nu} \| \cdot \|_{L^2_t \dot{H}^1_x}$. Thereafter, we consider the fixed-point problem for equation (\ref{Alpha-model-q}):
	\begin{equation*}
	\begin{split}
	\vpsi(t,\cdot)=&\, e^{\nu \Delta (t+t_{n})} \vpsi_0 + \int_{-t_n}^{t} e^{\nu \Delta(t-s)} \vg \, ds - \int_{-t_n}^{t}e^{\nu \Delta(t-s)}\P \left(\big(\vu_\alpha\cdot\vec{\nabla}\big)\vpsi\,\right)(s,\cdot)ds \\
	=:&\, \U_0-  \int_{-t_n}^{t}e^{\nu \Delta(t-s)}\P \left(\big(\vu_\alpha\cdot\vec{\nabla}\big)\vpsi\,\right)(s,\cdot)ds=:\mathcal{T}(\vpsi).
	\end{split}
	\end{equation*}
	
	As $\vpsi_0 \in L^2_{\text{s}}(\T)$ and $\vg \in L^2_{\text{s}}(\T)$, the first term on the right-hand side is easy to estimate by using Lemmas \ref{Lem-Tech-1} and \ref{Lem-Tech-2}. We thus have 
	\begin{equation*}
	\| \U_0 \|_{E_T} \leq C\big( \| \vpsi_0 \|_{L^2}+ \| \vg \|_{L^2} \big). 
	\end{equation*}
	Then, for a fixed  $R>0$, we define the closed ball 
	\[ B:= \left\{  \vpsi \in E_T:\, \| \vpsi - \U_0\|_{E_T} \leq R  \right\}, \]
	and we will prove that for sufficiently small $T$ the operator $\mathcal{T}(\cdot)$ defined above is contractive on $B$. 
	
	\medskip
	
	We start by verifying that $\mathcal{T}(\cdot)$ maps $B$ into $B$. Let $\vpsi \in B$. Then, using the first and the second part of Lema \ref{Lem-Tech-2} (with $\sigma=-1$), and as  $\text{div}(\vu_\alpha)=\text{div}(\vpsi)=0$, we obtain
	\begin{equation*}
	\begin{split}
	\| \mathcal{T}(\vpsi)- U_0 \|_{E_T}=\,  \left\| \int_{-t_n}^{t}e^{\nu \Delta(t-s)}\P \left(\big(\vu_\alpha\cdot\vec{\nabla}\big)\vpsi\,\right)(s,\cdot)ds \right\|_{E_T} \leq C\,  \left\| \big(\vu_\alpha\cdot\vec{\nabla}\big)\vpsi \right\|_{L^2_t \dot{H}^{-1}}
	=\, C\,   \left\| \text{div} \big( \vpsi \otimes \vu_\alpha \big) \right\|_{L^2_t \dot{H}^{-1}}.
	\end{split}
	\end{equation*}
	
	By Sobolev embeddings and denoting $\ds{\| \vu_\alpha \|_{L^\infty_t H^2_x}:= \sup_{ t \in  \R}\| \vu_\alpha(t,\cdot)\|_{H^2}}$, we have
	\begin{equation*}
	\begin{split}
	C\,   \left\| \text{div} \big( \vpsi \otimes \vu_\alpha \big) \right\|_{L^2_t \dot{H}^{-1}} \leq  &\,  C\,  \left\|  \vpsi \otimes \vu_\alpha \right\|_{L^2_t L^2} \leq C\, \| \vpsi \|_{L^2_t L^2_x}\, \| \vu_\alpha \|_{L^\infty_t L^\infty_x} \leq C\, T^{\frac{1}{2}}\, \| \vpsi \|_{L^\infty_t L^2_x}\, \| \vu_\alpha \|_{L^\infty_t H^2_x}\\
	\leq &\, C\, T^{\frac{1}{2}}\| \vpsi \|_{E_T} \, \| \vu_\alpha \|_{L^\infty_t H^2_x}. 
	\end{split}
	\end{equation*}
	
	In the last term above, setting $0<T \leq 1$ sufficiently small such that
	\[ T^{\frac{1}{2}} 
	\leq \frac{R}{C\, \left( R+ \|\vpsi \|_{L^2}+\| \vg \|_{L^2}\right)\,  \| \vu_\alpha \|_{L^\infty_t H^2_x}},  \]
	we write
	\begin{equation*}
	\begin{split}
	C\, T^{\frac{1}{2}}\| \vpsi \|_{E_T} \, \| \vu_\alpha \|_{L^\infty_t H^2_x} \leq  &\, C\,\, T^{\frac{1}{2}} \left( \| \vpsi - \U_0 \|_{E_T}+ \| \U_0 \|_{E_T}\right)\, \| \vu_\alpha \|_{L^\infty_t H^2_x} \\
	\leq &\, C\,  \, T^{\frac{1}{2}}\, \left( R+ \|\vpsi \|_{L^2}+\| \vg \|_{L^2}\right)\,  \| \vu_\alpha \|_{L^\infty_t H^2_x}\leq R. 
	\end{split}
	\end{equation*}
	
	Setting again $T$ sufficiently small such that 
	\[  T^{\frac{1}{2}} <\frac{1}{C\,\| \vu_\alpha \|_{L^\infty_t H^2_x}}, \]
	we obtain that $\mathcal{T}(\cdot): B \to B$ is contractive.  In fact, for $\vpsi_1, \vpsi_2 \in B$, following the same arguments as above, we directly write
	\begin{equation*}
	\begin{split}
	\left\| \mathcal{T}(\vpsi_1)- \mathcal{T}(\vpsi_2) \right\|_{E_T} =&\, \left\| \int_{0}^{t}  e^{\nu \Delta(t-s)}\P \left(\big(\vu_\alpha\cdot\vec{\nabla}\big)(\vpsi_1-\vpsi_2)\,\right)(s,\cdot)ds \right\|_{E_T}\\
	\leq &\, C\, T^{\frac{1}{2}}\, \left\| \vpsi_1 - \vpsi_2 \right\|_{E_T}\, \| \vu_\alpha \|_{L^\infty_t H^2_x}. 
	\end{split}
	\end{equation*}
	
	Consequently, from the Banach's contraction principle  there exists a unique (local in time) solution $\vpsi \in B \subset E_T$ for equation (\ref{Alpha-model-q}). 
	
	\medskip
	
	Now, we will prove that this solution is the unique one in the full space $E_T$. As usual, assuming that $\vpsi_1, \vpsi_2 \in E_T$ are two solutions of (\ref{Alpha-model-q}) associated with the same data $\vpsi_0$ and $\vg$, defining $\vec{\Psi}:= \vpsi_1-\vpsi_2$ we find that it solves the equation 
	\begin{equation*}
	\partial_t \vec{\Psi}-\nu\Delta \vec{\Psi}+\P\left(\big(\vu_\alpha\cdot\vec{\nabla}\big)\vec{\Psi}\,\right)=0, \quad \text{div}(\vec{\Psi})=0, \quad \vec{\Psi}(-t_n,\cdot)=0.
	\end{equation*}
	
	For the transport term,  we have  $\P\left(\big(\vu_\alpha\cdot\vec{\nabla}\big)\vec{\Psi}\,\right) \in L^2\big( [-t_n, -t_n+T], L^{\frac{6}{5}}(\T)\big)$. Indeed, using the information $\vu_\alpha \in L^\infty\big( [-t_n, -t_n+T], H^2(\T)\big)$ and $\Psi \in L^2\big([-t_n, -t_n+T], \dot{H}^{1}(\T)\big)$, applying  H\"older inequalities together with  Sobolev embeddings and Lebesgue embeddings (as we work in the torus $\T$), we write 
	\begin{equation}\label{Estim-Nonlinear}
	\left\|  \P\left(\big(\vu_\alpha\cdot\vec{\nabla}\big)\vec{\Psi}\,\right)  \right\|_{L^2_t L^{\frac{6}{5}}_x} \leq \| \vu_\alpha \|_{L^\infty_t L^\infty_x}\, \left\| \vec{\nabla} \otimes \vec{\Psi} \right\|_{L^2_t L^{\frac{6}{5}}_x}\leq C\, \| \vu_\alpha \|_{L^\infty_t H^2_x}\, \left\| \vec{\nabla} \otimes \vec{\Psi} \right\|_{L^2_t L^2_x}<+\infty. 
	\end{equation}

	With this fact and as $ \vec{\Psi} \in L^2\big([-t_n, -t_n+T], \dot{H}^1(\T)\big)\subset L^2\big([-t_n, -t_n+T], L^6(\T)\big)$, observe that  the integral $\ds{\int_{\T} \P\left(\big(\vu_\alpha\cdot\vec{\nabla}\big)\vec{\Psi}\,\right) \cdot \Psi\, dx} $ is well-defined. Moreover, from the divergence-free of $\vu_\alpha$ and $\Psi$, we obtain
	\[ \int_{\T} \P\left(\big(\vu_\alpha\cdot\vec{\nabla}\big)\vec{\Psi}\,\right) \cdot \vec{\Psi}\, dx=0. \]
	
	Therefore, through standard computations the following estimate holds:
	\begin{equation*}
	\frac{d}{dt}\| \vec{\Psi}(t,\cdot)\|^2_{L^2}+2\nu \| \vec{\Psi}(t,\cdot)\|^2_{\dot{H}^1}\leq 0,
	\end{equation*}
	hence $\vec{\Psi}=0$. 
	
	\medskip
	
	To finish the proof, note that this same arguments allows us to easily derive the desired energy control (\ref{Energy-estimate-psi}). In particular, this control enables to extend the unique solution $\vpsi$ to the entire time interval $[-t_n, +\infty[$. Lemma \ref{Lem:GWP-Unified-Equation} is now proven.  \end{proof}	

Within the framework of Lemma \ref{Lem:GWP-Unified-Equation}, for the  first system (\ref{Alpha-model-vn}), by setting $\vpsi_0=0$ and $\vg = \vf$ we obtain  a unique solution 
\[ \vv_n \in  (L^\infty_{t})_{\rm loc} (L^2_{\text{s}})_x  \cap (L^2_{t})_{\rm loc} \dot{H}^{1}_x \big( [-t_n, +\infty[ \times \T\big). \]
Similarly, for the second system (\ref{Alpha-model-wn}), setting now $\vpsi_0=\vu(-t_n, \cdot)$ and $\vg =0$, there exists a unique solution 
\[ \vw_n \in   (L^\infty_{t})_{\rm loc} (L^2_{\text{s}})_x  \cap (L^2_{t})_{\rm loc} \dot{H}^{1}_x \big( [-t_n, +\infty[ \times \T\big). \]

\emph{Second point}.   We have that $\vv_n + \vw_n$ solves the following system
\begin{equation*}
\begin{cases}\vspace{2mm}
\partial_t \big( \vv_n + \vw_n \big) -\nu\Delta \big( \vv_n + \vw_n \big)  +\P\left(\big(\vu_\alpha\cdot\vec{\nabla}\big)\big( \vv_n + \vw_n \big) \,\right)=\vf, \quad {\rm div}\left(\big( \vv_n + \vw_n \big) \right)=0, \\
\big( \vv_n + \vw_n \big) (-t_n,\cdot)=\vu(-t_n, \cdot). 
\end{cases}
\end{equation*}

On the other hand, from the uniqueness of weak Leray solutions for the following initial value problem (see Theorem \ref{Th-GWP}):
\begin{equation*}
\begin{cases}\vspace{2mm}
\partial_t \vu - \nu\Delta \vu + \P\big( (\vu_{\alpha} \cdot \vec{\nabla})\vu\big)  = \vf, \qquad {\rm div}(\vu)=0, \\ 
\vu(-t_n, \cdot)=\vu(-t_n,\cdot),
\end{cases} 
\end{equation*} 
we conclude the identity   (\ref{Decomposition}). Proposition \ref{Prop1:Decomposition-Eternal-Sol} is proven. 
\end{proof}	

\emph{Step 2: Uniform  $H^{1}-$estimate  for $\vv_n$}. Regarding the equation (\ref{Alpha-model-vn}), as the initial datum is equal to zero, we are able to prove the next uniform regularity estimate.

\begin{Proposition}\label{Prop:Tech-Unif-Estim-vn-m-0}  For any $n\in \mathbb{N}$, let $\vv_n \in  (L^\infty_{t})_{loc} (L^2_{\text{s}})_x  \cap (L^2_{t})_{loc} \dot{H}^{1}_x \big( [-t_n, +\infty[ \times \T\big)$ be the unique solution of equation (\ref{Alpha-model-vn}), obtained in the first point of Proposition \ref{Prop1:Decomposition-Eternal-Sol}. Additionally, let $t_0 \in \R$ be the fixed but arbitrary time also given in the statement of Proposition \ref{Prop1:Decomposition-Eternal-Sol}. 
	
\medskip

There exists a constant $\mathcal{K}_0>0$, depending on  $ \vf$,   $A(\cdot)$, $\varphi$, $\alpha$ and $\nu$, and which does not depend on $t_0$, such that the following uniform estimate holds:
\begin{equation}\label{Uniform-Estimate-vn-H1}
\sup_{n\in \mathbb{N}}\left(\, \sup_{t \in [t_0-1, t_0]} \| \vv_n(t,\cdot)\|^2_{H^{1}}\right) \leq \mathcal{K}_0.  
\end{equation} 
\end{Proposition}	
\begin{proof} As equation (\ref{Alpha-model-vn}) contains the term $\vu_\alpha$ in its transport term, we begin by studying higher-order regularity properties of $\vu_\alpha$. 
	
\begin{Lemma}\label{Lem-Reg-u-alpha-H1} For any $t_n>\max\{2-t_0,0\}$ we have  $\ds{\vu_\alpha \in \mathcal{C}\big( [-t_n, +\infty[, H^{3}(\T) \big)}$. 
\end{Lemma}		
\begin{proof} Let $t_n$ be fixed. Moreover, let $\varepsilon_n>0$ be fixed. Then, for a generic unknown $\vpsi$, consider the initial value problem of the coupled system:
	\begin{equation*}
	\begin{cases}\vspace{2mm}
	\partial_t \vpsi - \nu\Delta \vpsi + \P\big( (\vpsi_{\alpha} \cdot \vec{\nabla})\vpsi\big)  = \vf, \qquad {\rm div}(\vpsi)=0, \\  \vspace{2mm}
	-\alpha^2 \, \P\Big( {\rm div} \left( A (\varphi \ast \vpsi) (\vec{\nabla} \otimes \vpsi_\alpha)^T \right)\Big)+\vpsi_{\alpha} = \vpsi, \qquad {\rm div}(\vpsi_{\alpha})=0, \\
	\vpsi(-t_n-\varepsilon_n)= \vu(-t_n-\varepsilon_n),
	\end{cases} 
	\end{equation*}	
	where $\vu$ is the bounded eternal solution of (\ref{Alpha-model-eternal-attractor}). 
	
	\medskip	
	
	Applying Theorem \ref{Th-GWP} to this system we obtain a unique solution 
	\[  \vpsi \in (L^\infty_t)_{loc}(L^2_{\text{s}})_x \cap (L^2_t)_{loc}\dot{H}^1_x \big( [-t_n-\varepsilon_n, +\infty[ \times \T \big) \quad \text{and} \quad \vpsi_{\alpha} \in (L^\infty_t)_{loc}H^1_x \big(  [-t_n-\varepsilon_n, +\infty[ \times \T\big). \]
	Thereafter, as $\vf, A(\cdot)$ and $\varphi$ satisfy (\ref{Reg-force-attractor-L2}) and (\ref{conditions-a-varphi-regularity-attractor-H1}), applying now Theorem \ref{Th-Regularity} (with $m=0$,  $T_1=-t_n>-t_n-\varepsilon_n$ and any $T_2>T_1$) we have that 
	\[ \vpsi \in \mathcal{C}\big( [-t_n, +\infty[, H^{1}(\T) \big) \quad \text{and} \quad \vpsi_\alpha \in \mathcal{C}\big( [-t_n, +\infty[, H^{3}(\T) \big).  \]

	On the other hand, in the interval of time $[-t_n-\varepsilon_n, +\infty[$, observe that $\vu$ and $\vu_\alpha$ also constitute a solution of the system above with the same initial datum $\vu(-t_n-\varepsilon_n)$. Consequently, by uniqueness of solutions given again by Theorem \ref{Th-GWP}, we conclude that $\vu_\alpha = \vpsi_{\alpha}$ on $[-t_n-\varepsilon_n, +\infty[$, which implies the desired gain of regularity for $\vu_{\alpha}$.  \end{proof}

\begin{Remark}\label{Rmk-Reg-vn-H1} Once we have that $\vu_\alpha \in \mathcal{C}\big( [-t_n, +\infty[, H^{3}(\T) \big)$, and since $\vv_n(-t_n, \cdot)=0$, applying the same arguments as in the proof of Theorem \ref{Th-Regularity} (with $m=0$) to the system (\ref{Alpha-model-vn}), we obtain 
	\[ \vv_n \in \mathcal{C}\big( [-t_n, +\infty[, H^{1}(\T)\big) \cap L^2_{loc} \big( [-t_n,+\infty[, \dot{H}^{2}(\T) \big). \]
\end{Remark}

With this information at our disposal, we can now deduce the uniform bound (\ref{Uniform-Estimate-vn-H1}) as follows. Returning to equation (\ref{Alpha-model-vn}), we write
\begin{equation*}
\partial_t \vv_n  =  \nu \Delta \vv_n  - \P \big( \text{div}\big( \vv_n \otimes \vu_\alpha \big) \big)+ \vf, 
\end{equation*}
where we can verify that each term on the right-hand side belongs to the space $(L^2_t)_{\rm loc}L^2_x$.

\medskip

In fact, for the first term, as  we have  $\vv_n \in \mathcal{C}_t H^{1}_x\cap (L^2_t) \dot{H}^{2}_x$ we easily deduce that  $ \nu \Delta \vv_n  \in (L^2_t)_{\rm loc} L^2_x$. 

\medskip

For the second term, as  $\vu_\alpha \in \mathcal{C}_t H^{2}_x$ and $\vv_n \in \mathcal{C}_t H^{1}_x\cap (L^2_t) \dot{H}^{2}_x\subset (L^2_t)_{\rm loc}H^{2}_x$,  from   the product laws in Sobolev space we obtain that 
\begin{equation*}
\begin{split}
 \left\| \P\big( {\rm div}\big( \vv_n \otimes \vu_\alpha \big)\big) \right\|_{(L^2_t)_{\rm loc} H^{1}_x}\leq C\, \left\|  \vv_n \otimes \vu_\alpha  \right\|_{(L^2_t)_{\rm loc} H^{2}_x}
\leq\,  C\, \| \vu_\alpha \|_{(L^\infty_t)_{\rm loc}H^{2}_x}\, \| \vv_n\|_{(L^2_t)_{\rm loc}H^{2}_x} <+\infty. 
\end{split}
\end{equation*}

Finally, for the third term,  by assumption (\ref{Reg-force-attractor}) we have that $\vf  \in (L^2_t)_{\rm loc} L^2_x $.

\medskip

On the other hand,  as $\vv_n \in \mathcal{C}_t H^{1}_x \cap (L^2_t)_{loc} \dot{H}^{2}_x \subset (L^2_t)_{loc}H^{m+2}_x$ we have $\mathcal{J}^{2} \vv_n \in  (L^2_t)_{loc}L^2_x$. 

\medskip

Thereafter, for the Bessel potential $\mathcal{J}^2:=I_d -\Delta$ we can write
\begin{equation*}
 \int_{\T}  \left(  \partial_t \vv_n -  \nu \Delta \vv_n  + \P \big( \text{div}\big( \vv_n \otimes \vu_\alpha \big) \big)- \vf \right)\cdot \mathcal{J}^{2} \vv_n\, dx=0.
\end{equation*}
Integrating by parts the first and second term on the left-hand side, 
we obtain 
\begin{equation}\label{Estim-Der-Time-H1}
\begin{split}
&\, \frac{1}{2}\frac{d}{dt} \| \vv_n(t,\cdot)\|^2_{H^{1}}+\nu \| \vec{\nabla}\otimes \vv_n (t,\cdot)\|^2_{H^{1}}\\
=  &\, -   \int_{\T}   \P\big((\vu_\alpha\cdot\vec{\nabla})\vv_n\big) \cdot  \mathcal{J}^{2}\vv_n \, dx +  \int_{\T} \vf \cdot \mathcal{J}^{2}\vv_n \, dx,
\end{split}
\end{equation}
where we must estimate each term on the right-hand side.  

\medskip

For the first term, applying the Cauchy-Schwarz and H\"older inequality, we have
\begin{equation*}
\begin{split}
&\, \left| \int_{\T}   \P\big((\vu_\alpha\cdot\vec{\nabla})\vv_n\big) \cdot  \mathcal{J}^{2}\vv_n \, dx  \right| \leq  C\, \|  (\vu_\alpha\cdot\vec{\nabla})\vv_n\big \|_{L^2}\, \| \vv_n \|_{H^{2}}  
\leq \, C\, \| \vu_\alpha \, (\vec{\nabla} \otimes \vv_n) \|_{L^2}\, \| \vv_n \|_{H^{2}} \\
\leq C &\, \| \vu_\alpha \|_{L^\infty}\, \|  \vec{\nabla} \otimes \vv_n \|_{L^2}\, \| \vv_n \|_{H^{2}} \leq \, C\,  \| \vu_\alpha \|_{L^\infty}\, \|   \vv_n \|_{H^{1}}\, \| \vv_n \|_{H^{2}},
\end{split}
\end{equation*}
where we must control the terms $\| \vu_\alpha \|_{L^\infty}$ and $\| \vv_n \|_{H^{2}}$. 

\medskip

For $\| \vu_\alpha \|_{L^\infty}$,  from Sobolev embeddings and estimate (\ref{Control-H2-Filter}), for any time $t \in \R$ we have
\begin{equation*}
\| \vu_\alpha(t, \cdot)\|_{L^\infty} \leq C\, \| \vu_\alpha(t,\cdot)\|_{H^2} \leq K_2 \big( \| \vu(t,\cdot)\|^2_{L^1}+\| \vu(t,\cdot)\|_{L^2}\big),
\end{equation*}
with a constant $K_2>0$ depending on $A(\cdot), \varphi$ and $\alpha$.  Additionally, since $\vu \in \mathcal{T}(\mathcal{A})$ (see Definition \ref{Def-Complete-Trajectory}) and the global attractor $\mathcal{A}(\vf)$ is contained in the absorbing ball defined in (\ref{Absorbing-Set}), from the radius  $\mathcal{R}>0$ given in expression (\ref{R}) we obtain the uniform-in-time bound:
\begin{equation*}
\| \vu_\alpha(t, \cdot)\|_{L^\infty}  \leq K_2 \big(  \mathcal{R}^2+\mathcal{R}\big)=:L_0.  
\end{equation*}

 For $\| \vv_n \|_{H^{2}}$, using the Poincar\'e's inequality   we have
\begin{equation*}
\| \vv_n \|_{H^{2}} = \, \| \vv_n \|_{L^2}+\|  \vv_n  \|_{\dot{H}^{2}}\leq \frac{L}{2\pi} \| \vec{\nabla} \otimes \vv_n \|_{L^2}+ \| \vec{\nabla} \otimes \vv_n \|_{\dot{H}^{1}}  \leq \, C\, \| \vec{\nabla} \otimes \vv_n \|_{H^{1}}. 
\end{equation*}

Gathering these estimates, and using the discrete Young inequalities (with $1=\frac{1}{2}+\frac{1}{2}$), we obtain
\begin{equation}\label{Estim-Der-Time-H1-1}
\begin{split}
&\left| \int_{\T}   \P\big((\vu_\alpha\cdot\vec{\nabla})\vv_n\big) \cdot  \mathcal{J}^{2}\vv_n \, dx  \right|  \leq  \, C\,L_0 \| \vv_n \|_{H^1}\, \| \vec{\nabla}\otimes \vv_n \|_{H^1}  \\
\leq  &\, C\,\frac{L^2_0}{\nu}\, \| \vv_n \|^2_{H^{1}}+ \frac{\nu}{4}\, \| \vec{\nabla}\otimes \vv_n \|^2_{H^{1}} =: \mathfrak{C}_0\,  \| \vv_n \|^2_{H^{1}}+ \frac{\nu}{4}\, \| \vec{\nabla}\otimes \vv_n \|^2_{H^{1}}.
\end{split}.
\end{equation}

\medskip

For the second term, applying the Cauchy-Schwarz inequality together with the estimate  above on $\| \vv_n \|_{H^{2}}$ and the same discrete young inequalities, we write
\begin{equation}\label{Estim-Der-Time-H1-2}
\begin{split}
 \left| \int_{\T} \vf \cdot \mathcal{J}^{2} \vv_n\, dx \right| \leq &\, \| \vf \|_{L^2}\, \| \vv_n \|_{H^{2}} \leq C\, \| \vf \|_{L^2}\, \| \vec{\nabla}\otimes \vv_n \|_{H^{1}} \\
 \leq &\,C \frac{\| \vf \|^2_{L^2}}{\nu}  + \frac{\nu}{4}\, \| \vec{\nabla} \otimes \vv_n \|_{H^{1}}. 
\end{split}
\end{equation}

Collecting estimates (\ref{Estim-Der-Time-H1-1}) and (\ref{Estim-Der-Time-H1-2}) into (\ref{Estim-Der-Time-H1}), we obtain
\begin{equation*}
\frac{1}{2} \frac{d}{dt}\| \vv_n(t,\cdot)\|^2_{H^{1}} + \frac{\nu}{2}\| \vec{\nabla} \otimes \vv_n \|^2_{H^{1}} \leq  \mathfrak{C}_0\,  \| \vv_n \|^2_{H^{1}}+ C\frac{\| \vf \|_{L^2}}{\nu},
\end{equation*}
hence,
\begin{equation}\label{Gronwall-vn-H1}
\frac{d}{dt}\| \vv_n(t,\cdot)\|^2_{H^{1}} \leq \mathfrak{C}_0\,  \| \vv_n \|^2_{H^{1}}+ C\frac{\| \vf \|_{L^2}}{\nu}. 
\end{equation}

At this point, we need the following technical lemmas.
\begin{Lemma}\label{Lem:Control-vn-tau-n-H1} There exists a constant $M_0>0$, which only depends on $\|\vf\|_{\dot{H}^{-1}}$  and $\nu$,  such that for $T>-t_n$ and any $n \in \mathbb{N}$ the following uniform control holds:
\begin{equation}\label{Control-vn-tau-n-H1}
\int_{T}^{T+1}\| \vv_n(t,\cdot)\|^2_{H^1}\, dt \leq M_0. 
\end{equation}	
\end{Lemma}	
\begin{proof} First, we will verify the following uniform bound:
\begin{equation}\label{Uniform-Boun-vn-L2}
\sup_{n\in \mathbb{N}} \left( \sup_{t \geq - t_n} \| \vv_n(t,\cdot)\|^2_{L^2} \right) \leq C\frac{ \| \vf \|^2_{\dot{H}^{-1}}}{\nu^2}. 
\end{equation}

\medskip 
	
In fact, from the energy estimate (\ref{Energy-estimate-psi}) (with $\vec{\Psi}=\vv_n$ and $\vec{g}=\vf$),  we can write
\begin{equation*}
\frac{d}{dt}\| \vv(t,\cdot)\|^2_{L^2}+2\nu \| \vv_n(t,\cdot)\|^2_{\dot{H}^1} \leq 2 \int_{\T} \vf \cdot \vv_n(t,\cdot) dx,
\end{equation*}	
where the term on the right-hand side is estimated as
\begin{equation*}
\left| 2 \int_{\T} \vf \cdot \vv_n(t,\cdot) dx \right| \leq C \| \vf \|_{\dot{H}^{-1}}\, \| \vv_n(t,\cdot)\|_{\dot{H}^1} \leq C \frac{\| \vf \|^2_{\dot{H}^{-1}}}{\nu}+ \nu \| \vv_n(t,\cdot)\|^2_{\dot{H}^1}.
\end{equation*}

Then, we obtain
\begin{equation}\label{Gronwall-vn-L2}
\frac{d}{dt}\| \vv(t,\cdot)\|^2_{L^2}+\nu \| \vv_n(t,\cdot)\|^2_{\dot{H}^1} \leq C\frac{\| \vf \|^2_{\dot{H}^{-1}}}{\nu}.
\end{equation}
Using the Poincar\'e's inequality, we have
\[  \| \vv_n \|_{L^2} \leq \frac{L}{2\pi}\, \|  \vv_n \|_{\dot{H}^1}, \]
hence,
\begin{equation*}
-\nu \| \vv_n \|^2_{\dot{H}^1} \leq  -\frac{L^2\, \nu}{4\pi^2}\, \| \vv_n \|^2_{L^2}=:- C\nu \, \| \vv_n \|^2_{L^2}.
\end{equation*}

With this estimate, returning to the  differential inequality (\ref{Gronwall-vn-L2}) it follows that
\begin{equation*}
\frac{d}{dt}\| \vv(t,\cdot)\|^2_{L^2} \leq - C \nu  \, \| \vv_n \|^2_{L^2}+C \frac{\| \vf \|^2_{\dot{H}^{-1}}}{\nu}.
\end{equation*}
Applying the Gr\"onwall lemma and the fact that $\vv_n(-t_n, \cdot)=0$, for any $t\geq -t_n$ we have
\begin{equation*}
\| \vv_n(t,\cdot)\|^2_{L^2} \leq C\,\frac{\| \vf \|^2_{\dot{H}^{-1}}}{\nu^2 }  \left(1-e^{-C\nu (t+t_n)}\right) \leq  C\frac{\| \vf \|^2_{\dot{H}^{-1}}}{\nu^2} .
\end{equation*}
Finally, since $\ds{C\,\frac{\| \vf \|^2_{\dot{H}^{-1}}}{\nu^2} }$ does not depend on $n$, we obtain the uniform bound (\ref{Uniform-Boun-vn-L2}). 

\medskip

With estimate (\ref{Uniform-Boun-vn-L2}) at our disposal, we can now verify (\ref{Control-vn-tau-n-H1}) as follows.  Returning to the differential inequality (\ref{Gronwall-vn-L2}), we integrate over the interval $[T,T+1]$ to obtain
\begin{equation*}
\| \vv_n(T+1,\cdot)\|^2_{L^2} - \| \vv_n(T,\cdot)\|^2_{L^2} + \nu \int_{T}^{T+1}\| \vv_n(t,\cdot)\|^2_{\dot{H}^1} dt \leq C \frac{\| \vf \|^2_{\dot{H}^{-1}}}{\nu}.
\end{equation*}
Since $\| \vv_n(T+1,\cdot)\|^2_{L^2} \geq 0$ and using (\ref{Uniform-Boun-vn-L2})  (where $T\geq -t_n$), we have
\begin{equation*}
\nu \int_{T}^{T+1}\| \vv_n(t,\cdot)\|^2_{\dot{H}^1} dt  \leq \| \vv_n(T,\cdot)\|^2_{L^2} +  C\frac{\| \vf \|^2_{\dot{H}^{-1}} }{\nu} \leq C\frac{\| \vf \|^2_{\dot{H}^{-1}}}{\nu^2} + C \frac{\| \vf \|^2_{\dot{H}^{-1}}}{\nu}= C \left(\frac{1}{\nu^2}+\frac{1}{\nu}\right)\| \vf \|^2_{\dot{H}^{-1}},
\end{equation*}
hence, 
\begin{equation*}
\int_{T}^{T+1}\| \vv_n(t,\cdot)\|^2_{\dot{H}^1} \leq C \left(\frac{1}{\nu^3}+\frac{1}{\nu^2}\right)\| \vf \|^2_{\dot{H}^{-1}}.  
\end{equation*}

On the other hand, always from (\ref{Uniform-Boun-vn-L2}) we have 
\[  \int_{T}^{T+1} \| \vv_n(t,\cdot)\|^2_{L^2} dt \leq C \frac{\| \vf \|^2_{\dot{H}^{-1}}}{\nu^2}. \]

Gathering these estimates, we finally write
\[ \int_{T}^{T+1} \| \vv_n(t,\cdot)\|^2_{H^1} dt \leq  \int_{T}^{T+1} \| \vv_n(t,\cdot)\|^2_{L^2} dt  + \int_{T}^{T+1} \| \vv_n(t,\cdot)\|^2_{\dot{H}^1} dt \leq  C \left(\frac{1}{\nu^3}+\frac{1}{\nu^2}\right)\| \vf \|^2_{\dot{H}^{-1}}=: M_0, \]
which is the desired estimate (\ref{Control-vn-tau-n-H1}). 
\end{proof}	

Estimate (\ref{Control-vn-tau-n-H1}) allows us to verify the following key fact.
\begin{Lemma} Let $t_0 \in \R$ be the fixed but arbitrary time given in the statement of Proposition \ref{Prop1:Decomposition-Eternal-Sol}. Additionally, let $M_0>0$ be given in Lemma \ref{Lem:Control-vn-tau-n-H1}. Then, for any $n\in \mathbb{N}$ there exists a time $\tau_n \in [t_0-2, t_0-1]$ such that 
\begin{equation}\label{Control-vn-tau-pointwise}
\| \vv_n(\tau_n, \cdot)\|^2_{H^1} \leq M_0. 
\end{equation}
\end{Lemma}	 
\begin{Remark}\label{Rmk-tn-t0} Recall that we have assumed $t_n>2-t_0$. Thus, $-t_n<t_0-2\leq \tau_n$. Since $\vv_n$ is defined on the time interval $[-t_n,+\infty[$, the expression $\vv_n(\tau_n,\cdot)$ is therefore well defined.
\end{Remark}	
\begin{proof} By contradiction, assume that there exists $n \in \mathbb{N}$ such that for any $\tau \in [t_0-2, t_0-1]$ we have 
	$\| \vv_n(\tau, \cdot)\|^2_{H^1} > M_0$. Then, within the framework of Lemma \ref{Lem:Control-vn-tau-n-H1} we set  $T:=t_0-2$. Then we can write
	\[ \int_{t_0-2}^{t_0-1} \| \vv_n(\tau, \cdot)\|^2_{H^1} d \tau > \int_{t_0-2}^{t_0-1} M_0\, d \tau =M_0, \]
	which directly contradicts the estimate (\ref{Control-vn-tau-n-H1})   \end{proof}

With estimate (\ref{Control-vn-tau-pointwise}), we deduce now the uniform bound (\ref{Uniform-Estimate-vn-H1}) as follows.  Returning to the differential inequality (\ref{Gronwall-vn-H1}), we apply the Gr\"onwall lemma over the interval of time $[\tau_n, t_0+\varepsilon]$ for some $\varepsilon>0$. Here, recall that we have $\tau_n\leq t_0-1<t_0+\varepsilon$. Then,  for any $t \in [\tau_n, t_0+\varepsilon]$ we obtain
\begin{equation*}
\| \vv_n(t,\cdot)\|^2_{H^1} \leq \| \vv_n(\tau_n, \cdot)\|^2_{H^1}\, e^{\mathfrak{C}_0(t-\tau_n)}+ C\frac{\| \vf \|^2_{L^2}}{\mathfrak{C}_0\,\nu}  \left( 1- e^{\mathfrak{C}_0 (t-\tau_n)} \right) \leq \| \vv_n(\tau_n, \cdot)\|^2_{H^1}\, e^{\mathfrak{C}_0(t-\tau_n)}+ C\frac{\| \vf \|^2_{L^2}}{\mathfrak{C}_0\,\nu}. 
\end{equation*}

To control we first term on the right-hand side, on the one hand recall that the term $\| \vv_n(\tau_n, \cdot)\|^2_{H^1}$ is uniform controlled in (\ref{Control-vn-tau-pointwise}). On the other hand, as $t_0-2 \leq  \tau_n$ and $t\leq t_0+\varepsilon$ we have $0\leq t-\tau_n \leq 2+\varepsilon$, where $\varepsilon$ is independent of $n$ and $t_0$. 

\medskip

We thus obtain
\begin{equation}\label{K0}
\| \vv_n(t,\cdot)\|^2_{H^1} \leq M_0 \, e^{\mathfrak{C}_0(2+\varepsilon)}+C\frac{\| \vf \|^2_{L^2}}{\mathfrak{C}_0\,\nu}=: \mathcal{K}_0. 
\end{equation}

Finally, note that  we have $[t_0-1, t_0] \subset [\tau_n, t_0+\varepsilon]$. Additionally, note that the constant $\mathcal{K}_0$ defined above is also  independent of $n$ and $t_0$.  Then, we can write
\[ \sup_{n \in \mathbb{N}} \left( \sup_{t \in [t_0-1, t_0]} \| \vv_n(t,\cdot)\|^2_{H^1} \right) \leq \mathcal{K}_0, \]
which is the desired estimate (\ref{Uniform-Estimate-vn-H1}). Proposition \ref{Prop:Tech-Unif-Estim-vn-m-0} is thus proven. 
\end{proof}	

\emph{Step 3: Convergence to zero for $\vw_n$}. Regarding the equation (\ref{Alpha-model-wn}), as the external force is equal to zero, we have the following convergence.

\begin{Proposition}\label{Prop2:Vanishing-wn} Let $t_0 \in \R$ be the fixed but arbitrary time given in the statement of Proposition \ref{Prop:Tech-Unif-Estim-vn-m-0}. Additionally, for any $n\in \mathbb{N}$, let $\vw_n \in  (L^\infty_{t})_{\rm loc} (L^2_{\text{s}})_x  \cap (L^2_{t})_{\rm loc} \dot{H}^{1}_x \big( [-t_n, +\infty[ \times \T\big)$ be the unique solution of equation (\ref{Alpha-model-wn}), obtained in the first point of Proposition \ref{Prop1:Decomposition-Eternal-Sol}.
	
\medskip
	
Consider the family of functions $(\vw_n)_{n \in \mathbb{N}}$.  Then,  we have that $\| \vw_n(t_0,\cdot)\|_{L^2}\to 0$ as $n\to +\infty$.
 \end{Proposition}	
\begin{proof} In  the energy estimate (\ref{Energy-estimate-psi}), we setting  $\vec{\Psi}=\vw_n$ and $\vg=0$, for any $n\in \mathbb{N}$ and $t\geq -t_n$ we obtain that 
	\begin{equation*}
	\frac{d}{dt} \| \vw_n(t,\cdot)\|^2_{L^2}+ 2\nu \| \vw_n(t,\cdot)\|_{\dot{H}^1}\leq 0. 
	 \end{equation*}

From the Poincar\'e's inequality, we have 
\[ \| \vu(t,\cdot)\|_{L^2}\leq \frac{L}{2\pi}\| \vu(t,\cdot)\|_{\dot{H}^1},\]
and setting $\eta:= \frac{8\pi^2 \nu}{L^2}>0$   we obtain 
\begin{equation*}
	\frac{d}{dt} \| \vw_n(t,\cdot)\|^2_{L^2} \leq - \eta \| \vw_n(t,\cdot)\|^2_{L^2}.
\end{equation*} 
Then, for any $t\geq -t_n$,   Gr\"onwall's lemma and the fact that $\vw_n(-t_n,\cdot)=\vu(-t_n,\cdot)$ gives:
\begin{equation*}
\| \vw_n (t,\cdot)\|^2_{L^2} \leq \| \vu(-t_n,\cdot)\|^2_{L^2}\, e^{-\eta(t+t_n)}. 
\end{equation*}

Finally, as $\vu \in L^\infty\big( \R, L^2_{\text{s}}(\T)\big)$ we denote $\| \vu \|_{L^\infty_t L^2_x}:=\ds{\sup_{t\in \R} \| \vu(t,\cdot)\|_{L^2}}$. Then  we can write
\begin{equation*}
\| \vw_n (t,\cdot)\|^2_{L^2} \leq \| \vu \|_{L^\infty_t L^2_x}\, e^{-\eta(t+t_n)},
\end{equation*}
from which the desired convergence follows by setting $t=t_0$ and recalling that $t_n \to +\infty$ and $n\to +\infty$. Proposition \ref{Prop2:Vanishing-wn} is proven.  \end{proof}		

\emph{Step 4: En of the proof of Theorem \ref{Th-Regularity-Atrtractor}}. On the one hand,  from the uniform estimate (\ref{Uniform-Estimate-vn-H1}) obtained in Proposition \ref{Prop:Tech-Unif-Estim-vn-m-0},  	by setting $t=t_0$ we obtain  
\[  \sup_{n\in \mathbb{N}} \| \vv_n(t_0,\cdot) \|^2_{H^{1}} \leq \mathcal{K}_0. \]
Then, there exists $\vv \in H^{1}(\T)$ such that $\vv_n(t_0,\cdot)\to \vv$, as $n\to +\infty$, in the weak topology of the space $H^{1}(\T)$. 

\medskip

On the other hand,  in the identity (\ref{Decomposition}) proven in Proposition \ref{Prop1:Decomposition-Eternal-Sol}, by setting $t=t_0$  for any $n\in \mathbb{N}$ we write
\[ \vu(t_0,\cdot)-\vv_n(t_0,\cdot)= \vw_n(t_0,\cdot). \]
Therefore, from the convergence property obtained in Proposition \ref{Prop2:Vanishing-wn} we obtain that $\vv_n(t_0,\cdot)\to \vu(t_0,\cdot)$, as $n\to +\infty$, in the strong topology of the space $L^2(\T)$. 

\medskip

We thus conclude that $\vu(t_0,\cdot)=\vv$. Moreover, from the convergence in the weak topology of $H^{1}(\T)$, we obtain that 
\begin{equation}\label{Uniform-Control-u-t0}
\| \vu(t_0,\cdot)\|^2_{H^{1}} \leq \liminf_{n\to +\infty}\| \vv_n(t_0,\cdot)\|^2_{H^{1}} \leq \mathcal{K}_0. 
\end{equation}

Finally, since $\mathcal{K}_0$ (defined in expression (\ref{K0})) does not depend on $t_0$, and since $\vu \in \mathcal{T}(\mathcal{A})$ and $t_0\in \R$ were arbitrary,  we obtain the desired uniform estimate (\ref{Regularity-Atrractor}). Theorem \ref{Th-Regularity-Atrtractor} is proven.

\subsection{Proof of Proposition \ref{Prop-Hihger-Regularity-Atrtractor}}  For a fixed $m\geq 1$, assume that the external force $\vf$ satisfies (\ref{Reg-force-attractor}) and that the functions $A(\cdot)$ and $\varphi$ in the filter equation (\ref{Filter-Equation}) satisfy (\ref{conditions-a-varphi-regularity-attractor}).

\medskip

Let $(\vu, \vu_\alpha) \in \mathcal{T}(\mathcal{A}) \times L^\infty\big( \R, H^2(\T)\big)$ be an arbitrary solution of the coupled system (\ref{Alpha-model-eternal-attractor}). The proof of estimate (\ref{Regularity-Atrractor-H-m+1}) follows from the same four steps detailed in the proof of Theorem \ref{Th-Regularity-Atrtractor}.

\medskip

Therefore, we focus solely on the second step, in which we obtain uniform higher-order regularity estimates for solutions $\vv_n$ of equation (\ref{Alpha-model-vn}), as stated in Lemma \ref{Lem:Tech-Unif-Estim-vn-H-m+1} below. 

\medskip

To obtain these uniform higher-order regularity estimates for $\vv_n$, we first need some regularity properties of $\vu_\alpha$ and $\vv_n$ in equation (\ref{Alpha-model-vn}).

\medskip
 
\begin{Lemma}\label{Lem-Reg-u-alpha-Hm} For $m \geq 1$ assume (\ref{Reg-force-attractor}) and (\ref{conditions-a-varphi-regularity-attractor}). Let $t_0 \in \R$ be the fixed time given in the statement of Proposition \ref{Prop1:Decomposition-Eternal-Sol} and, for any $n\in \mathbb{N}$, let $t_n>\max{ 2-t_0, 0 }$ be the sequence of times also given in the statement of Proposition \ref{Prop1:Decomposition-Eternal-Sol}. The following statements hold.
	\begin{enumerate}
		\item For any $n\in \mathbb{N}$, we have 
		\[\vu_\alpha \in \mathcal{C}\big( [-t_n, +\infty[, H^{m+3}(\T) \big).\]
		
		\item For any $n\in \mathbb{N}$, let
		\[\vv_n \in  L^\infty_{\rm loc} \big( [-t_n, +\infty], L^2_{\rm s}(\T)\big) \cap L^2_{\rm loc} \big([-t_n, +\infty], \dot{H}^1(\T)\big),\]
		be the unique solution of equation (\ref{Alpha-model-vn}), obtained in the first point of Proposition \ref{Prop1:Decomposition-Eternal-Sol}. Then, we have
		\[ \vv_n \in \mathcal{C}\big( [-t_n, +\infty[, H^{m+1}(\T)\big) \cap L^2_{loc} \big( [-t_n,+\infty[, \dot{H}^{m+2}(\T) \big).  \]
	\end{enumerate}
\end{Lemma}

\begin{proof} \emph{First point}. Let $t_n$ be fixed. Moreover, let $\varepsilon_n>0$ be fixed. Then, for a generic unknown $\vpsi$, consider the initial value problem for the coupled system:
	\begin{equation*}
	\begin{cases}\vspace{2mm}
	\partial_t \vpsi - \nu\Delta \vpsi + \P\big( (\vpsi_{\alpha} \cdot \vec{\nabla})\vpsi\big)  = \vf, \qquad {\rm div}(\vpsi)=0, \\  \vspace{2mm}
	-\alpha^2 , \P\Big( {\rm div} \left( A (\varphi \ast \vpsi) (\vec{\nabla} \otimes \vpsi_\alpha)^T \right)\Big)+\vpsi_{\alpha} = \vpsi, \qquad {\rm div}(\vpsi_{\alpha})=0, \\
	\vpsi(-t_n-\varepsilon_n)= \vu(-t_n-\varepsilon_n),
	\end{cases}
	\end{equation*}
	where $\vu \in \mathcal{T}(\mathcal{A})$ is the bounded eternal solution of (\ref{Alpha-model-eternal-attractor}).

	\medskip	
	
	Applying Theorem \ref{Th-GWP} to this system we obtain a unique solution 
	\[  \vpsi \in (L^\infty_t)_{\rm loc}(L^2_{\text{s}})_x \cap (L^2_t)_{\rm loc}\dot{H}^1_x \big( [-t_n-\varepsilon_n, +\infty[ \times \T \big) \quad \text{and} \quad \vpsi_{\alpha} \in (L^\infty_t)_{\rm loc}H^1_x \big(  [-t_n-\varepsilon_n, +\infty[ \times \T\big). \]
	Thereafter, as $\vf, A(\cdot)$ and $\varphi$ satisfy (\ref{Reg-force-attractor}) and (\ref{conditions-a-varphi-regularity-attractor}), applying now Theorem \ref{Th-Regularity} (with $T_1=-t_n>-t_n-\varepsilon_n$ and any $T_2>T_1$) we have that 
	\[ \vpsi \in \mathcal{C}\big( [-t_n, +\infty[, H^{m+1}(\T) \big) \quad \text{and} \quad \vpsi_\alpha \in \mathcal{C}\big( [-t_n, +\infty[, H^{m+3}(\T) \big).  \]

	On the other hand, in the interval of time $[-t_n-\varepsilon_n, +\infty[$, observe that $\vu$ and $\vu_\alpha$ also constitute a solution of the system above with the same initial datum $\vu(-t_n-\varepsilon_n)$. Consequently, by uniqueness of solutions given again by Theorem \ref{Th-GWP}, we conclude that $\vu_\alpha = \vpsi_{\alpha}$ on $[-t_n-\varepsilon_n, +\infty[$, which implies the desired gain of regularity for $\vu_{\alpha}$. 
	
	\medskip
	
\emph{Second point}. 	Once we have that $\vu_\alpha \in \mathcal{C}\big( [-t_n, +\infty[, H^{m+2}(\T) \big)$, and since $\vv_n(-t_n, \cdot)=0$, applying the same arguments as in the proof of Theorem \ref{Th-Regularity} to the system (\ref{Alpha-model-vn}), we obtain 
	\[ \vv_n \in \mathcal{C}\big( [-t_n, +\infty[, H^{m+1}(\T)\big) \cap L^2_{\rm loc} \big( [-t_n,+\infty[, \dot{H}^{m+2}(\T) \big). \]
Lemma \ref{Lem-Reg-u-alpha-Hm} is proven. 	
\end{proof}

These regularity properties of $\vu_\alpha$ and $\vv_n$ now allow us to verify the following $H^{m+1}-$energy estimate. In particular, note that, since $\vu_\alpha \in  \mathcal{C}_t H^{m+2}_x $ by Sobolev embeddings, we have  $\vu_\alpha \in \mathcal{C}_t W^{m,\infty}_x$.

\begin{Lemma}\label{Lem-Gronwall-vn-Hm}  Under the same framework as in Lemma \ref{Lem-Reg-u-alpha-Hm}, for any $t \geq -t_n$  the following estimate holds:
	\begin{equation}\label{Gronwall-vn-Hm}
	\frac{d}{dt}\| \vv_n(t,\cdot)\|^2_{H^{m+1}} + \nu\| \vec{\nabla} \otimes \vv_n(t,\cdot) \|^2_{H^{m+1}} \leq  \frac{C}{\nu} \| \vu_\alpha(t,\cdot) \|^2_{W^{m,\infty}}\, \| \vv_n(t,\cdot) \|^2_{H^{m+1}}+ C \frac{\| \vf \|^2_{H^m}}{\nu}. 
	\end{equation}
\end{Lemma}	
\begin{proof} As $\vv_n$ satisfies the equation 
	\begin{equation*}
	\partial_t \vv_n  =  \nu \Delta \vv_n  - \P \big( \text{div}\big( \vv_n \otimes \vu_\alpha \big) \big)+ \vf, 
	\end{equation*}
	we begin by verifying that each term on the right-hand side belongs to $(L^2_t)_{loc} H^m_x$. The first and third terms can be estimated directly.  For the second term, as  $\vu_\alpha \in  (L^\infty_t)_{\rm loc}H^{m+2}_x$ and $\vv_n \in \mathcal{C}_t H^{m+1}_x\cap (L^2_t) \dot{H}^{m+2}_x\subset (L^2_t)_{\rm loc}H^{m+2}_x$, since $m+2> \frac{3}{2}$, from   the product laws in Sobolev spaces we obtain that 
	\begin{equation*}
	\begin{split}
	\left\| \P\big( \text{div}\big( \vv_n \otimes \vu_\alpha \big)\big) \right\|_{(L^2_t)_{loc} H^{m}_x} \leq &\, \left\| \P\big( \text{div}\big( \vv_n \otimes \vu_\alpha \big)\big) \right\|_{(L^2_t)_{loc} H^{m+1}_x}\leq \,  C\,  \left\|  \vv_n \otimes \vu_\alpha  \right\|_{(L^2_t)_{loc} H^{m+2}_x} \\
	\leq &\,  C\, \| \vu_\alpha \|_{(L^\infty_t)_{loc}H^{m+2}_x}\, \| \vv_n\|_{(L^2_t)_{loc}H^{m+2}_x} <+\infty. 
	\end{split}
	\end{equation*}
	
With this information, returning to the last equation and applying the Bessel potential $ \mathcal{J}^m:=(I_d-\Delta)^{\frac{m}{2}}$ (given by expression (\ref{Bessel-Potental})) to each term, we write
	\begin{equation}\label{Identity-Tech-Jm}
	\partial_t \mathcal{J}^m \vv_n  =  \nu \Delta \mathcal{J}^m\vv_n  - \mathcal{J}^m \P \big( \text{div}\big( \vv_n \otimes \vu_\alpha \big) \big)+\mathcal{J}^m\vf, 
	\end{equation}
	where each term belongs to $(L^2_t)_{loc}L^2_x$. 
	
	\medskip
	
	On the other hand,  as $\vv_n \in \mathcal{C}_t H^{m+1}_x \cap (L^2_t)_{loc} \dot{H}^{m+2}_x \subset (L^2_t)_{loc}H^{m+2}_x$ we have $\mathcal{J}^{m+2} \vv_n \in  (L^2_t)_{loc}L^2_x$. Thereafter, we can write
	\begin{equation*}
	\int_{\T}  \left(  \partial_t \mathcal{J}^m \vv_n -  \nu \Delta \mathcal{J}^m\vv_n  + \mathcal{J}^m\P \big( \text{div}\big( \vv_n \otimes \vu_\alpha \big) \big)-\mathcal{J}^m \vf \right)\cdot \mathcal{J}^{m+2} \vv_n\, dx=0.
	\end{equation*}
	Integrating by parts the first and second term on the left-hand side, 
	we obtain 
	\begin{equation}\label{Estim-Der-Time-Hm}
	\begin{split}
	&\, \frac{1}{2}\frac{d}{dt} \| \vv_n(t,\cdot)\|^2_{H^{m+1}}+\nu \| \vec{\nabla}\otimes \vv_n (t,\cdot)\|^2_{H^{m+1}}\\
	=  &\, -   \int_{\T}  \mathcal{J}^m \P\big((\vu_\alpha\cdot\vec{\nabla})\vv_n\big) \cdot  \mathcal{J}^{m+2}\vv_n \, dx +  \int_{\T} \mathcal{J}^m\vf \cdot \mathcal{J}^{m+2}\vv_n \, dx,
	\end{split}
	\end{equation}
	where we must estimate each term on the right-hand side.  
	
	\medskip
	
For the first term, applying the Cauchy-Schwarz inequality and the Moser-type inequality\footnote{For any $m\in \mathbb{N}$ we have $\| fg \|_{H^m}\leq C \| f \|_{W^{m,\infty}}\, \| g \|_{H^m}$.} together with Sobolev embeddings, we have
	\begin{equation*}
	\begin{split}
	&\, \left| \int_{\T}  \mathcal{J}^m \P\big((\vu_\alpha\cdot\vec{\nabla})\vv_n\big) \cdot  \mathcal{J}^{m+2}\vv_n \, dx  \right| \leq  C\, \|  (\vu_\alpha\cdot\vec{\nabla})\vv_n\big \|_{H^m}\, \| \vv_n \|_{H^{m+2}}  \\
	\leq &\, C\, \| \vu_\alpha \, (\vec{\nabla} \otimes \vv_n) \|_{H^m}\, \| \vv_n \|_{H^{m+2}} \leq C\, \| \vu_\alpha \|_{W^{m,\infty}}\, \|  \vec{\nabla} \otimes \vv_n \|_{H^m}\, \| \vv_n \|_{H^{m+2}}\\
	\leq &\, C\,  \| \vu_\alpha \|_{W^{m,\infty}}\, \|   \vv_n \|_{H^{m+1}}\, \| \vv_n \|_{H^{m+2}}.
	\end{split}
	\end{equation*}
	
	Here, for the  term $\| \vv_n \|_{H^{m+2}}$,   using the Poincar\'e's inequality    we obtain
	\begin{equation*}
	\| \vv_n \|_{H^{m+2}} = \, \| \vv_n \|_{L^2}+\|  \vv_n  \|_{\dot{H}^{m+2}}\leq \frac{L}{2\pi} \| \vec{\nabla} \otimes \vv_n \|_{L^2}+ \| \vec{\nabla} \otimes \vv_n \|_{\dot{H}^{m+1}}  \leq \, C\, \| \vec{\nabla} \otimes \vv_n \|_{H^{m+1}}. 
	\end{equation*}
	
	Gathering these estimates and	applying the discrete Young inequalities (with $1=\frac{1}{2}+\frac{1}{2}$), we write
	\begin{equation}\label{Estim-Der-Time-Hm-1}
	\begin{split}
	&\left| \int_{\T}  \mathcal{J}^m \P\big((\vu_\alpha\cdot\vec{\nabla})\vv_n\big) \cdot  \mathcal{J}^{m+2}\vv_n \, dx  \right|  \leq  \, C \| \vu_\alpha \|_{H^{m+2}}\, \|   \vv_n \|_{H^{m+1}} \| \vec{\nabla}\otimes \vv_n \|_{H^{m+1}}  \\
	\leq  &\, \frac{C}{\nu} \| \vu_\alpha \|^2_{W^{m,\infty}}\, \| \vv_n \|^2_{H^{m+1}}+ \frac{\nu}{4}\, \| \vec{\nabla}\otimes \vv_n \|^2_{H^{m+1}}. 
	\end{split}
	\end{equation}
	
For the second term, applying the Cauchy-Schwarz inequality together with the above estimate for $\| \vv_n \|_{H^{m+2}}$ and the same discrete Young inequalities, we write
	\begin{equation}\label{Estim-Der-Time-Hm-2}
	\begin{split}
	\left| \int_{\T} \mathcal{J}^m \vf \cdot \mathcal{J}^{m+2} \vv_n\, dx \right| \leq &\, \| \vf \|_{H^m}\, \| \vv_n \|_{H^{m+2}} \leq C\, \| \vf \|_{H^m}\, \| \vec{\nabla}\otimes \vv_n \|_{H^{m+1}} \\
	\leq &\, \frac{C}{\nu} \| \vf \|^2_{H^m} + \frac{\nu}{4}\, \| \vec{\nabla} \otimes \vv_n \|^2_{H^{m+1}}. 
	\end{split}
	\end{equation}
	
	Collecting estimates (\ref{Estim-Der-Time-Hm-1}) and (\ref{Estim-Der-Time-Hm-2}) in (\ref{Estim-Der-Time-Hm}), we obtain the desired estimate (\ref{Gronwall-vn-Hm}).  Lemma \ref{Lem-Gronwall-vn-Hm} is now proven. 
\end{proof}	

Once we have estimate (\ref{Gronwall-vn-Hm}) at our disposal, we are able to prove the uniform estimate (\ref{Regularity-Atrractor-H-m+1}). As mentioned, we follow the same steps as in the proof of Theorem \ref{Th-Regularity-Atrtractor}. Consequently, for simplicity, it is sufficient to focus on the following technical lemma, which is analogous to Proposition \ref{Prop:Tech-Unif-Estim-vn-m-0} when $m=0$, and this is the only step that differs from the corresponding argument in the proof of Theorem \ref{Th-Regularity-Atrtractor}.

\begin{Lemma}\label{Lem:Tech-Unif-Estim-vn-H-m+1} Under the same framework as in Lemma \ref{Lem-Reg-u-alpha-Hm}, assume the uniform control for the filtered trajectory $\vu_\alpha$   given in (\ref{Assumption-unif-bound-filtered}). 
	
	\medskip
	
	Then,  the following uniform estimate holds:
	\begin{equation}\label{Uniform-Estimate-vn-H-m+1}
	\sup_{n\in \mathbb{N}}\left(\, \sup_{t  \geq -t_n } \| \vv_n(t,\cdot)\|^2_{H^{m+1}}\right) \leq  \frac{C}{\nu^2 - \Lambda^2(\alpha, m)}\, \| \vf \|^2_{H^m}.   
	\end{equation} 
\end{Lemma}	
\begin{proof} 
From the differential inequality (\ref{Gronwall-vn-Hm}) and the control (\ref{Assumption-unif-bound-filtered}),  for any $t\geq -t_n$ we write
\begin{equation*}
\frac{d}{dt}\| \vv_n(t,\cdot)\|^2_{H^{m+1}} + \nu\| \vec{\nabla} \otimes \vv_n(t,\cdot) \|^2_{H^{m+1}} \leq  \frac{C}{\nu} \Lambda^2(\alpha,m)\, \| \vv_n(t,\cdot) \|^2_{H^{m+1}}+ C \frac{\| \vf \|^2_{H^m}}{\nu},
\end{equation*}
hence, 
\begin{equation*}
\frac{d}{dt}\| \vv_n(t,\cdot)\|^2_{H^{m+1}} \leq  \frac{C}{\nu} \Lambda^2(\alpha,m)\, \| \vv_n(t,\cdot) \|^2_{H^{m+1}} - \nu\| \vec{\nabla} \otimes \vv_n(t,\cdot) \|^2_{H^{m+1}} + C \frac{\| \vf \|^2_{H^m}}{\nu}.
\end{equation*}

Using again the Poincar\'e's inequality, we have 
\[  \| \vv_n \|_{H^m} \leq \frac{L}{2\pi}\, \| \vec{\nabla} \otimes \vv_n \|_{H^{m+1}}, \]
which implies  
\begin{equation}\label{Control-from-Poincare}
-\nu \| \vec{\nabla} \otimes \vv_n \|^2_{H^{m+1}} \leq  -\frac{L^2\, \nu}{4\pi^2}\, \| \vv_n \|^2_{H^{m+1}}:=-C\,\nu \, \| \vv_n \|^2_{H^{m+1}}.
\end{equation}

Returning to the last estimate, we find that 
\begin{equation*}
\frac{d}{dt}\| \vv_n(t,\cdot)\|^2_{H^{m+1}} \leq  C\, \left( \frac{1}{\nu}\Lambda^2(\alpha,m)- \nu\right) \|\vv_n(t,\cdot) \|^2_{H^{m+1}} + C\frac{\| \vf \|^2_{H^m}}{\nu},
\end{equation*} 

Now, observe that from assumption (\ref{Assumption-unif-bound-filtered}) we have 
\[ \frac{1}{\nu}\Lambda^2(\alpha,m)- \nu <0.  \]
 Then applying the Gr\"onwall  lemma, and the fact that $\vv_n(-t_n,\cdot)=0$, for any $t\geq -t_n$ we can write
\begin{equation*}
\| \vv_n(t,\cdot) \|^2_{H^{m+1}} \leq  \frac{C}{\ds{ - \left( \frac{1}{\nu}\Lambda^2(\alpha,m)- \nu\right)}}\times \frac{\| \vf \|^2_{H^m}}{\nu}  =  \frac{C}{\nu^2-\Lambda^2(\alpha,m)}\,  \| \vf \|^2_{H^m}.
\end{equation*}
which implies the desired uniform estimate  (\ref{Uniform-Estimate-vn-H-m+1}). 
\end{proof}	

\emph{End of the proof of Proposition \ref{Prop-Hihger-Regularity-Atrtractor}}. From the uniform estimate (\ref{Uniform-Estimate-vn-H-m+1}), by setting $t=t_0$ (recall that we have  $t_n>2-t_0$ and then $t_0>2-t_n>-t_n$), we obtain
\[  \sup_{n\in \mathbb{N}} \| \vv_n(t_0,\cdot) \|^2_{H^{m+1}} \leq \frac{C}{\nu^2-\Lambda^2(\alpha,m)}\,  \| \vf \|^2_{H^m}. \]

Then, following the same arguments as in estimate (\ref{Uniform-Control-u-t0}), it follows that 
\[ \| \vu(t_0,\cdot)\|^2_{H^{m+1}} \leq \liminf_{n\to +\infty}\| \vv_n(t_0,\cdot)\|^2_{H^{m+1}} \leq  \frac{C}{\nu^2-\Lambda^2(\alpha,m)}\,  \| \vf \|^2_{H^m}.\]

Finally, using again the fact that $\vu \in \mathcal{T}(\mathcal{A})$ and $t_0 \in \R$ were arbitrary, we obtain the desired estimate (\ref{Regularity-Atrractor-H-m+1}). Proposition \ref{Prop-Hihger-Regularity-Atrtractor} is thus proven.

\section{Theory of determining modes}\label{Sec:Determining-Modes}
\subsection{Proof of Theorem \ref{Th:Determining-Modes}} 
 Let $\vu_{0,1}, \vu_{0,2} \in L^2_{\text{s}}(\T)$ be two initial data let $\vu_1, \vu_2 \in L^\infty_{\rm loc}\big([0,+\infty[, L^2_{\text{s}}(\T) \big)\cap L^2_{\rm loc}\big( [0,+\infty[, \dot{H}^1(\T)\big)$  be the unique  associated Leray solutions of  (\ref{Alpha-model}), obtained by Theorem \ref{Th-GWP}.

\medskip

Define $\vv:= \vu_1-\vv_2$. Then $\vv$  satisfies the following homogeneous  initial value problem:
\begin{equation}\label{eq:v-eq}
\begin{cases}\vspace{2mm}
\partial_t \vv-\nu\Delta \vv+\P\Big((\vu_{1,\alpha}\cdot\vec{\nabla})\vv+\big((\vu_{1,\alpha}-\vu_{2,\alpha})\cdot\vec{\nabla}\big)\vu_2\Big)=0,
\qquad {\rm div}(\vv)=0, \\
\vv(0,\cdot)= \vu_{0,1}- \vu_{0,2}.
\end{cases}
\end{equation}
Here,  for $i=1,2$,  the filtered velocities $\vu_{i,\alpha}$  satisfy the nonlinear filter equation:
\[  -\alpha^2 \, \P \Big( \mbox{div}\left( A (\varphi\ast\vu_i) (\vec{\nabla} \otimes \vu_{i,\alpha})^T \right)\Big)+\vu_{i,\alpha} = \vu_i, \quad \mbox{div}(\vu_{i,\alpha})=0, \]
and from  Propositions \ref{Prop:Existence-Filter} and \ref{Prop:H2-Filter} we have that $\vu_{i,\alpha}\in L^\infty_{\rm loc}\big( [0,+\infty[, H^2(\T) \big)$.

\medskip

The key idea of the proof of Theorem \ref{Th:Determining-Modes} is to sharply estimate the expression $\| \vv(t,\cdot)\|^2_{L^2}$. To this end, first  recall that from estimate (\ref{Estim-Nonlinear}) we have that
\[   \P\Big((\vu_{1,\alpha}\cdot\vec{\nabla})\vv\Big), \, \P\Big(\big((\vu_{1,\alpha}-\vu_{2,\alpha})\cdot\vec{\nabla}\big)\vu_2\Big)  \in L^2_{\rm loc}\big([0,+\infty[, L^\frac{6}{5}(\T)\big). \]
On the other hand, by Sobolev embeddings we have $\vv \in L^2_{\rm loc}\big([0,+\infty[, \dot{H}^{1}(\T)\big)\subset  L^2_{\rm loc}\big([0,+\infty[, L^6(\T)\big)$. 

\medspace

Therefore, in the equation above, we can take the $L^2$-inner product with $\vv$. After integration by parts and  using the divergence-free property of $\vv$, we obtain

\begin{equation}\label{Iden-vv}
\frac{1}{2}\frac{d}{dt}\| \vv(t,\cdot)\|^2_{L^2}+\nu \| \vec{\nabla} \otimes \vv(t,\cdot)\|^2_{L^2} = - \int_{\T}  \big( \big((\vu_{1,\alpha}-\vu_{2,\alpha})\cdot\vec{\nabla}\big)\vu_2\big)\cdot \vw \, dx, 
\end{equation}
where we must estimate the term on the right-hand side. 

\medskip

Applying Cauchy-Schwarz inequalities,  H\"older inequalities and Sobolev embeddings, we write
\begin{equation*}
\begin{split}
&\, \left| \int_{\T}  \big( \big((\vu_{1,\alpha}-\vu_{2,\alpha})\cdot\vec{\nabla}\big)\vu_2\big)\cdot \vw \, dx\right|  \leq \| \big((\vu_{1,\alpha}-\vu_{2,\alpha})\cdot\vec{\nabla}\big)\vu_2 \|_{L^2}\,  \| \vw \|_{L^2}  \\
\leq &\, \| \vu_{1,\alpha}-\vu_{2,\alpha} \|_{L^\infty}\, \| \vec{\nabla} \otimes \vu_2\|_{L^2}\, \| \vv \|_{L^2}\leq C\, \| \vu_{1,\alpha}-\vu_{2,\alpha} \|_{H^2}\, \| \vec{\nabla} \otimes \vu_2\|_{L^2}\, \| \vv \|_{L^2}. 
\end{split}
\end{equation*}

Recall that the term $\| \vu_{1,\alpha}-\vu_{2,\alpha} \|_{H^2}$ was estimated in (\ref{Control-Filter-Diff-H2}). On the other hand, without loss of generality, we can assume that $\vu_1$ and $\vu_2$ belong to the absorbing set $\mathcal{B}$ defined in (\ref{Absorbing-Set}). Consequently, for the radius $\mathcal{R}>0$ defined in (\ref{R}) we have $\| \vu_1(t,\cdot)\|_{L^2}\leq \mathcal{R}$ and $\| \vu_2(t,\cdot) \|_{L^2} \leq \mathcal{R}$. Then, we obtain
\begin{equation*}
\begin{split}
\left| \int_{\T}  \big( \big((\vu_{1,\alpha}-\vu_{2,\alpha})\cdot\vec{\nabla}\big)\vu_2\big)\cdot \vw \, dx\right| \leq  &\, C\, K'_2 \, \big( \| \vu_1\|_{L^2}+\| \vu_2 \|_{L^2}+1 \big)^2\, \| \vv \|_{L^2}\,  \| \vec{\nabla} \otimes \vu_2\|_{L^2}\, \| \vv \|_{L^2} \\
\leq  & \, C\, K'_2 \, \big( 2\mathcal{R} +1 \big)^2\, \| \vv \|_{L^2}\,  \| \vec{\nabla} \otimes \vu_2\|_{L^2}\, \| \vv \|_{L^2}.
\end{split}
\end{equation*}
Additionally, from the Poincaré's inequality we have 
\[ \| \vv \|_{L^2} \leq \frac{L}{2\pi}\, \| \vec{\nabla} \otimes \vv \|_{L^2}, \]
hence, 
\[  \left| \int_{\T}  \big( \big((\vu_{1,\alpha}-\vu_{2,\alpha})\cdot\vec{\nabla}\big)\vu_2\big)\cdot \vw \, dx\right| \leq  \, \frac{C\,K'_2 L}{2\pi}  \big( 2\mathcal{R} +1 \big)^2\, \| \vec{\nabla} \otimes  \vv \|_{L^2}\,  \| \vec{\nabla} \otimes \vu_2\|_{L^2}\, \| \vv \|_{L^2}.  \]
Finally,  applying the discrete Young inequality (with $1=\frac{1}{2}+\frac{1}{2}$), we write
\begin{equation*}
\left| \int_{\T}  \big( \big((\vu_{1,\alpha}-\vu_{2,\alpha})\cdot\vec{\nabla}\big)\vu_2\big)\cdot \vw \, dx\right|  \leq \, \frac{\nu}{2}\, \| \vec{\nabla} \otimes \vv \|^2_{L^2}+ \frac{(C\, K'_2\,L)^2}{4\pi^2\, \nu}\, \big( 2\mathcal{R} +1 \big)^4\,  \| \vec{\nabla} \otimes \vu_2\|^2_{L^2}\, \| \vv \|^2_{L^2}.
\end{equation*}

By substituting this estimate on the right-hand side of identity (\ref{Iden-vv}), we find that 
\begin{equation}\label{Estim-Der-vv}
\frac{d}{dt}\| \vv(t,\cdot)\|^2_{L^2}+\nu \| \vec{\nabla} \otimes \vv(t,\cdot)\|^2_{L^2} \leq \frac{(C\, K'_2\,L)^2}{4\pi^2\, \nu}\, \big( 2\mathcal{R} +1 \big)^4\,  \| \vec{\nabla} \otimes \vu_2(t,\cdot)\|^2_{L^2}\, \| \vv(t,\cdot) \|^2_{L^2}.
\end{equation}

In the following technical lemma, we study the term $\nu \| \vec{\nabla} \otimes \vv(t,\cdot)\|^2_{L^2}$.

\begin{Lemma} Let $m\geq 1$  and $P_m$ be the orthogonal projector defined through (\ref{Stokes}) and (\ref{span-Hm}). The following lower bound holds: 
	\begin{equation}\label{Lower-Bound}
\frac{\nu\, \lambda_{m+1}}{C}	\Big( \| \vv \|^2_{L^2}- 2\| P_m(\vv)\|^2_{L^2} \Big) \leq  \nu \| \vec{\nabla} \otimes \vv \|^2_{L^2},
	\end{equation}
	for some generic constant $C>0$ independent of $m$. 
\end{Lemma}	
\begin{proof}
	We write 
	\[ \| \vv \|^2_{L^2} \leq 2\| P_m(\vv) \|^2_{L^2}+ 2\| \vv - P_m(\vv)\|^2_{L^2}. \]
	
	For the second term on the right-hand side, we have the estimate 
\begin{equation*}
2\| \vv - P_m(\vv)\|^2_{L^2} \leq \frac{C}{\lambda_{m+1}}\, \| \vec{\nabla} \otimes \vv \|_{L^2},
\end{equation*}
for some generic constant $C>0$ independent of $m$. The proof of this estimate is fairly standard, so we will omit the details. 
	
	\medskip
	
	Returning to the previous estimate, we find that 
	\[ \| \vv \|^2_{L^2} \leq 2\| P_m(\vv) \|^2_{L^2} +  \frac{C}{\lambda_{m+1}}\, \| \vec{\nabla} \otimes \vv \|^2_{L^2}, \]
	hence we deduce the desired estimate (\ref{Lower-Bound}). 
\end{proof} 	

Once we have estimate (\ref{Lower-Bound}) at our disposal, returning to (\ref{Estim-Der-vv}) and after rearranging terms, we obtain 

\begin{equation}\label{Estim-Der-vv-2}
\begin{split}
&\, \frac{d}{dt}\| \vv(t,\cdot)\|^2_{L^2} + \, \left( \frac{\nu\, \lambda_{m+1}}{C}-  \frac{(C\, K'_2\,L)^2}{4\pi^2\, \nu}\, \big( 2\mathcal{R} +1 \big)^4 \,  \| \vec{\nabla} \otimes \vu_2(t,\cdot)\|^2_{L^2}\right)\, \| \vv(t,\cdot)\|^2_{L^2} \\
\leq  &\, \frac{2\nu\, \lambda_{m+1}}{C}\, \| P_m(\vv)(t,\cdot)\|^2_{L^2}.
\end{split} 
\end{equation} 

From this estimate, we will deduce now the asymptotic determining modes property  stated in Definition \ref{def:asymptotic-determining-modes}. Roughly speaking, we will prove that for $m$ sufficiently large the assumption that 
\[ \lim_{t\to \infty}  \| P_m(\vv)(t,\cdot)\|^2_{L^2}=0, \]
implies that 
\[ \lim_{t\to \infty} \| \vv(t,\cdot)\|^2_{L^2}=0. \] 

To this end,  we will use  the  following Gr\"onwall-type estimate. For a proof of this result, please refer to  \cite{FoiasManleyRosaTemam2001} and \cite{JonesTiti1992} 

\begin{Lemma}[Uniform Gr\"onwall lemma]\label{lem:gronwall-uniform}  Let 	 $a(\cdot),b(\cdot)\in L^1_{\rm loc}\big([0,+\infty[\big)$. Additionally, let $\phi $ be nonnegative and absolutely continuous function such that,   for \emph{a.e.,}  $t>0$, satisfies: 
\begin{equation}\label{ODI}
	\phi'(t)+a(t)\phi(t)\le b(t).
\end{equation}

For some fixed $T>0$, assume that the functions $a(\cdot)$ and $b(\cdot)$ satisfy:
\begin{equation}\label{Conditions-a-Gronwall}
\liminf_{t\to\infty}\frac1T\int_t^{t+T}a(\tau)\,d\tau >0,\qquad
\limsup_{t\to\infty}\frac1T\int_t^{t+T}a^-(\tau)\,d\tau<+\infty,
\end{equation}	
and
\begin{equation}\label{Conditions-b-Gronwall}
	\liminf_{t\to\infty}\frac1T\int_t^{t+T}b^+(\tau)\,d\tau=0,
\end{equation}
respectively, where we denote  $a^-(t):=\max\big\{-a(t),0\big\}$ and $b^+(t):=\max\big\{b(t),0\big\}$. Then, we have 
\[  \lim_{t\to +\infty}\phi(t)=0. \]
\end{Lemma}

Within the framework of Lemma \ref{lem:gronwall-uniform}, we set
\begin{equation}\label{Functions}
\begin{cases}\vspace{2mm}
\phi(t):= \| \vv(t,\cdot)\|^2_{L^2}, \qquad   \ds{a_m(t):= \left( \frac{\nu\, \lambda_{m+1}}{C}-  \frac{(C\, K'_2\,L)^2}{4\pi^2\, \nu}\, \big( 2\mathcal{R} +1 \big)^4 \,  \| \vec{\nabla} \otimes \vu_2(t,\cdot)\|^2_{L^2}\right)}, \\ 
\ds{b_m(t):= \frac{2\nu\, \lambda_{m+1}}{C}\, \| P_m(\vv)(t,\cdot)\|^2_{L^2}},
\end{cases}
\end{equation}
where we must verify all the assumptions above. For clarity, in the following technical lemmas we study the functions $\phi(t)$,  $a(t)$ separately.

\begin{Lemma}\label{Lem-Tech-phi(t)} The function $t\to \phi(t)$ defined in (\ref{Functions}) is absolutely continuous and it satisfies the inequality (\ref{ODI}).
\end{Lemma}	
\begin{proof} Note that inequality (\ref{ODI}) directly follows from (\ref{Estim-Der-vv-2}). Thus, it remains to  verify the absolute continuity of $\phi(t)$.  
	
\medskip 

To do this, we begin by  verifying that $\partial_t \vv \in (L^2_t)_{\rm loc}\dot{H}^{-1}_x$. From equation (\ref{eq:v-eq}), using the divergence-free property of $\vu_{1,\alpha}$ and $\vu_{1,\alpha}-\vu_{2,\alpha}$,   we write
\[ \partial_t \vv=\nu\Delta \vv-\P \text{div}\Big( \vv \otimes \vu_{1,\alpha} + \vu_2 \otimes (\vu_{1,\alpha}-\vu_{2,\alpha}) \Big).  \]

For the first term on the right-hand side, as $\vv \in  (L^2_t)_{\rm loc}\dot{H}^{1}_x$, we directly have that 
\[ \nu\Delta \vv  \in  (L^2_t)_{\rm loc}\dot{H}^{-1}_x. \]

For the second term, from Proposition \ref{Prop:H2-Filter} and Sobolev embeddings we have $\vu_{1,\alpha}, \,  \vu_{1,\alpha}-\vu_{2,\alpha} \in (L^\infty_t)_{\rm loc} H^2_x \subset (L^\infty_t)_{\rm loc} L^\infty_x$. Therefore, as $\vv, \vu_2 \in (L^\infty_t)_{\rm loc} L^2_x$, we obtain that $\vv \otimes \vu_{1,\alpha}\, \vu_2 \otimes (\vu_{1,\alpha}-\vu_{2,\alpha})  \in (L^\infty_t)_{\rm loc} L^2_x \subset (L^2_t)_{\rm loc} L^2_x$, which implies that 
\[\P \text{div}\Big( \vv \otimes \vu_{1,\alpha} + \vu_2 \otimes (\vu_{1,\alpha}-\vu_{2,\alpha}) \Big) \in (L^2_t)_{\rm loc} \dot{H}^{-1}_x.  \]

Once we have that $\partial_t \vv \in (L^2_t)_{\rm loc} \dot{H}^{-1}_x$, using again the fact that $\vv \in (L^2_t)_{\rm loc}\dot{H}^1_x$, for any $T>0$ and any  $0<t<T$ we can  write
\[ \| \vv(t,\cdot) \|^2_{L^2}= \| \vv(0,\cdot)\|^2_{L^2}+ 2 \int_{0}^{t} \Big\langle \partial_\tau \vv(\tau, \cdot), \vv(\tau, \cdot)\Big\rangle_{\dot{H}^{-1}\times \dot{H}^1}\, d \tau.  \]
Since $\tau \mapsto \Big\langle \partial_\tau \vv(\tau, \cdot), \vv(\tau, \cdot)\Big\rangle_{\dot{H}^{-1}\times \dot{H}^1}$  belongs to $L^1([0,T])$, we conclude that $t \mapsto \phi(t):= \| \vv(t,\cdot)\|^2_{L^2}$ is  absolutely continuous. Lemma \ref{Lem-Tech-phi(t)} is now proven. 
\end{proof}	

\begin{Lemma}\label{Lem-Tech-am(t)} Let $m_0 \geq 1$ be sufficiently large such that (\ref{eq:lambda-condition-explicit}) holds. Then, for $m\geq m_0$ the function $t \mapsto a_m(t)$ defined in (\ref{Functions}) satisfies the conditions given in (\ref{Conditions-a-Gronwall}). 
\end{Lemma}	
\begin{proof} Let $T>0$ be fixed. Without loss of generality, assume that $T\geq 1$. We begin by studying the long-time behavior of the expression $\| \vec{\nabla} \otimes \vu_2(t,\cdot)\|^2_{L^2}$. Specifically, for the quantity $\ds{\eta:=\frac{4\pi^2 \nu}{L^2}>0}$ we will prove the following estimate
\begin{equation}\label{Estim-u2-H1}
\frac{1}{T} \int_{t}^{t+T} \| \vec{\nabla} \otimes \vu_2(\tau, \cdot)\|^2_{L^2}\, d \tau \leq  \frac{ e^{-\eta t}}{\nu}\,\| \vu_{0,2} \|^2_{L^2}+ \frac{\mathcal{R}^2}{\nu} + \frac{\| \vf \|^2_{\dot{H}^{-1}}}{\nu^2}. 
\end{equation}	

Indeed, since $\vu_2$ satisfies the energy inequality (\ref{Energy-Control}), replacing $0$ by $t$ and $t$ by $t+T$, we can write
\begin{equation*}
\| \vu_2(t+T, \cdot)\|^2_{L^2} + 2 \nu \int_{t}^{t+T}\| \vec{\nabla} \otimes \vu_2(\tau, \cdot)\|^2_{L^2} \leq \| \vu_2(t,\cdot)\|^2_{L^2}+ 2 \int_{t}^{t+T} \left\langle \vf, \, \vu_2(\tau,\cdot)\right\rangle_{\dot{H}^{-1}\times \dot{H}^1}\, d \tau, 
\end{equation*}
hence, as $\| \vu_2(t+T, \cdot)\|^2_{L^2} $ is positive and using the discrete Young inequalities (with $1=\frac{1}{2}+\frac{1}{2}$), we obtain 
\begin{equation*}
\begin{split}
 2 \nu \int_{t}^{t+T}\| \vec{\nabla} \otimes \vu_2(\tau, \cdot)\|^2_{L^2} \leq  &\, \| \vu_2(t,\cdot)\|^2_{L^2}+ 2 \int_{t}^{t+T} \left\langle \vf, \, \vu_2(\tau,\cdot)\right\rangle_{\dot{H}^{-1}\times \dot{H}^1}\, d \tau\\
 \leq &\, \| \vu_2(t,\cdot)\|^2_{L^2} + 2 \int_{t}^{t+T} \| \vf \|_{\dot{H}^{-1}}\, \| \vu_2(\tau,\cdot)\|_{\dot{H}^1}\, d \tau\\
 \leq &\, \| \vu_2(t,\cdot)\|^2_{L^2} +  \int_{t}^{t+T} \left( \frac{\| \vf \|_{\dot{H}^{-1}}^2}{\nu} + \nu  \| \vec{\nabla} \otimes \vu_2(\tau,\cdot)\|^2_{L^2}  \right) \, d \tau.
 \end{split}
\end{equation*}
As $\vf$ is a time-independent function, we have
\begin{equation*}
\nu \int_{t}^{t+T}\| \vec{\nabla} \otimes \vu_2(\tau, \cdot)\|^2_{L^2}\, d \tau  \leq \| \vu_2(t,\cdot)\|^2_{L^2}+ \frac{T}{\nu}\, \| \vf \|^2_{\dot{H}^{-1}},
\end{equation*}
hence
\begin{equation}\label{Estim-Tech-nabla-u2}
\frac{1}{T} \int_{t}^{t+T}\| \vec{\nabla} \otimes \vu_2(\tau, \cdot)\|^2_{L^2}\, d \tau  \leq  \frac{1}{\nu T} \,\| \vu_2(t,\cdot)\|^2_{L^2}+ \frac{1}{\nu^2}\, \| \vf \|^2_{\dot{H}^{-1}}. 
\end{equation}

Here, we still need to study the term $ \ds{\frac{1}{\nu\, T}\| \vu_2(t,\cdot)\|^2_{L^2}}$. First, note that $\vu_2$  also satisfies the energy identity
\begin{equation}\label{Energy-Control-Derivative}
\frac{d}{dt}\| \vu_2(t,\cdot)\|^2_{L^2}+2\nu \| \vu_2(t,\cdot)\|^2_{\dot{H}^1}=\left\langle \vf , \vu_2(t,\cdot)\right\rangle_{\dot{H}^{-1}\times \dot{H}^1} \leq \frac{1}{\nu}\| \vf \|^{2}_{\dot{H}^{-1}}+\nu \|\vu_2(t,\cdot)\|^2_{\dot{H}^1}, 
\end{equation}
hence we can write 
\begin{equation*}
\frac{d}{dt}\| \vu_2(t,\cdot)\|^2_{L^2}  \leq - \nu  \| \vu_2(t,\cdot)\|^2_{\dot{H}^1} + \frac{1}{\nu}\| \vf \|^{2}_{\dot{H}^{-1}}.   
\end{equation*}

By the Poincaré's inequality, we get 
\[ \| \vu_2(t,\cdot)\|_{L^2}\leq \frac{L}{2\pi}\| \vu_2(t,\cdot)\|_{\dot{H}^1},\]
and setting  $\ds{\eta:=\frac{4\pi^2 \nu}{L^2}>0}$, we obtain
\begin{equation*}
\frac{d}{dt}\| \vu_2(t,\cdot)\|^2_{L^2}  \leq - \eta \| \vu_2(t,\cdot)\|^2_{L^2} + \frac{1}{\nu}\| \vf \|^{2}_{\dot{H}^{-1}}.   
\end{equation*}
Then, applying the Gr\"onwall lemma and using the radius $\mathcal{R}$ given in (\ref{R}), we find
\begin{equation*}
\|\vu_2(t,\cdot)\|^2_{L^2} \leq  e^{-\eta t}\| \vu_{0,2} \|^2_{L^2}+ \frac{\mathcal{R}^2}{2}.
\end{equation*}
Then, as $T\geq 1$, we have
\begin{equation*}
\frac{1}{\nu T}\|\vu_2(t,\cdot)\|^2_{L^2} \leq  \frac{ e^{-\eta t}}{\nu T} \,\| \vu_{0,2} \|^2_{L^2}+\frac{\mathcal{R}^2}{2\nu T}  \leq  \frac{ e^{-\eta t}}{\nu}\,\| \vu_{0,2} \|^2_{L^2}+ \frac{\mathcal{R}^2}{\nu}. 
\end{equation*}

With this inequality, returning to (\ref{Estim-Tech-nabla-u2}) we obtain the desired estimate (\ref{Estim-u2-H1}). 

\medskip

Once we have (\ref{Estim-u2-H1}), we can now verify each condition stated in (\ref{Conditions-a-Gronwall}). For the first condition, from the definition of the function $a_m(t)$ (given in expression (\ref{Functions})) and using (\ref{Estim-u2-H1}), we write

\begin{equation*}
\begin{split}
\frac{1}{T}\int_{t}^{t+T} a(\tau)\, d \tau  =&\,  \frac{\nu \lambda_{m+1}}{C}-  \frac{(C\, K'_2\,L)^2}{4\pi^2\, \nu}\, \big( 2\mathcal{R} +1 \big)^4\,  \frac{1}{T}\int_{t}^{t+T} \| \vec{\nabla} \otimes \vu_2(\tau,\cdot)\|^2_{L^2}\, d \tau \\
\geq &\,  \frac{\nu \lambda_{m+1}}{C} - \frac{(C\, K'_2\,L)^2}{4\pi^2\, \nu}\, \big( 2\mathcal{R} +1 \big)^4\, \left( \frac{ e^{-\eta t}}{\nu T}\,\| \vu_{0,2} \|^2_{L^2}+ \frac{\mathcal{R}^2}{\nu} + \frac{\| \vf \|^2_{\dot{H}^{-1}}}{\nu^2} \right).
\end{split}
\end{equation*}
Then we have
\begin{equation*}
\liminf_{t\to\infty}  \frac{1}{T}\int_{t}^{t+T} a(\tau)\, d \tau \geq  \frac{\nu \lambda_{m+1}}{C} - \frac{(C\, K'_2\,L)^2}{4\pi^2\, \nu}\, \big( 2\mathcal{R} +1 \big)^4\, \left( \frac{\mathcal{R}^2}{\nu} + \frac{\| \vf \|^2_{\dot{H}^{-1}}}{\nu^2}\right).
\end{equation*}

At this point, note that the expression on the right-hand side is positive as long as
\begin{equation*}
\lambda_{m+1} > \frac{C(C\, K'_2\,L)^2}{4\pi^2\, \nu^2}\, \big( 2\mathcal{R} +1 \big)^4\, \left( \frac{\mathcal{R}^2}{\nu} + \frac{\| \vf \|^2_{\dot{H}^{-1}}}{\nu^2}\right)
\end{equation*}
Additionally, recalling the definition of the radius $\mathcal{R}$ given in (\ref{R}), we write
\begin{equation*}
\lambda_{m+1} > \frac{C(C\, K'_2\,L)^2}{4\pi^2\, \nu^2}\, \left( \frac{2L}{\sqrt{2}\pi \nu}\| \vf \|_{\dot{H}^{-1}} +1 \right)^4\, \left( \frac{L^2}{2 \pi^2 \nu^3}\| \vf \|^2_{\dot{H}^{-1}} + \frac{\| \vf \|^2_{\dot{H}^{-1}}}{\nu^2}\right). 
\end{equation*}

In order to simplify the notation, we omit all the constants which do not involve $\nu$ and $\| \vf \|_{\dot{H}^{-1}}$. Additionally, we will use the following simpler estimate of the constant $K'_2$, which only considers the filtering parameter $\alpha$:
\begin{equation}\label{Estim-K2-Phi}
K'_2 \simeq \Phi(\alpha),
\end{equation} 
where the function $\Phi(\alpha)$ is defined in expression (\ref{Phi}).  For the reader's convenience, we will give a proof of (\ref{Estim-K2-Phi}) in Appendix \ref{Sec:Estimate-K2-Phi}.

\medskip

In this context, since $(\lambda_m)_{m\geq 1}$ (the eigenvalues of the Stokes operator) is an increasing sequence, we can define $m_0 \geq 1$ as the smallest parameter such that $\lambda_{m+1}:= \lambda_{m_0}$ satisfies the inequality above. We thus obtain
\begin{equation*}
\lambda_{m_0} \gtrsim \frac{\Phi^2(\alpha)}{\nu^2} \left( \frac{\| \vf \|_{\dot{H}^{-1}}}{\nu} +1\right)^4\, \left( \frac{\| \vf \|^2_{\dot{H}^{-1}}}{\nu^3}+ \frac{\| \vf \|^2_{\dot{H}^{-1}}}{\nu^2}\right). 
\end{equation*}
which is the condition  (\ref{eq:lambda-condition-explicit}). 

\medskip

Consequently, assuming (\ref{eq:lambda-condition-explicit}) and for any $m\geq m_0$, we obtain that $\ds{\liminf_{t\to\infty}  \frac{1}{T}\int_{t}^{t+T} a_m(\tau)\, d \tau>0}$, thus verifying the first condition in (\ref{Conditions-a-Gronwall}).

\medskip

The second condition in (\ref{Conditions-a-Gronwall}) is directly satisfied by the estimate (\ref{Estim-u2-H1}). In fact, using again the definition of $a_m(t)$ (given in (\ref{Functions})) and (\ref{Estim-u2-H1}), we write

\begin{equation*}
\begin{split}
&\, \limsup_{t\to\infty} \frac{1}{T} \int_{t}^{t+T} a^{-}_{m}(\tau)\, d \tau \leq \, \limsup_{t\to\infty} \frac{1}{T} \int_{t}^{t+T} |a_m(\tau)|\, d \tau \\
\leq &\,  \frac{\nu\, \lambda_{m+1}}{C} + \frac{(C\, K'_2 L)^2}{4 \pi^2 \nu} (2R+1)^4 \left( \limsup_{t\to\infty} \frac{1}{T}\int_{t}^{t+T} \| \vec{\nabla}\otimes \vu_2(\tau,\cdot)\|^2_{L^2}\, d \tau \right)\\
\leq &\, \frac{\nu\, \lambda_{m+1}}{C} + \frac{(C\, K'_2 L)^2}{4 \pi^2 \nu} (2\mathcal{R}+1)^4 \left( \frac{\mathcal{R}^2}{\nu}+ \frac{\| \vf \|^2_{\dot{H}^{-1}}}{\nu^2}\right)<+\infty. 
\end{split} 
\end{equation*}
Lemma \ref{Lem-Tech-am(t)} is proven. 
\end{proof}

Finally, for the function $t \mapsto b_m(t)$ defined in (\ref{Functions}), since we assume (\ref{Determining-Modes-Pm}) and have $b_m(t)\geq 0$, the condition (\ref{Conditions-b-Gronwall}) is directly satisfied for any $m\geq 1$.

\medskip

From Lemmas \ref{Lem-Tech-am(t)}, \ref{Lem-Tech-phi(t)} and the assumption (\ref{Determining-Modes-Pm}), we have verified all the assumptions of Lemma \ref{lem:gronwall-uniform}. Thus, we have
\[ \lim_{t\to +\infty} \| \vv(t,\cdot)\|^2_{L^2}=0, \]
which concludes the proof of Theorem \ref{Th:Determining-Modes}. 

\subsection{Proof of Theorem \ref{Th:Determining-modes-on-attractor}}  Let $\vu_{\rm e,1}, \vu_{\rm e,2} \in \mathcal{T}(\mathcal{A})$. To simplify the notation, we will denote these functions as $\vu_1$ and $\vu_2$ respectively. 

\medskip

Following the same lines as in the proof of Theorem \ref{Th:Determining-Modes}, we define $\vv:= \vu_1-\vu_2$ which, for any time $t\in \R$, solves the first equation in (\ref{eq:v-eq}); hence, we finally arrive at estimate (\ref{Estim-Der-vv-2}).

\medskip

By assumption (\ref{Assumption-Determining-Modes-Attractor}) we have $\| P_m(\vv(t,\cdot))\|_{L^2}=0$ for any $t \in \R$. Therefore, from (\ref{Estim-Der-vv-2}) we can write
\begin{equation}\label{Estim-Der-vv-2-Corollary}
\frac{d}{dt}\| \vv(t,\cdot)\|^2_{L^2} +  \left( \frac{\nu\, \lambda_{m+1}}{C}-  \frac{(C\, K'_2\,L)^2}{4\pi^2\, \nu}\, \big( 2 \mathcal{R} +1 \big)^4 \,  \| \vec{\nabla} \otimes \vu_2(t,\cdot)\|^2_{L^2}\right)\, \| \vv(t,\cdot)\|^2_{L^2}  \leq 0.
\end{equation}

Here, we must study the second term on the left-hand side. Specifically, we have 
\begin{Lemma}\label{Lem-Tech-determining-modes-attractor} Let $\mathcal{K}_0>0$ be the constant given in (\ref{Regularity-Atrractor}). Let $n_0 \geq 1$ be sufficiently large such that its associated eigenvalue of the Stokes operator (\ref{Stokes}) satisfies (\ref{eq:lambda-condition-explicit-attractor}). 

\medskip
	
Then, for any $m \geq n_0$ and any $t\in \R$ we have  
\begin{equation}\label{eta0}
\begin{split}
0<\eta_0:= &\, \frac{\nu\, \lambda_{n_0}}{C} -  \frac{(C\, K'_2\,L)^2}{4\pi^2\, \nu}\, \big( 2\mathcal{R} +1 \big)^4 \, \mathcal{K}_0 \\
 \leq &\,  	\left( \frac{\nu\, \lambda_{m+1}}{C}-  \frac{(C\, K'_2\,L)^2}{4\pi^2\, \nu}\, \big( 2 \mathcal{R} +1 \big)^4 \,  \| \vec{\nabla} \otimes \vu_2(t,\cdot)\|^2_{L^2}\right). 
\end{split}
\end{equation}  
\end{Lemma}	
\begin{proof} From estimate (\ref{Regularity-Atrractor}), obtained in Theorem \ref{Th-Regularity-Atrtractor}, for any $t\in \R$ we have 
	\begin{equation*}
	\| \vec{\nabla}\otimes \vu_2(t,\cdot)\|^2_{L^2} \leq \mathcal{K}_0.
	\end{equation*}
Then, we can write
\begin{equation*}
\frac{\nu\, \lambda_{m+1}}{C}-  \frac{(C\, K'_2\,L)^2}{4\pi^2\, \nu}\, \big( 2\mathcal{R} +1 \big)^4 \,  \| \vec{\nabla} \otimes \vu_2(t,\cdot)\|^2_{L^2} \geq  \frac{\nu\, \lambda_{m+1}}{C} -  \frac{(C\, K'_2\,L)^2}{4\pi^2\, \nu}\, \big( 2\mathcal{R} +1 \big)^4 \, \mathcal{K}_0.
\end{equation*}

Here, the expression on the right-hand side is positive as long as 
\begin{equation*}
\lambda_{m+1}>\frac{C(C\, K'_2\,L)^2}{4\pi^2\, \nu^2}\, \big( 2\mathcal{R} +1 \big)^4 \, \mathcal{K}_0,
\end{equation*}
and using again the definition of the radius $\mathcal{R}$ given in (\ref{R}), we find
\begin{equation*}
\lambda_{m+1}>\frac{C(C\, K'_2\,L)^2}{4\pi^2\, \nu^2}\, \left( \frac{2L}{\sqrt{2}\pi \nu}\| \vf \|_{\dot{H}^{-1}}+1 \right)^4 \, \mathcal{K}_0,
\end{equation*}
We then define $n_0\geq 1$ as the smallest parameter such that $\lambda_{m+1}:=\lambda_{n_0}$ satisfies the inequality above. As before, to simplify the notation, we omit all the constants which do not involve $\nu$ and $\| \vf \|_{\dot{H}^{-1}}$; and again we use the estimate (\ref{Estim-K2-Phi}). We thus have
\begin{equation*}
\lambda_{n_0} \gtrsim \frac{\Phi^2(\alpha)}{\nu^2} \left(\frac{\|\vf \|_{\dot{H}^{-1}}}{\nu}+1\right)^4\,  \mathcal{K}_0,
\end{equation*} 
which is the estimate (\ref{eq:lambda-condition-explicit-attractor}).  Lemma \ref{Lem-Tech-determining-modes-attractor} is proven. 
\end{proof}	

With estimate (\ref{eta0}) at our disposal, we return to (\ref{Estim-Der-vv-2-Corollary}). Hence, for any $m\geq n_0$ we obtain 
\begin{equation*}
\frac{d}{dt} \| \vv(t,\cdot)\|^2_{L^2}+ \eta_0\, \| \vv(t,\cdot)\|^2_{L^2} \leq 0.
\end{equation*}

Finally, let $t \in \R$ be fixed and let $t_0 \in \R$ such that $t \geq t_0$. 
From the  Gr\"onwall's lemma  we have 
\begin{equation*}
\| \vv(t,\cdot)\|^2_{L^2} \leq \| \vv(t_0, \cdot)\|^2_{L^2}\,  e^{-\eta_0 (t-t_0)} \leq \| \vv \|_{L^\infty_t L^2_x}\,  e^{-\eta_0 (t-t_0)}. 
\end{equation*}
Then, letting $t_0 \to -\infty$, as $\eta_0>0$ we obtain that $\| \vu(t,\cdot)\|^2_{L^2}=0$. As the time $t_0$ is arbitrary, we conclude (\ref{Conclusion-Determining-Modes-Attractor}). Theorem \ref{Th:Determining-modes-on-attractor} is proven.

\appendix
\section{An example of indicator function $A(\cdot)$}\label{Sec:Appendix-Function-A}
Assume, for instance, that the effective region in which the filtering must be applied is characterized by the ball in $\Rt$ defined by $|x|\leq R$, for some fixed $R>0$. For any $z \in \R$, define the function
\begin{equation*}
\Theta(z):=\begin{cases}\vspace{2mm}
1, & z\leq 0, \\ \vspace{2mm}
\frac{e^{-\frac{1}{1-z}}}{e^{-\frac{1}{z}}+e^{-\frac{1}{1-z}}}, & 0 <z <1, \\
0, & z \geq 1.
\end{cases}
\end{equation*}
Next, for a parameter $0<\beta <1$, we define, for any $x \in \Rt$,
\[ A_{R,\beta}(x)= \beta + (1-\beta) \,\Theta \left( \frac{|x|-R}{R} \right). \]
We then have $A_{R,\beta}(\cdot)\in \mathcal{C}^{\infty}(\Rt)$ and $\beta \leq A_{R,\beta}(\cdot)\leq 1$. In particular, $A_{R,\beta}(x)=1$ when $|x|\leq R$ and $A_{R,\beta}(x)=\beta$ when $|x|\geq 2R$. Moreover, for any integer $m\geq 1$ and any multi-index $|{\bf a}|=m$, we have
\[  \| \partial^{{\bf a}} A_{R,\beta}(\cdot)\|_{L^\infty} \leq C_m \frac{1-\beta}{R^m}, \]
for some constant $C_m>0$ depending on $m$. Therefore, for every integer $m\geq 1$, we conclude that
\[ A_{R,\beta}(\cdot) \in \mathcal{C}^{m,1}(\Rt).\]

\section{An explicit expression in terms of $\alpha$}\label{Sec:Estimate-K2-Phi} 
Here, we provide a brief proof of estimate (\ref{Estim-K2-Phi}). Recall that the constant $K'_2$ is introduced in estimate (\ref{Control-Filter-Diff-H2}) of Proposition \ref{Prop:H2-Filter}. As mentioned above, this proposition was proved in \cite[Proposition $4.4$]{CortezJarrin2026}, where the constant $K'_2$ (denoted there by $K_3$) is explicitly given by
\begin{equation*}
K'_2 \simeq \max \left( K_0 K_2 L_1, K_0 K_1L_2, K^{3/2}_0 L_3 \right),
\end{equation*}
where the relevant auxiliary quantities satisfy
\begin{equation*}
\begin{cases}
K_0 \simeq \ds{\frac{1}{\min(\alpha^2, 1)}}, & \text{given in expression $(28)$ of \cite{CortezJarrin2026}},\\
K_1 \simeq K_2  \simeq \max\left( K^{3/2}_0\, \alpha^2, K_0  \right), & \text{given in expressions $(32)$ and $(38)$ of \cite{CortezJarrin2026}}, \\
L_1\simeq L_2 \simeq L_3 \simeq \alpha^2, & \text{given in expressions $(31), (37)$ and $(44)$ of \cite{CortezJarrin2026}}.
\end{cases}
\end{equation*}
It is worth emphasizing that these expressions also depend on the remaining parameters, namely, the parameter $\beta$ appearing in (\ref{Conditions-a-existence-1}) and certain norms of the functions $A(\cdot)$ and $\varphi$ in the nonlinear filter equation (\ref{Filter-Equation-Intro}). Nevertheless, for simplicity, we explicitly indicate only their dependence on $\alpha$.

\medskip

Based on the expression for $K_0$ above, we distinguish between the cases $0<\alpha <1$ and $\alpha >1$.

\begin{itemize}
	\item Consider first the case $0<\alpha <1$. We have $K_0 \simeq \frac{1}{\alpha^2}$ and $K^{3/2}_0 \simeq \frac{1}{\alpha^3}$. Consequently, $K_1 \simeq K_2 \simeq \max\left( \frac{1}{\alpha}, \frac{1}{\alpha^2} \right)= \frac{1}{\alpha^2}$. Therefore, using the expression for $K'_2$ given above, we obtain
	\[K'_2 \simeq \max\left( \frac{1}{\alpha^2}\times \frac{1}{\alpha^2}\times \alpha^2,  \frac{1}{\alpha^2}\times \frac{1}{\alpha^2}\times \alpha^2, \frac{1}{\alpha^3}\times \alpha^2  \right) = \frac{1}{\alpha^2}. \]
	
	\item Consider now the case $\alpha >1$. We have $K_0 \simeq 1$, and hence $K^{3/2}_0 \simeq 1$. Consequently, $K_1 \simeq K_2 \simeq \max\left( \alpha^2, 1 \right)= \alpha^2$. We thus obtain
	\[ K'_2 \simeq \max\left( 1\times \alpha^2\times \alpha^2,  1\times \alpha^2\times \alpha^2, 1\times \alpha^2 \right)= \alpha^4. \]
\end{itemize}

Combining these two cases yields estimate (\ref{Estim-K2-Phi}), with the function $\Phi(\alpha)$ defined in (\ref{Phi}).

\medskip

 {\bf Statements and Declaration}. Data sharing does not apply to this article as no datasets were generated or analyzed during the current study. In addition, the authors declare that they have no conflicts of interest,
and all of them have equally contributed to this paper. 

\medskip 

\end{document}